\documentclass[10pt,english,british]{article}
\usepackage[T1]{fontenc}
\usepackage[a4paper]{geometry}
\usepackage{color}
\usepackage{babel}
\usepackage{amsmath}
\usepackage{amssymb}
\usepackage[all]{xy}
\usepackage[unicode=true,pdfusetitle,
 bookmarks=true,bookmarksnumbered=false,bookmarksopen=false,
 breaklinks=false,pdfborder={0 0 0},backref=page,colorlinks=true]
 {hyperref}
\hypersetup{
 linkcolor=blue, urlcolor=blue, citecolor=blue}
\usepackage{breakurl}

\makeatletter

\providecommand{\tabularnewline}{\\}

\usepackage{amsthm}
 \newcommand\thmsname{Theorem}
 \newcommand\nm@thmtype{thm}
 \theoremstyle{plain}
 
 \newenvironment{namedthm}[1]{
   \renewcommand\thmsname{#1}\renewcommand\nm@thmtype{namedtheorem}
   \begin{\nm@thmtype}
}
   {\end{\nm@thmtype}
}

\usepackage{stmaryrd}

\usepackage{amsthm}

\swapnumbers
\theoremstyle{plain}

\newtheorem{thm}{Theorem}[section]
\newtheorem*{thm*}{Theorem}
\newtheorem{cor}[thm]{Corollary}
\newtheorem*{cor*}{Corollary}
\newtheorem{lem}[thm]{Lemma}
\newtheorem*{lem*}{Lemma}
\newtheorem{prop}[thm]{Proposition}
\newtheorem*{prop*}{Proposition}

\newtheorem*{conjecture*}{Conjecture}

\newtheorem*{fact*}{Conjecture}

\newtheorem*{criterion*}{Criterion}

\newtheorem*{algorithm*}{Algorithm}

\newtheorem*{ax*}{Axiom}

\newtheorem*{assumption*}{Assumption}

\newtheorem*{question*}{Question}

\theoremstyle{remark}

\newtheorem{rem}[thm]{Remark}
\newtheorem*{rem*}{Remark}
\newtheorem{rems}[thm]{Remarks}
\newtheorem*{rems*}{Remarks}

\newtheorem*{claim*}{Claim}

\newtheorem*{exercise*}{Exercise}

\newtheorem*{note*}{Note}
\newtheorem{notation}[thm]{Notation}
\newtheorem*{notation*}{Notation}

\newtheorem*{summary*}{Summary}

\newtheorem*{acknowledgement*}{Acknowledgement}

\newtheorem*{conclusion*}{Conclusion}

\theoremstyle{definition}

\newtheorem{defn}[thm]{Definition}
\newtheorem*{defn*}{Definition}

\newtheorem*{example*}{Example}

\newtheorem*{examples*}{Examples}

\newtheorem*{problem*}{Problem}

\newtheorem*{xca*}{Exercise}

\newtheorem*{xcas*}{Exercises}

\newtheorem*{condition*}{Condition}
\newtheorem{void}[thm]{}
\newtheorem{proviso}[thm]{Proviso}

\theoremstyle{plain}

\usepackage[all]{xy}        
\newcommand{\xyL}[1]{%
\xydef@\xymatrixrowsep@{#1}
} 
\newcommand{\xyC}[1]{%
\xydef@\xymatrixcolsep@{#1}
} 
\newcommand{\rest}{\upharpoonright}
\DeclareMathOperator{\sopp}{Supp}
\newcommand{\zeroinf}{[0,\infty)}
\newcommand{\zeroapinf}{(0,\infty)}

\makeatother

\begin{document}

\title{Quantifier Elimination and Rectilinearisation Theorem for Generalised
Quasianalytic Algebras }

\author{J.-P. Rolin\thanks{Partially supported by Convénio FCT/CNRS 2011 Project CNRS128447776310533.}
$\ $and T. Servi\thanks{Partially supported by FCT PEst OE/MAT/UI0209/2011, by Convénio FCT/CNRS 2011 Project CNRS128447776310533 and by FCT Project PTDC/MAT/122844/2010.}}

\date{{}}
\maketitle
\begin{abstract}
We consider for every $n\in\mathbb{N}$ an algebra $\mathcal{A}_{n}$
of germs at $0\in\mathbb{R}^{n}$ of continuous real-valued functions,
such that we can associate to every germ $f\in\mathcal{A}_{n}$ a
(divergent) series $\mathcal{T}\left(f\right)$ with nonnegative real
exponents, which can be thought of as an asymptotic expansion of $f$.
We require that the $\mathbb{R}$-algebra homomorphism $f\mapsto\mathcal{T}\left(f\right)$
be injective\emph{ }(quasianalyticity property). In this setting we
prove analogue results to Denef and van den Dries' quantifier elimination
and Hironaka's rectilinearisation theorems for subanalytic sets.
\end{abstract}

\section{Introduction}

In \cite{parusinski:preparation} the author proves a preparation
theorem for subanalytic functions (whose original statement can be
found in \cite{lr:prep,paru:lip1}), as a multivariable Puiseux Theorem,
or as a primitive version of Hironaka's Rectilinearisation Theorem.
Our aim is to extend this result to a general quasianalytic setting.

\paragraph{{}}

The results of this paper lie within the framework of o-minimal geometry.
The notion of \emph{o-minimal structure} was introduced as a model-theoretic
concept. However, in order to render it and our results accessible
to a larger audience, we give a geometric version of its definition
(see \cite{vdd:tame} and \cite{vdd:mill:omin} for an overview of
the subject).

We consider a collection $\mathcal{F}$ of functions $f:\mathbb{R}^{m}\to\mathbb{R}$
for $m\in\mathbb{N}$. A set $A\subseteq\mathbb{R}^{n}$ is \emph{definable}
in the \emph{structure} $\mathbb{R}_{\mathcal{F}}=\left(\mathbb{R},<,0,1,+,\cdot,\mathcal{F}\right)$
(which we call the \emph{expansion of the real field by }$\mathcal{F}$)
if $A$ belongs to the smallest collection $\mathcal{S}=\left(\mathcal{S}_{n}\right)_{n\in\mathbb{N}}$
of subsets of $\mathbb{R}^{n}$ ($n\in\mathbb{N}$) such that:
\begin{enumerate}
\item $\mathcal{S}_{n}$ contains all semi-algebraic subsets of $\mathbb{R}^{n}$,
\item $\mathcal{S}$ contains the graphs of all functions in $\mathcal{F}$,
\item $\mathcal{S}$ is closed under the boolean set operations and projections.
\end{enumerate}
A map $f:A\subseteq\mathbb{R}^{m}\to\mathbb{R}^{n}$ is \emph{definable}
in $\mathbb{R}_{\mathcal{F}}$ if its graph is a definable subset
of $\mathbb{R}^{m+n}$ (this implies that $A$ itself is definable).

The structure $\mathbb{R}_{\mathcal{F}}$ is said to be:
\begin{enumerate}
\item \emph{model-complete} if in order to generate the whole family of
definable sets one can dispense with taking complements; 
\item \emph{o-minimal }if all definable sets have finitely many connected
components;
\item \emph{polynomially bounded} if the germs at infinity of definable
unary functions have at most polynomial growth. 
\end{enumerate}
The two following classical example of polynomially bounded model-complete
o-minimal expansion of the real field show how o-minimal geometry
can be seen as a generalisation of real algebraic and analytic geometry.
\begin{itemize}
\item $\bar{\mathbb{R}}$, where $\mathcal{F}=\emptyset$, i.e. the structure
whose definable sets are the semi-algebraic sets (the proof of model-completeness
and o-minimality follows from the Tarski-Seidenberg \emph{elimination}
principle \cite{tarski,seidenberg}).
\item $\mathbb{R}_{\mathrm{an}}$, where $\mathcal{F}$ is the set of functions
$f:\mathbb{R}^{n}\to\mathbb{R}$ ($n\in\mathbb{N}$) whose restriction
to the unit cube $\left[-1,1\right]^{n}$ is real analytic and which
are identically zero outside the unit cube (the proof is a consequence
of Gabrielov's \emph{Theorem of the Complement}, which states that\emph{
}the complement of a subanalytic subset of a real analytic manifold
is again subanalytic \cite{gabriel:proj,gabriel:expli,bm_semi_subanalytic}).
\end{itemize}
Besides the finiteness property mentioned in the definition, sets
and functions definable in an o-minimal structure share many good
geometric properties with semi-algebraic and subanalytic sets. In
particular, for all $k\in\mathbb{N}$, a definable set admits a finite
stratification with $\mathcal{C}^{k}$ definable manifolds (see for
example \cite{vdd:mill:omin}).

\paragraph{{}}

In the last two decades the list of examples has grown considerably,
including several polynomially bounded structures for which o-minimality
is a consequence of model-completeness. Among these examples we find:
\begin{itemize}
\item $\mathbb{R}_{\mathrm{an}^{*}}$, where $\mathcal{F}$ is the collection
of all so-called \emph{convergent generalised power series, }namely
all functions $f:\mathbb{R}^{n}\to\mathbb{R}$ whose restriction to
the unit cube is given by a convergent series of monomials with positive
real exponents (with a well-order condition on the support) and which
are identically zero outside the unit cube (see \cite{vdd:speiss:gen}).
\item $\mathbb{R}_{\mathcal{G}}$, where $\mathcal{F}$ is a collection
of Gevrey functions in several variables (see \cite{vdd:speiss:multisum}
for the definitions and proofs).
\item $\mathbb{R}_{\mathcal{C}\left(M\right)}$, where $\mathcal{F}$ is
a collection of $\mathcal{C}^{\infty}$ functions, restricted to the
unit cube, whose derivatives are bounded by a strictly increasing
sequence of positive constants $M=\left(M_{n}\right)_{n\in\mathbb{N}}$
satisfying the \emph{Denjoy-Carleman quasianalyticity }condition (see
\cite{rsw} for the definitions and proofs).
\item $\mathbb{R}_{\text{an},H}$, where $\mathcal{F}$ is a collection
of functions containing all real analytic functions restricted to
the unit cube and a solution $H$ of a first order analytic differential
equation which is singular at the origin (see \cite{rss}).
\item $\mathbb{R}_{\mathcal{Q}}$, in which certain Dulac transition maps
are definable (see \cite{ilyashenko:centennial} for a complete survey
on Dulac's problem and \cite{krs} for the proof of model-completeness
and o-minimality). In this example $\mathcal{F}$ is a collection
of functions, restricted to the unit cube, whose germ at zero admits
an asymptotic expansion (with positive real exponents) which is, in
general, divergent (as opposed to the case of $\mathbb{R}_{\text{an}^{*}}$).
\end{itemize}
The Tarski-Seidenberg elimination principle, which implies model-completeness,
is a \emph{quantifier elimination} result, which can be resumed by
saying that the projection of a semi-algebraic set is again semi-algebraic.
It is well known that the analogue result does not hold for $\mathbb{R}_{\text{an}}$.
However, Denef and van den Dries proved in \cite{vdd:d} that every
relatively compact subanalytic set can be described by a system of
equations and inequalities satisfied by some compositions of restricted
analytic functions and quotients. In other words, the structure $\mathbb{R}_{\text{an}}$
admits quantifier elimination in the language of restricted analytic
functions expanded by the function $D:\left(x,y\right)\mapsto x/y$
for $|y|\geq|x|$ and zero otherwise.

\paragraph{{}}

Our main goal is to prove a quantifier elimination result in the spirit
of \cite{vdd:d} for all the structures in the above examples. In
order to do this, we first define a common framework of which each
of the above examples appears as a special case. Secondly, we isolate
the common properties of these structures which are relevant to the
proof of their o-minimality. Finally, we proceed to develop a new
strategy for the proof of quantifier elimination. The originality
of this work lies especially in this last part, since the main tool
in Denef and van den Dries' proof, namely the Weierstrass Preparation
Theorem, is not available to us in this setting, as we will explain
later.

\paragraph{{}}

The proofs of model-completeness and o-minimality for the above examples
are similar and all inspired by \cite{gabriel:proj,bm_semi_subanalytic}.
The common strategy is to parametrise a definable set by maps whose
components are in $\mathcal{F}$. To obtain this, the key property
is the \emph{quasianalyticity} (explained below) of the algebras generated
by the germs of functions in $\mathcal{F}$. However, there are significant
differences between these examples. In some of them (see \cite{vdd:speiss:gen,vdd:speiss:multisum,krs}),
we can use the\emph{ }Weierstrass Preparation Theorem\emph{ }with
respect to a certain type of variables. In this case the proof of
o-minimality is very close to Gabrielov's proof for $\mathbb{R}_{\text{an}}$
in \cite{gabriel:proj}. In the other cases, one uses instead some
form of \emph{resolution of singularities}, which allows to get to
the required parametrisation result in a more indirect way.

\paragraph{{}}

Our first goal is to unify all these different proofs, so that all
the examples above appear as particular cases of a general o-minimality
statement. We consider, for all $n\in\mathbb{N}$, an algebra $\mathcal{A}_{n}$
of continuous functions such that to each germ at $0\in\mathbb{R}^{n}$
of $f\in\mathcal{A}_{n}$ we can associate a (divergent) series $\mathcal{T}\left(f\right)$
with nonnegative real exponents. This series can be thought of as
an \emph{asymptotic expansion} of $f$. The algebra of all germs at
$0$ of functions in $\mathcal{A}_{n}$ is \emph{quasianalytic }if
the $\mathbb{R}$-algebra homomorphism $\mathcal{T}:f\mapsto\mathcal{T}\left(f\right)$
is injective. The link between quasianalyticity and o-minimality appears
already in \cite{vdd:o_minimal_real_analytic,vdd:speiss:multisum}.
The quasianalyticity property can be obtained de facto (as in \cite{vdd:speiss:gen},
where the functions under consideration are \emph{equal} to the sum
of their convergent expansion), or it has to be proved using analysis
techniques (resummation methods in \cite{vdd:speiss:multisum} and
\cite{rss}, Denjoy-Carleman's theorem in \cite{rsw}, Ilyashenko's
method for Dulac's problem in \cite{krs}).

\paragraph{{}}

It is not our purpose here to prove quasianalyticity. We assume rather
to have been given a collection of quasianalytic algebras. We show,
then, that the structure $\mathbb{R}_{\mathcal{A}}$ generated by
the algebras $\mathcal{A}_{n}$ is model-complete, o-minimal and polynomially
bounded. Given the level of generality we aim to keep, we cannot make
use of the Weierstrass Preparation Theorem. Hence, we show how to
prove the parametrisation result for definable sets by using an appropriate
\emph{blow-up} procedure. This latter takes inspiration from the methods
in \cite{bm_semi_subanalytic}, adapted to series with real exponents
in \cite{vdd:speiss:gen}.

\paragraph{{}}

Once we have proved o-minimality for the structures which fall into
this general framework, we can proceed towards our quantifier elimination
result. In \cite{vdd:d}, by a clever use of the Weierstrass Preparation
Theorem and of the Tarski-Seidenberg elimination principle, Denef
and van den Dries prove their result without needing to solve explicitly
any system of analytic equations.

In our framework, recall that we have already shown that a bounded
$\mathbb{R}_{\mathcal{A}}$-definable set can be parametrised by maps
whose components are in the algebras $\mathcal{A}_{n}$. Our aim is
to eliminate the parameters. Since we cannot use the Weierstrass Preparation
Theorem and hence reduce to the polynomial situation, our strategy
is, instead, to solve explicitly the parametrising equations with
respect to the parameters. In order to do that, we use o-minimality
of the structure $\mathbb{R}_{\mathcal{A}}$, established in the first
part of the paper, and apply an \emph{o-minimal Preparation Theorem}
for functions definable in a polynomially bounded o-minimal structure
\cite{vdd:speiss:preparation_theorems}. This latter result, whose
proof uses valuation and model-theoretic methods, allows to find,
piecewise, a ``principal part''\emph{ }of a definable function.
This is the starting point for a Newton-Puiseux solving method for
quasianalytic equations.

\paragraph{{}}

Finally, Denef and van den Dries deduce Hironaka's \emph{Rectilinearisation
Theorem} \cite{hironaka_real_analytic} as a corollary of their main
result. This result states that every subanalytic set $A\subseteq\mathbb{R}^{n}$
can be transformed into a finite union of quadrants of dimension at
most $\text{dim}\left(A\right)$, via a finite sequence of blow-ups
of the ambient space $\mathbb{R}^{n}$. In the same spirit, we prove
a Rectilinearisation Theorem for bounded $\mathbb{R}_{\mathcal{A}}$-definable
sets.

\paragraph{{}}

The plan of the paper is the following. In Section \ref{sec:Setting-and-main}
we introduce formally the setting we are working in and, in Subsection
\ref{sub:The-theorems}, we give the two main statements we prove,
namely o-minimality of $\mathbb{R}_{\mathcal{A}}$ (Theorem A) and
quantifier elimination (Theorem B). Section \ref{sec:Monomialisation of series}
is dedicated to a monomialisation, or desingularisation, algorithm.
In Section \ref{sec:Parametrisation-of--subanalytic-1} we prove the
parametrisation result mentioned above. The proof of Theorem A is
completed in Subsection \ref{subsec:Proof-of-o-minimality}, following
a traditional approach à la Gabrielov, thanks to Proposition \ref{prop: trivial manifolds for subanal sets}.
Section \ref{sec:Vertical-monomialisation} is dedicated to the proof
of Theorem B. The key result is a monomialisation theorem for $\mathbb{R}_{\mathcal{A}}$-definable
functions (Theorem \ref{thm: monomialis of def functions}), from
which we deduce the Rectilinearisation Theorem \ref{cor: rectilinearisation}
and Theorem B. The proof of Theorem \ref{thm: monomialis of def functions}
is obtained by a significant modification of the monomialisation process
described in Subsection \ref{sub:Monomialisation-of-generalised}.
We develop a \emph{vertical monomialisation algorithm} which allows
to solve explicitly a system of quasianalytic equations (Theorem \ref{thm: ABC},
part B), inverse a parametrisation (Theorem \ref{thm: ABC}, part
C) and finally monomialise definable functions (Theorem \ref{thm: ABC},
part A). The first step for solving explicitly quasianalytic equations
is to ``weakly monomialise'' the solutions. This is done in Lemma
\ref{lem: weak monomialisation}, where we use the o-minimal Preparation
Theorem in \cite{vdd:speiss:preparation_theorems} mentioned above.

\section{Setting and main results\label{sec:Setting-and-main}}

\subsection{Generalised power series\label{sub:gen power series}}
\begin{defn}
\label{def:good set}Let $m\in\mathbb{N}$. A set $S\subset[0,\infty)^{m}$
is called \emph{good} if $S$ is contained in a cartesian product
$S_{1}\times\ldots\times S_{m}$ of well ordered subsets of $\zeroinf$.
If $S$ is a good set, define $S_{\mathrm{min}}$ as the set of minimal
elements of $S$. By \cite[Lemma 4.2]{vdd:speiss:gen}, $S_{\mathrm{min}}$
is finite. For $\alpha=\left(\alpha_{1},\ldots,\alpha_{m}\right),\beta=\left(\beta_{1},\ldots,\beta_{m}\right)\in S$
we write $\alpha\leq\beta$ if $\alpha_{i}\leq\beta_{i}$ for all
$i=1,\ldots,m$. So if $S$ is good, $\forall\alpha\in S\ \exists\beta\in S_{\mathrm{min}}$
such that $\alpha\geq\beta$.

Denote by $\Sigma\left(S\right)$ the set of all finite sums (done
component-wise) of elements of $S$. By \cite[Lemma 4.3]{vdd:speiss:gen},
if $S$ is good then $\Sigma\left(S\right)$ is also good.
\end{defn}
We recall the definition of generalised formal power series, originally
due to \cite{vdd:speiss:gen}.
\begin{defn}
\label{def:power series} Let $\mathbb{A}$ be a commutative ring,
$m\in\mathbb{N}$ and $X=(X_{1},\ldots,X_{m})$ be a tuple of variables.
We consider formal series 
\[
F(X)=\sum_{\alpha}c_{\alpha}X^{\alpha},
\]
 where $\alpha=(\alpha_{1},\ldots,\alpha_{m})\in\zeroinf^{m},\ c_{\alpha}\in\mathbb{A}$
and $X^{\alpha}$ denotes the formal monomial $X_{1}^{\alpha_{1}}\cdot\ldots\cdot X_{m}^{\alpha_{m}}$,
and the set 
\[
\sopp(F):=\{\alpha\in\zeroinf^{m}:\ c_{\alpha}\not=0\}\ (\mathrm{the\ support\ of\ }F)
\]
 is a good set. These series are added the usual way and form a ring
denoted by $\mathbb{A}\llbracket X^{*}\rrbracket$. 
\end{defn}
The ring of usual power series $\mathbb{A}\llbracket X\rrbracket$
can be seen as the subring of $\mathbb{A}\llbracket X^{*}\rrbracket$
consisting of the series $F$ with $\sopp(F)\subset\mathbb{N}^{m}$.

\begin{defn}
\label{def:total support}Let $\mathcal{F}\subset\mathbb{A}\left\llbracket X^{*}\right\rrbracket $
be a (possibly infinite) family of series such that the \emph{total
support of $\mathcal{F}$} 
\[
\mathrm{Supp}\left(\mathcal{F}\right):=\bigcup_{F\in\mathcal{F}}\mathrm{Supp}\left(F\right)
\]
 is a good set. Then $\mathrm{Supp}\left(\mathcal{F}\right)_{\mathrm{min}}$
is finite and we define the \emph{set of minimal monomials of $\mathcal{F}$}
\[
\mathcal{F}_{\mathrm{min}}:=\left\{ X^{\alpha}:\ \alpha\in\mathrm{Supp}\left(\mathcal{F}\right)_{\mathrm{min}}\right\} .
\]

\end{defn}

\begin{defn}
\label{def:mixed series} We fix $m,n\in\mathbb{N}$. Let $(X,Y)=(X_{1},\ldots,X_{m},Y_{1},\ldots,Y_{n})$.
We define $\mathbb{A}\llbracket X^{*},Y\rrbracket$ as the subring
of $\mathbb{A}\llbracket(X,Y)^{*}\rrbracket$ consisting of those
series $F$ such that $\sopp(F)\subset\zeroinf^{m}\times\mathbb{N}^{n}$.
Moreover, for $F\in\mathbb{A}\left\llbracket X^{*},Y\right\rrbracket $,
we define 
\[
\mathrm{Supp}_{X}\left(F\right):=\left\{ \alpha\in[0,\infty)^{m}:\ \exists N\in\mathbb{N}^{n}\mathrm{\ s.t.\ }\left(\alpha,N\right)\in\mathrm{Supp}\left(F\right)\right\} .
\]

Notice that $\mathrm{Supp}\left(F\right)$ is good if and only if
$\mathrm{Supp}_{X}\left(F\right)$ is good.
\end{defn}

\begin{notation}
\label{not: derivatives}Let $F\in\mathbb{R}\left\llbracket X^{*},Y\right\rrbracket $.
For $i\in\left\{ 1,\ldots,m\right\} $ and $j\in\left\{ 1,\ldots,n\right\} $,
let 
\[
\hat{X}=\left(X_{1},\ldots,X_{i-1},X_{i+1},\ldots,X_{m}\right)\text{ and }\hat{Y}=\left(Y_{1},\ldots,Y_{j-1},Y_{j+1},\ldots,Y_{n}\right).
\]
We can write $F\left(X,Y\right)=\sum_{\alpha\in S}G_{\alpha}\left(\hat{X},Y\right)X_{i}^{\alpha}=\sum_{k\in\mathbb{N}}H_{k}\left(X,\hat{Y}\right)Y_{j}^{k}$,
where $S$ is the $i^{th}$ projection of $\mathrm{Supp}\left(F\right)$
and $G_{\alpha}\in\mathbb{R}\left\llbracket \hat{X},Y\right\rrbracket ,\ H_{k}\in\mathbb{R}\left\llbracket X^{*},\hat{Y}\right\rrbracket $.
We denote by $\partial_{i}F$ the series $\sum_{\alpha\in S}\alpha G_{\alpha}X_{i}^{\alpha}$
and by $\frac{\partial F}{\partial Y_{j}}$ the series $\sum_{k\geq1}kH_{k}Y_{j}^{k-1}$.
\end{notation}

\begin{notation}
Let $\lambda,\alpha>0$. We denote by $\left(Y+\lambda\right)^{\alpha}$
the power series $\lambda^{\alpha}\sum_{i\in\mathbb{N}}\binom{\alpha}{i}\left(\frac{Y}{\lambda}\right)^{i}\in\mathbb{R}\left\llbracket Y\right\rrbracket $.
\end{notation}

\subsection{Quasianalytic algebras\label{sub:Quasi-analytic-algebras}}

In this subsection we define the basic object of our interest, namely
we fix a family $\mathcal{A}$ of real functions satisfying the properties
in \ref{vuoto:conditions on functions} and \ref{def: A-analytic}.
Moreover, we require the collection of all germs at zero of the functions
in $\mathcal{A}$ to satisfy the properties in \ref{def: quasi-analyticity}
and \ref{emp:properties of the morph}.
\begin{notation}
Let $m,n\in\mathbb{N}$. A \emph{polyradius }of type $\left(m,n\right)$
in $\mathbb{R}^{m+n}$ is a tuple of the form $r=(s,t)=(s_{1},\ldots,s_{m},t_{1},\ldots,t_{n})\in[0,\infty)^{m+n}$.
If $r,r'$ are polyradii in $\mathbb{R}^{m+n}$, we write $r'\leq r$
if $s_{i}'\leq s_{i}$ for all $i=1,\ldots,m$ and $t_{j}'\leq t_{j}$
for all $j=1,\ldots,n$. If $r'$ is of type $\left(m,n-1\right)$
and $r$ is of type $\left(m,n\right)$, we write $r'\leq r$ if $\left(r',0\right)\leq r$.
We also define: 
\[
I_{m,n,r}:=(0,s_{1})\times\ldots\times(0,s_{m})\times(-t_{1},t_{1})\times\ldots\times(-t_{n},t_{n}),
\]
 
\[
\hat{I}_{m,n,r}:=[0,s_{1})\times\ldots\times[0,s_{m})\times(-t_{1},t_{1})\times\ldots\times(-t_{n},t_{n}).
\]
We also denote by $\hat{I}_{m,n,\infty}$ the set $[0,+\infty)^{m}\times\mathbb{R}^{n}$.\end{notation}
\begin{void}
\label{vuoto:conditions on functions}For every $m,n\in\mathbb{N}$
and $r\in\zeroapinf^{m+n}$, we let $\mathcal{A}_{m,n,r}$ be an algebra
of real functions, which are $\mathcal{C}^{0}$ on $\hat{I}_{m,n,r}$
and $\mathcal{C}^{1}$ on $I_{m,n,r}$. We require that the algebras
$\mathcal{A}_{m,n,r}$ satisfy the following list of conditions. Let
$x=\left(x_{1,}\ldots,x_{m}\right)$ and $y=\left(y_{1},\ldots,y_{n}\right)$.
\begin{itemize}
\item The coordinate functions of $\mathbb{R}^{m+n}$ are in $\mathcal{A}_{m,n,r}$. 
\item If $r'\leq r$ and $f\in\mathcal{A}_{m,n,r}$, then $f\rest\hat{I}_{m,n,r'}\in\mathcal{A}_{m,n,r'}$.
\item If $f\in\mathcal{A}_{m,n,r}$ then there exists $r'>r$ and $g\in\mathcal{A}_{m,n,r'}$
such that $g\rest\hat{I}_{m,n,r}=f$.
\item If $f\in\mathcal{A}_{m,n,r}$, $s\in\zeroapinf$ and $r'=\left(s_{1},\ldots,s_{m},s,t_{1},\ldots,t_{n}\right)$
then the function
\[
\xyC{0mm}\xyL{0mm}\xymatrix{F\colon & \hat{I}_{m+1,n,r'}\ar[rrrr] & \  & \  & \  & \mathbb{R}\\
 & \left(x_{1},\ldots,x_{m},z,y\right)\ar@{|->}[rrrr] &  &  &  & f\left(x,y\right)
}
\]
is in $\mathcal{A}_{m+1,n,r'}$.
\item $\mathcal{A}_{m,n,r}\subset\mathcal{A}_{m+n,0,r}$, in the sense that
we identify $f\left(x_{1},\ldots,x_{m},y_{1},\ldots,y_{n}\right)\in\mathcal{A}_{m,n,r}$
with $f\left(x_{1},\ldots,x_{m+n}\right)\in\mathcal{A}_{m+n,0,r}$. 
\item Let $\sigma\in\Sigma_{m}$ be a permutation and let $\sigma\left(x\right)=\left(x_{\sigma\left(1\right)},\ldots,x_{\sigma\left(m\right)}\right)$.
If $f\in\mathcal{A}_{m,n,r}$ then there exists $f_{\sigma}\in\mathcal{A}_{m,n,r}$
such that $f_{\sigma}\left(x,y\right)=f\left(\sigma\left(x\right),y\right)$.
\item If $f\in\mathcal{A}_{m,n,r}$ then there exists $g\in\mathcal{A}_{m-1,n,r}$
such that 
\[
g\left(x_{1},\ldots,x_{m-1},y\right)=f\left(x_{1},\ldots,x_{m-1},0,y\right).
\]

\item If $f\in\mathcal{A}_{m,n,r}$ then $f\left(x_{1}/s_{1},\ldots,x_{m}/s_{m},y_{1}/t_{1},\ldots,y_{n}/t_{n}\right)\in\mathcal{A}_{m,n,1}.$
\end{itemize}
\end{void}
\begin{notation}
\label{not: germs}We denote by $\mathcal{A}_{m,n}$ the algebra of
germs at the origin of the elements of $\mathcal{A}_{m,n,r}$, for
$r$ a polyradius in $\zeroapinf^{m+n}$. When $n=0$ we write $\mathcal{A}_{m}$
for $\mathcal{A}_{m,0}$. We will often denote the germ and a representative
by the same letter. 
\end{notation}
We require all the functions in $\mathcal{A}$ to satisfy the property
described in the next definition, which mimics the property of real
analytic germs of being analytic on a whole neighbourhood of the origin.
\begin{defn}
\label{def: A-analytic}Let $f\left(x,y\right)\in\mathcal{A}_{m,n,r}$,
for some $m,n\in\mathbb{N}$ and $r\in\left(0,\infty\right)^{m+n}$.
For $a\in\hat{I}_{m,n,r}$, put $m'=|\left\{ i:\ 1\leq i\leq m,\ a_{i}=0\right\} |$
and choose a permutation $\sigma$ of $\left\{ 1,\ldots,m\right\} $
such that $\sigma\left(\left\{ i:\ 1\leq i\leq m,\ a_{i}=0\right\} \right)=\left\{ 1,\ldots,m'\right\} $.
We say that $f$ is $\mathcal{A}$-analytic if for all $a\in\hat{I}_{m,n,r}$
there exists a germ $g_{a}\in\mathcal{A}_{m',m+n-m'}$ such that the
germ at $a$ of $f$ is equal to the germ at $a$ of $ $
\[
\left(x,y\right)\mapsto g_{a}\left(x_{\sigma\left(1\right)}-a_{\sigma\left(1\right)},\ldots,x_{\sigma\left(m\right)}-a_{\sigma\left(m\right)},y_{1}-a_{m+1},\ldots,y_{n}-a_{m+n}\right).
\]

\end{defn}

The next definition establishes a relevant analogy with the behaviour
of analytic germs, namely the fact that a germ in the collection $\{\mathcal{A}_{m,n}:\ m,n\in\mathbb{N}\}$
is uniquely determined by its ``generalised Taylor expansion''. 
\begin{defn}
\label{def: quasi-analyticity}We say that $\left\{ \mathcal{A}_{m,n}:\ m,n\in\mathbb{N}\right\} $
is a collection of \emph{quasianalytic algebras} if, for all $m,n\in\mathbb{N}$,
there exists an \textbf{injective} $\mathbb{R}$-algebra morphism
\[
\mathcal{T}_{m,n}:\mathcal{A}_{m,n}\to\mathbb{R}\left\llbracket X^{*},Y\right\rrbracket ,
\]

where $X=\left(X_{1},\ldots,X_{m}\right),\ Y=\left(Y_{1},\ldots,Y_{n}\right)$.
Moreover, for all $m'\geq m,\ n'\geq n$ we require that the morphism
$\mathcal{T}_{m',n'}$ extend $\mathcal{T}_{m,n}$, hence, from now
on we will write $\mathcal{T}$ for $\mathcal{T}_{m,n}$.
\end{defn}
We require $\{\mathcal{A}_{m,n}:\ m,n\in\mathbb{N}\}$ to be a collection
of quasianalytic algebras which, together with the morphism $\mathcal{T}$,
satisfies the list of closure and compatibility properties in \ref{emp:properties of the morph}
below. First, we need some definitions.
\begin{defn}
\label{def: admissible exponents}A number $\alpha\in\zeroinf$ is
an \emph{admissible exponent} if there are $m,n\in\mathbb{N},$ $f\in\mathcal{A}_{m,n},\ \beta\in\sopp\left(\mathcal{T}\left(f\right)\right)\subset\mathbb{R}^{m}\times\mathbb{N}^{n}$
such that $\alpha$ is a component of $\beta$. We denote by $\mathbb{A}$
semi-ring generated by all admissible exponents and by $\mathbb{K}$
the field of fractions of $\mathbb{A}$.
\end{defn}

\begin{defn}
\label{def: blow-up charts}Let $m,n\in\mathbb{N},\ \left(x,y\right)=\left(x_{1},\ldots,x_{m},y_{1},\ldots,y_{n}\right)$.
For $m',n'\in\mathbb{N}$ with $m'+n'=m+n$, we set $\left(x',y'\right)=\left(x_{1}',\ldots,x_{m'}',y_{1}',\ldots,y_{n'}'\right)$.
Let $r,r'$ be polyradii in $\mathbb{R}^{m+n}$. A \emph{blow-up chart
}is a map

\[
\xyC{0mm}\xyL{0mm}\xymatrix{\pi\colon & \hat{I}_{m',n',r'}\ar[rrrr] & \  & \  & \  & \hat{I}_{m,n,r}\\
 & \left(x',y'\right)\ar@{|->}[rrrr] &  &  &  & \left(x,y\right)
}
\]
of either of the following forms:
\begin{itemize}
\item For $1\leq j<i\leq m$ and $\lambda\in(0,\infty)$, let $m'=m-1$
and $n'=n+1$ and define
\begin{align*}
\pi_{i,j}^{\lambda} & \left(x',y'\right)=\left(x,y\right),\ \ \ \mathrm{where}\ \begin{cases}
x_{k}=x_{k}' & 1\leq k<i\\
x_{i}=x_{j}'\left(\lambda+y_{1}'\right)\\
x_{k}=x_{k-1}' & i<k\leq m\\
y_{k}=y_{k+1}' & 1\leq k\leq n
\end{cases}.
\end{align*}

\item For $1\leq j,i\leq m$, let $m'=m$ and $n'=n$ and define
\begin{align*}
\pi_{i,j}^{0} & \left(x',y'\right)=\left(x,y\right),\ \ \ \mathrm{where}\ \begin{cases}
x_{k}=x_{k}' & 1\leq k\leq m,\ k\not=i\\
x_{i}=x_{j}'x_{i}'\\
y_{k}=y_{k}' & 1\leq k\leq n
\end{cases}\\
\mathrm{and}\  & \pi_{i,j}^{\infty}=\pi_{j,i}^{0};
\end{align*}

\item For $1\leq i\leq n,\ 1\leq j\leq m$ and $\lambda\in\mathbb{R}$,
let $m'=m$ and $n'=n$ and define
\begin{align*}
\pi_{m+i,j}^{\lambda}\left(x',y'\right) & =\left(x,y\right),\ \ \ \mathrm{where}\ \begin{cases}
x_{k}=x_{k}' & 1\leq k\leq m\\
y_{i}=x_{j}'\left(\lambda+y_{i}'\right)\\
y_{k}=y_{k}' & 1\leq k\leq n,\ k\not=i
\end{cases}.
\end{align*}

\item For$1\leq i\leq n,\ 1\leq j\leq m$, let $m'=m+1$ and $n'=n-1$ and
define
\[
\pi_{m+i,j}^{\pm\infty}\left(x',y'\right)=\left(x,y\right),\ \ \ \mathrm{where}\ \begin{cases}
x_{k}=x_{k}' & 1\leq k\leq m,\ k\not=j\\
x_{j}=x_{m+1}'x_{j}'\\
y_{k}=y_{k}' & 1\leq k<i\\
y_{i}=\pm x_{m+1}'\\
y_{k}=y_{k-1}' & i<k\leq n
\end{cases}.
\]

\item For $1\leq i,j\leq n$ and $\lambda\in\mathbb{R}$, let $m'=m$ and
$n'=n$ and define
\begin{align*}
\pi_{m+i,m+j}^{\lambda}\left(x',y'\right) & =\left(x,y\right),\ \ \ \mathrm{where}\ \begin{cases}
x_{k}=x_{k}' & 1\leq k\leq m\\
y_{i}=y_{j}'\left(\lambda+y_{i}'\right)\\
y_{k}=y_{k}' & 1\leq k\leq n,\ k\not=i
\end{cases}\\
\mathrm{and}\  & \pi_{m+i,m+j}^{\infty}=\pi_{m+j,m+i}^{0}.
\end{align*}

\end{itemize}
\begin{flushleft}
We also define the following collections:
\par\end{flushleft}

\begin{tabular}{l}
$\mathrm{for}\ 1\leq i,j\leq m,\ \ \pi_{i,j}:=\left\{ \pi_{i,j}^{\lambda}:\ \lambda\in\left[0,\infty\right]\right\} ,$\tabularnewline
$\mathrm{for}\ 1\leq i\leq n,\ 1\leq j\leq m,\ \ \pi_{m+i,j}:=\left\{ \pi_{m+i,j}^{\lambda}:\ \lambda\in\mathbb{R}\cup\left\{ \pm\infty\right\} \right\} ,$\tabularnewline
$\mathrm{for}\ 1\leq i,j\leq n,\ \ \pi_{m+i,m+j}:=\left\{ \pi_{i,j}^{\lambda}:\ \lambda\in\mathbb{R}\cup\left\{ \infty\right\} \right\} .$\tabularnewline
\end{tabular}\end{defn}
\begin{void}
\label{emp:properties of the morph}We require that the family of
algebras of germs $\left\{ \mathcal{A}_{m,n}:\ m,n\in\mathbb{N}\right\} $
satisfy the following properties:
\begin{enumerate}
\item \emph{Monomials and ramifications.} If $\mathbb{A}\not=$$\mathbb{N}$,
then for every $\alpha\in\mathbb{K}^{\geq0}$, the germ $x_{1}\mapsto x_{1}^{\alpha}$
is in $\mathcal{A}_{1}$ and $\mathcal{T}\left(x_{1}^{\alpha}\right)=X_{1}^{\alpha}$.
Moreover, if $f\in\mathcal{A}_{m,n}$, then $g\left(x,y\right):=f\left(x_{1}^{\alpha},x_{2},\ldots,x_{m},y\right)\in\mathcal{A}_{m,n}$
and $\mathcal{T}\left(g\right)=\mathcal{T}$$\left(f\right)\left(X_{1}^{\alpha},X_{2},\ldots,X_{m},Y\right)$.
\item \emph{Monomial division.} Let $f\in\mathcal{A}_{m,n}$ and suppose
that there exist $\alpha\in\mathbb{K},$ $n\in\mathbb{N}$ and $G\in\mathbb{R}\left\llbracket X^{*},Y\right\rrbracket $
such that $\mathcal{T}(f)\left(X,Y\right)=X_{1}^{\alpha}Y_{1}^{n}G\left(X,Y\right)$.
Then there exists $g\in\mathcal{A}_{m,n}$ such that $f\left(x,y\right)=x_{1}^{\alpha}y_{1}^{n}g\left(x,y\right)$.
It follows that $\mathcal{T}\left(g\right)=G$: in fact, $\mathcal{T}\left(f\right)=\mathcal{T}\left(x_{1}^{\alpha}y_{1}^{n}g\right)=X_{1}^{\alpha}Y_{1}^{n}\mathcal{T}\left(g\right)$
and $\mathcal{T}\left(f\right)$ is also equal to $X_{1}^{\alpha}Y_{1}^{n}G$. 
\item \emph{Permutations of the $x$-variables. }Let $\sigma\in\Sigma_{m}$
be a permutation and $f\in\mathcal{A}_{m,n}$. Then $\mathcal{T}\left(f_{\sigma}\right)=\mathcal{T}\left(f\right)_{\sigma}$.
\item \emph{Setting a variable equal to zero. }For $f\in\mathcal{A}_{m,n}$,
we have 
\[
\mathcal{T}\left(f\left(x_{1},\ldots,x_{m-1},0,y\right)\right)=\mathcal{T}\left(f\right)\left(X_{1},\ldots,X_{m-1},0,Y\right).
\]

\item \emph{Composition in the $y$-variables.} Let $g_{1},\ldots,g_{n}\in\mathcal{A}_{m',n'}$
with $g_{i}\left(0\right)=0$ and let $f\in\mathcal{A}_{m,n}$. Then
$h:=f\left(x,g_{1},\ldots,g_{n}\right)\in\mathcal{A}_{m+m',n'}$ and
$\mathcal{T}\left(h\right)=\mathcal{T}\left(f\right)\left(X,\mathcal{T}\left(g_{1}\right),\ldots,\mathcal{T}\left(g_{n}\right)\right)$.
\item \emph{Implicit functions in the $y$-variables.} Let $f\in\mathcal{A}_{m,n}$
and suppose that $\frac{\partial f}{\partial y_{n}}\left(0\right)$
exists and is nonzero. Then there exists $g\in\mathcal{A}_{m,n-1}$
such that $f\left(x,y_{1},\ldots,y_{n-1,}g\left(x,y_{1},\ldots,y_{n-1}\right)\right)=0$.
It follows that 
\[
\mathcal{T}\left(f\right)\left(X,Y_{1},\ldots,Y_{n-1},\mathcal{T}\left(g\right)\left(X,Y_{1},\ldots,Y_{n-1}\right)\right)=0.
\]

\item \emph{Blow-ups.} Let $f\in\mathcal{A}_{m,n}$ and $\pi:\hat{I}_{m',n',\infty}\to\hat{I}_{m,n,\infty}$
be a blow-up chart (see Definition \ref{def: blow-up charts}). Then
$f\circ\pi\in\mathcal{A}_{m',n'}$ and $\mathcal{T}\left(f\circ\pi\right)=\mathcal{T}\left(f\right)\circ\pi$.
\end{enumerate}
\end{void}
\begin{rems}
\label{rems: properties of the algebras}$\ $Here is a list of consequences
of the above properties.\smallskip{}

\noindent \emph{Closure under differentiation}. $\mathcal{A}_{m,n}$
is closed under partial derivatives with respect to the $y$-variables
and $\mathcal{T}$ is compatible with this operation: $f\in\mathcal{A}_{m,n}\Rightarrow\frac{\partial f}{\partial y_{i}}\in\mathcal{A}_{m,n}$
and $\mathcal{\mathcal{T}}\left(\frac{\partial f}{\partial y_{i}}\right)=\frac{\partial\mathcal{\mathcal{T}}\left(f\right)}{\partial y_{i}}$.
In fact, consider the germ $g\left(x,y,z_{1}\right):=f\left(x,\ldots,y_{i}+z_{1},\ldots\right)-f\left(x,y\right)\in\mathcal{A}_{m,n+1}$.
Then $g\left(x,y,0\right)=0$, so $\mathcal{\mathcal{T}}\left(g\right)\left(x,y,0\right)$;
hence $\mathcal{\mathcal{T}}\left(g\right)\left(x,y,z_{1}\right)=z_{1}H\left(x,y,z_{1}\right)$,
for some series $H$. By monomial division, there exists $h\in\mathcal{A}_{m,n+1}$
such that $g\left(x,y,z_{1}\right)=z_{1}h\left(x,y,z_{1}\right)$
and hence $\frac{\partial f}{\partial y_{i}}\left(x,y\right)=h\left(x,y,0\right)\in\mathcal{A}_{m,n}$.\smallskip{}

\noindent \emph{Taylor expansion.} If $f\in\mathcal{A}_{m,1}$, then
$\mathcal{\mathcal{T}}\left(f\right)\left(0,Y\right)$ is the Taylor
expansion of $f\left(0,y\right)$ with respect to $y$. In fact, $\mathcal{\mathcal{T}}\left(f\right)=\sum_{i=0}^{\infty}\frac{1}{i!}\frac{\partial^{i}\mathcal{\mathcal{T}}\left(f\right)}{\partial Y^{i}}\left(X,0\right)Y^{i}=\sum_{i=0}^{\infty}\frac{1}{i!}\mathcal{\mathcal{T}}\left(\frac{\partial^{i}f}{\partial y^{i}}\left(x,0\right)\right)Y^{i}$.\smallskip{}

\noindent \emph{Closure under $\partial_{i}$. }Let $f\in\mathcal{A}_{m,n}$
and let $\partial_{i}f$ be the germ at zero of $x_{i}\frac{\partial f}{\partial x_{i}}$
(extended by continuity at zero), for $i=1,\ldots,m$. Then $\partial_{i}f\in\mathcal{A}_{m,n}$
and $\mathcal{T}\left(\partial_{i}f\right)=\partial_{i}\mathcal{T}\left(f\right)$.
In fact, consider 
\[
g\left(x,y,z\right):=f\left(x_{1},\ldots,x_{i-1},x_{i}\left(1+z\right),x_{i+1},\ldots,x_{m},y\right)\in\mathcal{A}_{m,n+1}.
\]
Then, $\frac{\partial g}{\partial z}\left(x,y,0\right)\in\mathcal{A}_{m,n}$.
Notice that, for some $r\in\left(0,\infty\right)^{m+n+1}$, there
is a representative of $g$ (still denoted by $g$), such that
\[
\frac{\partial g}{\partial z}\left(x,y,z\right)=x_{i}\frac{\partial f}{\partial x_{i}}\left(x_{1},\ldots,x_{i}\left(1+z\right),\ldots,x_{m},y\right)\mathrm{\ on\ }I_{m,n+1,r}.
\]
 Hence, $\partial_{i}f=\frac{\partial g}{\partial z}\left(x,y,0\right)$.
The compatibility with $\mathcal{T}$ follows easily.\smallskip{}

\noindent \emph{Closure under homothety}. Let $f\in\mathcal{A}_{m}$
and $\lambda\in\zeroapinf$. Then $g\left(x\right)=f\left(\lambda x_{1},x_{2},\ldots,x_{m}\right)\in\mathcal{A}_{m}$.
In fact, let $F\left(x_{1},z,x_{2},\ldots,x_{m}\right)=f\left(x\right)\in\mathcal{A}_{m+1}$.
Using the appropriate blow-up chart involving $x_{1}$ and $z$, we
obtain $G\left(x_{1},z,x_{2}\ldots,x_{m}\right)=F\left(x_{1},x_{1}\left(\lambda+z\right),x_{2},\ldots,x_{m}\right)\in\mathcal{A}_{m+1}$
and finally $g\left(x\right)=G\left(x_{1},0,x_{2}\ldots,x_{m}\right)\in\mathcal{A}_{m}$.\smallskip{}

\noindent \emph{$\mathcal{C}^{\infty}$ germs}. Let $f\in\mathcal{A}_{1}$
be $\mathcal{C}^{\infty}$ at zero and suppose that $f\left(0\right)=0$.
Then $f\in\mathcal{A}_{0,1}$, i.e. the exponents appearing in $\mathcal{T}\left(f\right)$
are natural numbers. In fact, let $\alpha\in\left(0,\infty\right)\setminus\mathbb{N}$
be the smallest exponent of $\mathcal{\mathcal{T}}\left(f\right)$
which is not natural. Then there exists $g\in\mathcal{A}_{1}$ such
that $f\left(x\right)-c_{1}x^{n_{1}}-\ldots-c_{k}x^{n_{k}}=x^{\alpha}g\left(x\right)$
and $g\left(0\right)\not=0$. Since $f$ is $\mathcal{C}^{\infty}$
at zero, it must be $\alpha\in\mathbb{N}$.\smallskip{}

\noindent \emph{Binomial formula}. Let $\alpha\in\mathbb{K}$. Then
$\left(1+y\right)^{\alpha}\in\mathcal{A}_{0,1}$. In fact, suppose
first that $\alpha>0$. Notice that $\left(1+y\right)^{\alpha}$ is
an analytic germ and its Taylor expansion is $\sum_{n\in\mathbb{N}}\binom{\alpha}{n}y^{n}$;
let $g\left(x,y\right)=y^{\alpha}\in\mathcal{A}_{2}$. By composing
a ramification with a suitable blow-up chart, we obtain $g_{1}\left(x,y\right)=x^{\alpha}\left(1+y\right)^{\alpha}\in\mathcal{A}_{2}$
and $\mathcal{\mathcal{T}}\left(g_{1}\right)\left(X,Y\right)=X^{\alpha}\left(1+Y\right)^{\alpha}$.
By monomial division, there exists $h\in\mathcal{A}_{1}$ such that
$g_{1}\left(x,y\right)=x^{\alpha}h\left(y\right)$ and $\mathcal{\mathcal{T}}\left(h\right)\left(Y\right)=\sum_{n\in\mathbb{N}}\binom{\alpha}{n}Y^{n}$.
Hence, $h\left(y\right)=\left(1+y\right)^{\alpha}\in\mathcal{A}_{0,1}$,
by the previous remark. Now notice that $\left(1+y\right)^{-\alpha}-1$
is the solution of the implicit function problem $\left(1+y\right)^{\alpha}\left(1+z\right)-1=0$,
and hence $\left(1+y\right)^{-\alpha}\in\mathcal{A}_{0,1}$.\smallskip{}

\noindent \emph{Units.} A \emph{unit} of $\mathcal{A}_{m}$ is an
invertible element or, equivalently, a germ $U\in\mathcal{A}_{m}$
such that $U\left(0\right)\not=0$. We claim that, for all $\alpha\in\mathbb{K}$,
$U\left(x\right)^{\alpha}\in\mathcal{A}_{m}$. In fact, we may suppose
$U\left(0\right)=1$ and write $U\left(x\right)=1+\varepsilon\left(x\right)$,
where $\varepsilon\in\mathcal{A}_{m}$ and $\varepsilon\left(0\right)=0$.
By composition and the previous remark, $\left(1+\varepsilon\left(x\right)\right)^{\alpha}\in\mathcal{A}_{m}$.\end{rems}
\begin{defn}
\label{def: normal germ}Let $f\in\mathcal{A}_{m}$. We say that $f$
is \emph{normal} if there exist $\alpha\in\mathbb{A}^{m}$ and $u\in\mathcal{A}_{m}$
such that $f\left(x\right)=x^{\alpha}u\left(x\right)$ and $u\left(0\right)\not=0$. \end{defn}
\begin{lem}
\label{lem:composition with monomials}Let $f\in\mathcal{A}_{N}$
and $G=\left(g_{1},\ldots,g_{N}\right)\in\mathcal{A}_{M}^{N}$, and
suppose that the $g_{i}$ are all normal. Then $f\circ G\in\mathcal{A}_{M}$.\end{lem}
\begin{proof}
We will only treat the case $N=1,\ M=2$, the proof in the general
case being a straightforward generalisation. Let $g\left(x,y\right)=x^{\alpha}y^{\beta}U\left(x,y\right)$,
where $U\in\mathcal{A}_{2}$ is a unit. 

First, suppose that $\alpha=\beta=0$. Up to homothety, we may assume
$U\left(0,0\right)=1$ and write $U\left(x,y\right)=1+\varepsilon\left(x,y\right)$,
with $ $$\varepsilon\in\mathcal{A}_{2}$ and $\varepsilon\left(0,0\right)=0$.
By $\mathcal{A}$-analyticity of $f$ (see Definition \ref{def: A-analytic}),
there exists $\varphi\in\mathcal{A}_{0,1}$ such that $\varphi\left(x\right)=f\left(x+1\right)$.
Since $\mathcal{T}\left(\varphi\right)$ has only integer exponents,
we may apply Axiom 5 of \ref{emp:properties of the morph} and deduce
that the germ of $x\mapsto\varphi\left(\varepsilon\left(x,y\right)\right)\in\mathcal{A}_{2}$. 

Now suppose that $\alpha>0$. Let $F\left(u,x\right)=f\left(u\right)\in\mathcal{A}_{2}$.
By Axioms 1 and 7 of \ref{emp:properties of the morph}, $F\left(x^{\alpha}\left(u+1\right),x\right)$
belongs to $\mathcal{A}_{2}$. By Axiom 4 of \ref{emp:properties of the morph},
the germ $f_{1}:x\mapsto F\left(x^{\alpha},x\right)\in\mathcal{A}_{1}$.
This shows that for every germ $\tilde{f}\in\mathcal{A}_{n+1}$, we
have $\tilde{f}\left(x^{\alpha},y_{1},\ldots,y_{n}\right)\in\mathcal{A}_{n+1}$.
Let $\beta>0$ and $F_{1}\left(x,y\right)=f_{1}\left(x\right)\in\mathcal{A}_{2}$.
By Axioms 1 and 7 of \ref{emp:properties of the morph}, the germ
of $h:\left(x,y\right)\mapsto F_{1}\left(y^{\beta}x,x\right)\in\mathcal{A}_{2}$.
Now, by what we have just proved, the germ $h_{1}:\left(x,y\right)\mapsto h\left(x,y^{1/\alpha}\right)\in\mathcal{A}_{2}$.
Notice that $h_{1}\left(x,y\right)=f\left(x^{\alpha}y^{\beta}\right)$.
Recall that, by Remarks \ref{rems: properties of the algebras}, $U^{1/\beta}\in\mathcal{A}_{2}$.
We may assume $U^{1/\beta}\left(0,0\right)=1$ and write $U^{1/\beta}\left(x,y\right)=1+\varepsilon\left(x,y\right)$,
with $ $$\varepsilon\in\mathcal{A}_{2}$ and $\varepsilon\left(0,0\right)=0$.
Let $H\left(x,y,z\right)=h_{1}\left(x,z\right)\in\mathcal{A}_{3}$.
By Axiom 7 of \ref{emp:properties of the morph}, the germ $H_{1}:\left(x,y,z\right)\mapsto H\left(x,y,y\left(1+z\right)\right)\in\mathcal{A}_{2,1}$.
By Axiom 5 of \ref{rems: properties of the algebras}, the germ $h_{2}:\left(x,y\right)\mapsto H_{1}\left(x,y,\varepsilon\left(x,y\right)\right)\in\mathcal{A}_{2}$.
Since $h_{2}\left(x,y\right)=f\left(x^{\alpha}y^{\beta}U\left(x,y\right)\right)$,
we are done.
\end{proof}

\subsection{The theorems\label{sub:The-theorems}}
\begin{proviso}
\label{empty: proviso} For the rest of the paper $\mathcal{A}=\{\mathcal{A}_{m,n,r}\colon m,n\in\mathbb{N},\ r\in\zeroapinf^{m+n}\}$
will be a collection of algebras of functions as in the previous subsection,
i.e. such that:
\begin{itemize}
\item $\mathcal{A}$ satisfies the conditions in \ref{vuoto:conditions on functions}.
\item Every function in the collection in $\mathcal{A}$-analytic (see Definition
\ref{def: A-analytic}).
\item $\left\{ \mathcal{A}_{m,n}:\ m,n\in\mathbb{N}\right\} $ is a collection
of quasianalytic algebras of germs (see Definition \ref{def: quasi-analyticity}).
\item The collection $\left\{ \mathcal{A}_{m,n}:\ m,n\in\mathbb{N}\right\} $
and the morphism $\mathcal{T}$ satisfy the closure and compatibility
properties in \ref{emp:properties of the morph}.
\end{itemize}
\end{proviso}
\begin{defn}
\label{def:structures}Let $\mathcal{A}$ be as above.

For $f\in\mathcal{A}_{m,n,1}$ we define $\tilde{f}:\mathbb{R}^{m+n}\to\mathbb{R}$
as 
\[
\tilde{f}\left(x,y\right)=\begin{cases}
f\left(x,y\right) & \mathrm{if}\ \left(x,y\right)\in I_{m,n,1}\\
0 & \mathrm{otherwise}
\end{cases}.
\]

Let $\mathcal{L}_{\mathcal{A}}$ be the language of ordered rings
$\left\{ <,0,1,+,-,\cdot\right\} $ augmented by a new function symbol
for each function $\tilde{f}$. Let $\mathbb{R}_{\mathcal{A}}$ be
the real ordered field with its natural $\mathcal{L}_{\mathcal{A}}$-structure.
Let $^{-1}$ be a function symbol for $x\mapsto\frac{1}{x}$, where
$0^{-1}=0$ by convention, and, for $n\in\mathbb{N}$, let $\sqrt[n]{\ }$
be a function symbol for the function 

\[
x\mapsto\begin{cases}
x^{1/n} & \text{if\ }0\leq x\leq1\\
0 & \text{otherwise}
\end{cases}.
\]
\end{defn}
\begin{namedthm}
{Theorem A}\label{thm:o-minimality}The structure $\mathbb{R}_{\mathcal{A}}$
is model-complete, o-minimal and polynomially bounded. Its field of
exponents%
\footnote{Recall that the field of exponents of a polynomially bounded o-minimal
structure is the set of all $\alpha\in\mathbb{R}$ such the function
$x\mapsto x^{\alpha}$ is definable. It is indeed a field (see for
example \cite{miller:growth_dichotomy}).%
} is the field $\mathbb{K}$, defined in Definition \ref{def: admissible exponents}.
\end{namedthm}

\begin{namedthm}
{Theorem B}\label{thm:qe}The (natural) expansion of the structure
$\mathbb{R}_{\mathcal{A}}$ to the language $\mathcal{L}_{\mathcal{A}}\cup\left\{ ^{-1}\right\} \cup\left\{ \sqrt[n]{\ }:\ n\in\mathbb{N}\right\} $
admits quantifier elimination.\end{namedthm}
\begin{rem}
\label{rem: no need for homothety}The choice of putting in $\mathcal{L}_{\mathcal{A}}$
a function symbol only for representatives on the unit box is not
binding; we could have put a function symbol for every function in
$\mathcal{A}_{m,n,r}$ and dispensed with the last condition in \ref{vuoto:conditions on functions},
and the two above theorems would still be true. 
\end{rem}

\subsection{Examples\label{sub:Examples}}

Most known examples of o-minimal polynomially bounded expansions of
the real field can be generated by a family $\mathcal{A}$ of algebras
satisfying the requirements of Proviso \ref{empty: proviso}. In particular,
this is true for all the structures mentioned in the introduction. 

a) In \cite{vdd:speiss:gen} the authors consider, for every polyradius
$r$, the sub-algebra $\mathbb{R}\left\{ X^{*},Y\right\} _{r}$ of
$\mathbb{R}\left\llbracket X^{*},Y\right\rrbracket $ consisting of
all generalised power series which converge on $\hat{I}_{m,n,r}$
(see \cite[p. 4377]{vdd:speiss:gen}). The morphism $\mathcal{T}$
in this case is the inclusion $\mathbb{R}\left\{ X^{*},Y\right\} _{r}$$\hookrightarrow\mathbb{R}\left\llbracket X^{*},Y\right\rrbracket $,
so the quasianalyticity property is obvious, and $\mathcal{A}$-analyticity
is proved in \cite[Corollary 6.7]{vdd:speiss:gen}. 

b) In \cite{vdd:speiss:multisum} the authors consider a family of
algebras $\mathcal{G}_{m}$ of Gevrey functions in $m$ variables
(see \cite[Definition 2.20]{vdd:speiss:multisum}). The morphism $\mathcal{T}$
is the Taylor map at zero, the quasianalyticity property follows from
a fundamental result in multisummabilty theory (see \cite[Proposition 2.18]{vdd:speiss:multisum})
and $\mathcal{A}$-analyticity is proved in \cite[Lemma 4.8]{vdd:speiss:multisum}. 

c) In \cite{rsw} the authors consider a family of quasianalytic Denjoy-Carleman
algebras $\mathcal{C}_{B}\left(M\right)$, where $B$ is a box in
$\mathbb{R}^{n}$ and $M=\left(M_{1},M_{2},\ldots\right)$ is an increasing
sequence of positive constants (see \cite[p.751]{rsw}). The morphism
$\mathcal{T}$ is the Taylor map at zero, the quasianalyticity property
is equivalent to the condition $\sum_{i=0}^{\infty}\frac{M_{i}}{M_{i+1}}=\infty$
and $\mathcal{A}$-analyticity is automatically verified, since these
algebras are closed under translation. 

d) In \cite{rss} the authors consider a solution $H=\left(H_{1},\ldots,H_{r}\right):\left(-\varepsilon,\varepsilon\right)\to\mathbb{R}^{r}$
of a first order analytic differential equation which is singular
at the origin and they construct the smallest collection $\mathcal{A}_{H}$
of algebras of germs containing the germ of $H$ and all the analytic
germs, and closed under (a subset of) the operations implicit in \ref{emp:properties of the morph}.
They consider the family of algebras of functions consisting of all
$\mathcal{A}$-analytic representatives of the germs in $\mathcal{A}_{H}$
(it is proved in \cite[Lemma 3.4]{rss} that every germ in $\mathcal{A}_{H}$
has an $\mathcal{A}$-analytic representative) and let the morphism
$\mathcal{T}$ be the Taylor expansion at zero. The quasianalyticity
property in proved in \cite[Theorem 3.5]{rss} and $\mathcal{A}$-analyticity
is automatic by construction. 

e) Finally, in \cite{krs} the authors consider a family of algebras
$\mathcal{Q}_{m}^{m+n}\left(U\right)$ of real functions which have
a holomorphic extension to a so-called ``quadratic domain'' $U\subseteq\mathbf{L}^{m+n}$,
where $\mathbf{L}$ is the Riemann surface of the logarithm (see \cite[Definition 5.1]{krs}).
These algebras contain the \emph{Dulac transition maps} of real analytic
planar vector fields in a neighbourhood of hyperbolic non-resonant
singular points. The morphism $\mathcal{T}$ defined in \cite[Definition 2.6]{krs}
associates to the germ $f$ of a function $\mathcal{Q}_{m}^{m+n}$
an \emph{asymptotic expansion} $\mathcal{T}\left(f\right)\in\mathbb{R}\left\llbracket X^{*},Y\right\rrbracket $.
The quasianalyticity property is proved in \cite[Proposition 2.8]{krs}
and $\mathcal{A}$-analyticity is proved in \cite[Corollary 3.7]{krs}.
For all the above examples the closure and compatibility properties
in \ref{emp:properties of the morph} are also proven in the mentioned
papers.
\begin{rem}
\label{rem: how to construct examples}Let $f:\hat{I}_{m_{0},n_{0},r_{0}}\to\mathbb{R}$
be a \emph{weak}-$\mathcal{C}^{\infty}$ function (in the sense of
\cite{miller:infinite-differentiability}), i.e. for every point $a\in\hat{I}_{m_{0},n_{0},r_{0}}$
and for every $k\in\mathbb{N}$, the germ of $f$ at $a$ has a $\mathcal{C}^{k}$
representative. Suppose furthermore that $f$ is definable in a polynomially
bounded expansion of the real field. Let $\left\{ \mathcal{A}_{m,n}:\ m,n\in\mathbb{N}\right\} $
be smallest collection of $\mathbb{R}$-algebras of germs containing
the germs at zero of the functions $f_{a}\left(x\right):=f\left(x+a\right)$
(for all $a\in\hat{I}_{m_{0},n_{0},r_{0}}$) and of the coordinate
functions, and satisfying the closure properties in \ref{emp:properties of the morph}.
Let $\mathcal{A}=\left\{ \mathcal{A}_{m,n,r}\right\} $ be the collection
of the $\mathbb{R}$-algebras made of all $\mathcal{A}$-analytic
representatives of the germs in $\left\{ \mathcal{A}_{m,n}:\ m,n\in\mathbb{N}\right\} $.
It is easy to check (see for example \cite[Lemma 3.4]{rss}) that
every germ in $\mathcal{A}_{m,n}$ has an $\mathcal{A}$-analytic
representative. Clearly, $\mathcal{A}$ satisfies the first (see Remark
\ref{rem: no need for homothety}) and the last two conditions in
Proviso \ref{empty: proviso}. Finally, it follows from \cite[Corollary 2]{miller:infinite-differentiability}
that $\left\{ \mathcal{A}_{m,n}:\ m,n\in\mathbb{N}\right\} $ is a
collection of quasianalytic algebras. Hence, the structure $\mathbb{R}_{\mathcal{A}}$
satisfies Theorems A and B.
\end{rem}
In analogy with \cite{gabriel:expli} and \cite{rsw}, we deduce from
Theorem A the following explicit model-completeness result.
\begin{cor}
\label{cor: model compl from derivatives}Let $f:\hat{I}_{m_{0},n_{0},r_{0}}\to\mathbb{R}$
be a $\mathcal{C}^{\infty}$ function definable in some polynomially
bounded expansion of the real field. Let $\mathcal{L}_{f}$ be the
language of ordered rings $\left\{ <,0,1,+,-,\cdot\right\} $ augmented
by function symbols for $f$ and for each partial derivative of $f$
and by a constant symbol for every real number. Then the natural $\mathcal{L}_{f}$-expansion
$\mathbb{R}_{f}$ of the real field is model-complete.\end{cor}
\begin{proof}
Let us construct $\mathcal{A}$ as in Remark \ref{rem: how to construct examples}.
It is easy to see that the structures $\mathbb{R}_{\mathcal{A}}$
and $\mathbb{R}_{f}$ have the same 0-definable sets. We claim that
the primitives of $\mathbb{R}_{\mathcal{A}}$ are existentially definable.
To see this, the only non-trivial observation to make is that, if
$g\in\mathcal{A}_{m,n,r}$, and $h\in\mathcal{A}_{m,n,r}$ is obtained
from $g$ by monomial division, e.g. $g\left(x\right)=x_{1}\cdot h\left(x\right)$
(where $x=\left(x_{1},\ldots,x_{m+n}\right)$), then the graph of
$h$ is the set

\[
\left\{ \left(x,y\right):\ \left(x_{1}\not=0\land x_{1}\cdot y=g\left(x\right)\right)\vee\left(x=0\land y=\frac{\partial g}{\partial x_{1}}\left(x\right)\right)\right\} ,
\]
which is existentially definable from the graph of $g$.
\end{proof}

\subsection{Model-completeness of Euler's Gamma function\label{subsec: model completeness of gamma}}

The following application was obtained in collaboration with Gareth
Jones.
\begin{cor}
\label{cor: mod compl of gamma}Let $\mathcal{L}_{\Gamma}$ be the
language of ordered rings $\left\{ <,0,1,+,-,\cdot\right\} $ augmented
by function symbols for $\Gamma\restriction\left(0,\infty\right)$
and for each derivative of $\Gamma\restriction\left(0,\infty\right)$
and by a constant symbol for every real number. Then the natural $\mathcal{L}_{\Gamma}$-expansion
$\mathbb{R}_{\Gamma}$ of the real field is model-complete.\end{cor}
\begin{proof}
Let $\psi$ be defined as in \cite[Example 8.1]{vdd:speiss:multisum}.
In particular, recall that, by Binet's second formula, we have, for
all $x\in\left(1,\infty\right)$,
\[
\log\Gamma\left(x\right)=\left(x-\frac{1}{2}\right)\log x+\frac{1}{2}\log\left(2\pi\right)+\psi\left(\frac{1}{x}\right).
\]
As remarked in \cite[Example 8.1]{vdd:speiss:multisum}, the functions
$\psi\restriction\left(0,1\right)$ and $\Gamma\restriction\left(0,1\right)$
are both definable in the polynomially bounded o-minimal structure
$\mathbb{R}_{\mathcal{G}}$. Moreover, $\psi\restriction\left(0,1\right)$
is $\mathcal{C}^{\infty}$ at zero and $1/\Gamma\restriction\left(0,1\right)$
is analytic at zero, hence by Corollary \ref{cor: model compl from derivatives},
the expansion $\mathcal{R}$ of the real field by the functions $\exp\restriction\left(0,1\right)$,
$\psi\restriction\left(0,1\right)$ and $\Gamma\restriction\left(0,1\right)$
and their derivatives is model-complete. By \cite[Theorem B]{vdd:speiss:multisum},
the structure $\langle\mathcal{R},\exp\rangle$ is model-complete.
Now, using Legendre's Duplication Formula (see for example \cite[p.5]{erdelyi:higher_transcendental_functions_i})
\[
\Gamma\left(x\right)\Gamma\left(x+\frac{1}{2}\right)=2^{1-2x}\sqrt{\pi}\Gamma\left(2x\right),
\]
it is easy to show that the structures $\langle\mathcal{R},\exp\rangle$
and $\mathbb{R}_{\Gamma}$ have the same 0-definable sets and that
the primitives of $\langle\mathcal{R},\exp\rangle$ are existentially
definable in $\mathbb{R}_{\Gamma}$. 
\end{proof}

\section{Monomialisation\label{sec:Monomialisation of series}}

The aim of this section is to define a class of transformations of
$\mathbb{R}^{m+n}$, which we call admissible, which are bijective
outside a set with empty interior, and modulo which we may assume
that the germs in $\mathcal{A}_{m,n}$ are normal (see Definition
\ref{def: normal germ}). More precisely, in Subsection \ref{sub:Monomialisation-of-generalised}
we develop a monomialisation algorithm for generalised power series
and in Subsection \ref{subsec:monomialisation of germs} we use quasianalyticity
and the compatibility properties of the morphism $\mathcal{T}$ in
\ref{emp:properties of the morph} to deduce a monomialisation result
for germs.

\subsection{Admissible transformations and admissible trees\label{sub:Admissible-trees}}
\begin{defn}
\label{Def: elementary transformations}Let $m,n\in\mathbb{N},\ \left(x,y\right)=\left(x_{1},\ldots,x_{m},y_{1},\ldots,y_{n}\right)$.
For $m',n'\in\mathbb{N}$ with $m'+n'=m+n$, we set $\left(x',y'\right)=\left(x_{1}',\ldots,x_{m'}',y_{1}',\ldots,y_{n'}'\right)$.
Let $r,r'$ be polyradii in $\mathbb{R}^{m+n}$. An \emph{elementary
transformations} of $\mathbb{R}^{m+n}$ is a map $\nu:\hat{I}_{m',n',r'}\to\hat{I}_{m,n,r}$
of either of the following types. 
\begin{itemize}
\item A \emph{blow-up chart} (see Definition \ref{def: blow-up charts}),
i.e. an element $\pi:\hat{I}_{m',n',r'}\to\hat{I}_{m,n,r}$ of either
of the collections $\pi_{i,j},\ \pi_{m+i,j},\ \pi_{m+i,m+j}$.
\item A \emph{Tschirnhausen translation}: let $m=m',n=n'$, let $r''=\left(s_{1}',\ldots,s_{m}',t_{1}',\ldots,t_{n-1}'\right)$
and $h\in\mathcal{A}_{m,n-1,r''}$ with $h\left(0\right)=0$ and set
\[
\tau_{h}\left(x',y'\right)=\left(x,y\right),\ \ \ \mathrm{where}\ \begin{cases}
x_{k}=x{}_{k}' & 1\leq k\leq m\\
y_{n}=y_{n}'+h\left(x_{1}',\ldots,x_{m}',y_{1}',\ldots,y_{n-1}'\right)\\
y_{k}=y_{k}' & 1\leq k\leq n-1
\end{cases}.
\]

\item A \emph{linear transformation}: let $m=m',n=n'$, let $1\leq i\leq n\mathrm{\ and\ }c=\left(c_{1},\ldots,c_{i-1}\right)\in\mathbb{R}^{i-1}$,
and set
\[
L_{i,c}\left(x',y'\right)=\left(x,y\right),\ \ \ \mathrm{where}\ \begin{cases}
x_{k}=x_{k}' & 1\leq k\leq m\\
y_{k}=y_{k}' & i\leq k\leq n\\
y_{k}=y_{k}'+c_{k}y_{i}' & 1\leq k<i
\end{cases}.
\]

\item A \emph{ramification} is either of the following maps: let $m=m',n=n'$, 
\end{itemize}

for $1\leq i\leq m$ and $\gamma\in\mathbb{K}^{>0}$ (see Definition
\ref{def: admissible exponents}),

\begin{align*}
r_{i}^{\gamma} & \left(x',y'\right)=\left(x,y\right),\ \ \ \mathrm{where}\ \begin{cases}
x_{k}=x_{k}' & 1\leq k\leq m,\ k\not=i\\
x_{i}=x{}_{i}'^{\gamma}\\
y_{k}=y_{k} & 1\leq k\leq n
\end{cases}
\end{align*}
and for $1\leq i\leq n$ and $d\in\mathbb{N}$,
\begin{align*}
\end{align*}
\[
r_{m+i}^{d,\pm}\left(x',y'\right)=\left(x,y\right),\ \ \ \mathrm{where}\ \begin{cases}
x_{k}=x_{k}' & 1\leq k\leq m\\
y_{i}=\pm y{}_{i}'^{d}\\
y_{k}=y_{k} & 1\leq k\leq n,\ k\not=i
\end{cases}.
\]

\end{defn}
\begin{rem}
\label{rem: elemntary transf send quadrants to sectors}Notice that,
by \ref{emp:properties of the morph} and Lemma \ref{lem:composition with monomials},
the components of an elementary transformation $\nu:\hat{I}_{m',n',r'}\to\hat{I}_{m,n,r}$
are elements of $\mathcal{A}_{m',n',r'}$. Moreover, it follows from
the axioms in \ref{emp:properties of the morph} that, if $h\in\mathcal{A}_{m,n}$,
then $h\circ\nu\in\mathcal{A}_{m',\nu'}$. \end{rem}
\begin{defn}
\label{Def: admiss transf}Let $m,n,m',n',N\in\mathbb{N}$ with $m'+n'=m+n$
and let $\nu_{i}:\hat{I}_{m_{i}',n_{i}',r_{i}'}\to\hat{I}_{m_{i},n_{i},r_{i}}$
be elementary transformations (for $i=1,\ldots,N)$ with $m_{1}=m,\ n_{1}=n,\ m_{N}'=m',\ n_{N}'=n'$
and $m_{i}+n_{i}=m_{i}'+n_{i}'=m+n$. If $N>1$, then in order for
the composition $\nu_{1}\circ\ldots\circ\nu_{N}$ to be well-defined
it is enough that $m_{i}\geq m_{i-1}',\ n_{i}\leq n_{i-1}'$ and $r_{i}\leq r_{i-1}'$
for all $i=1,\ldots,N$. A map $\rho:\hat{I}_{m',n',r'}\to\hat{I}_{m,n,r}$
is called an \emph{admissible transformation} if $\rho=\nu_{1}\circ\ldots\circ\nu_{N}$
and moreover, if $N>1$, then $m_{i}=m_{i-1}'$ and $n_{i}=n_{i-1}'$
for all $i=1,\ldots,N$. The number $N$ is called the \emph{length}
of the admissible transformation $\rho$.\end{defn}
\begin{rem}
\label{Rem: admiss transf are in the algebra}By Remark \ref{rem: elemntary transf send quadrants to sectors}
and by induction on the length of $\rho$, it is easy to see that
the components of the admissible transformation $\rho:\hat{I}_{m',n',r'}\to\hat{I}_{m,n,r}$
are elements of $\mathcal{A}_{m',n',r'}$ and that $\rho$ induces
an algebra morphism

\[
\xyC{0mm}\xyL{0mm}\xymatrix{\rho\colon & \mathcal{A}_{m,n}\ar[rrrr] & \  & \  & \  & \mathcal{A}_{m',n'}\\
 & f\ar@{|->}[rrrr] &  &  &  & f\circ\rho
}
.
\]

\end{rem}

\begin{lem}
\label{Lem: elementary transf induce injective morph}An elementary
transformation $\nu:\hat{I}_{m',n',r'}\to\hat{I}_{m,n,r}$ induces
an injective algebra homomorphism

\[
\xyC{0mm}\xyL{0mm}\xymatrix{T_{\nu}\colon & \mathbb{R}\left\llbracket X^{*},Y\right\rrbracket \ar[rrrr] & \  & \  & \  & \mathbb{R}\left\llbracket X'^{*},Y'\right\rrbracket \\
 & F\ar@{|->}[rrrr] &  &  &  & F\circ\nu
}
,
\]
where we set
\[
F\circ\tau_{h}\left(X',Y'\right):=F\left(X',Y_{1}',\ldots,Y_{n-1}',Y_{n}'+\mathcal{T}\left(h\right)\left(X',Y_{1}',\ldots Y_{n-1}'\right)\right).
\]
Moreover, if $F\in\mathbb{R}\left\llbracket X^{*},Y\right\rrbracket \cap\text{Im\ensuremath{\left(\mathcal{T}\right)}},$
then $F\circ\nu\in\mathbb{R}\left\llbracket X'^{*},Y'\right\rrbracket \cap\text{Im\ensuremath{\left(\mathcal{T}\right)}}.$\end{lem}
\begin{proof}
It is clear that $T_{\nu}$ is a homomorphism, which preserves being
a member of $\text{Im}\left(\mathcal{T}\right)$ by the properties
in \ref{emp:properties of the morph}. If $\nu$ is either a linear
transformation, or a Tschirnhausen translation or a ramification,
then $\nu:\left(X',Y'\right)\mapsto\left(X,Y\right)$ is a bijective
change of coordinates, hence $T_{\nu}$ is injective. It remains to
show injectivity when $\nu$ is a blow-up chart. In order to do this,
let $u,v,w_{1},\ldots,w_{l}$ be variables, with $\bar{w}=\left(w_{1},\ldots,w_{l}\right)$,
and, for $\lambda\geq0$, consider the map $\pi^{\lambda}:\left(u,v,\bar{w}\right)\mapsto\left(u,u\left(\lambda+v\right),\bar{w}\right)$.
For $F\in\mathbb{R}\left\llbracket \left(u,v,\bar{w}\right)^{*}\right\rrbracket \setminus\left\{ 0\right\} $,
we prove that $F\circ\pi^{\lambda}\not\equiv0$. Write $F\left(u,v,\bar{w}\right)=\sum_{\alpha,\beta}a_{\alpha,\beta}\left(\bar{w}\right)u^{\alpha}v^{\beta}$,
where $a_{\alpha,\beta}\in\mathbb{R}\left\llbracket \bar{w}^{*}\right\rrbracket $.
Let us regroup the homogeneous terms as follows:

\[
F\left(u,v,\bar{w}\right)=\sum_{\gamma}Q_{\gamma}\left(u,v,\bar{w}\right)\ \text{where}\ Q_{\gamma}\left(u,v,\bar{w}\right)=\sum_{\alpha+\beta=\gamma}a_{\alpha,\beta}\left(\bar{w}\right)u^{\alpha}v^{\beta}.
\]
It follows from the well-order properties of the support (see for
example \cite[Lemma 4.2 (2)]{vdd:speiss:gen}) that $Q_{\gamma}$
is a finite sum, hence let us rewrite $Q_{\gamma}={\displaystyle \sum_{i=1}^{q}}c_{\beta_{i}}u^{\gamma-\beta_{i}}v^{\beta_{i}}$,
where $c_{\beta_{i}}=a_{\gamma-\beta_{i},\beta_{i}}\left(\bar{w}\right)$
and $\beta_{1}<\ldots<\beta_{q}$.

Let us first consider the case $\lambda=0$. Suppose that $0\equiv F\circ\pi^{0}\left(u,v,\bar{w}\right)={\displaystyle \sum_{\gamma}}u^{\gamma}{\displaystyle \sum_{\alpha+\beta=\gamma}}a_{\alpha,\beta}\left(\bar{w}\right)v^{\beta}$.
Then for every $\gamma$ we have that $a_{\alpha,\beta}\equiv0$ whenever
$\alpha+\beta=\gamma$, hence $F\equiv0$. 

Now suppose $\lambda>0$ and assume that 

\[
0\equiv F\circ\pi^{\lambda}\left(u,v,\bar{w}\right)={\displaystyle \sum_{\gamma}}u^{\gamma}{\displaystyle \sum_{\alpha+\beta=\gamma}}a_{\alpha,\beta}\left(\bar{w}\right){\displaystyle \sum_{k=0}^{\infty}}\binom{\beta}{k}\lambda^{\beta-k}v^{k}.
\]
 Then for every $\gamma$ the series 

\[
Q_{\gamma}\left(1,\lambda+v,\bar{w}\right)={\displaystyle \sum_{k=0}^{\infty}}\left({\displaystyle \sum_{i=1}^{q}}c_{\beta_{i}}\binom{\beta_{i}}{k}\lambda^{\beta_{i}-k}\right)v^{k}\in\mathbb{R}\left\llbracket \bar{w}^{*},v\right\rrbracket 
\]
is identically zero. It follows that ${\displaystyle \sum_{i=1}^{q}}c_{\beta_{i}}\lambda^{\beta_{i}}\binom{\beta_{i}}{k}=0\ \forall k\in\mathbb{N}$.
Now, it is not difficult to see that there exist $j_{1},\ldots,j_{q}\in\mathbb{N}$
such that the determinant of the linear system $\left\{ \sum_{i=1}^{q}\binom{\beta_{i}}{j_{s}}Z_{i}=0\ \text{for}\ s=1,\ldots q\right\} $
is nonzero and hence the only solution is $Z_{1}=\ldots=Z_{q}=0$.
It follows that for every $\gamma$ we have that $a_{\alpha,\beta}\equiv0$
whenever $\alpha+\beta=\gamma$, hence $F\equiv0$. \end{proof}
\begin{defn}
\label{Def: admissible trees}An \emph{elementary tree} has either
of the following forms: a vertex $\bullet$, or
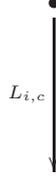
\[
\xyL{2cm}\xymatrix{\bullet\ar@{->}[d]\sb-{L_{i,c}}\\
\ 
}
\]
where $L_{i,c}$ is a linear transformation, or
\[
\xyL{2cm}\xymatrix{\bullet\ar@{->}[d]\sb-{\tau_{h}}\\
\ 
}
\]
where $\tau_{h}$ is a Tschirnhausen transformation, or

\[
\xyL{2cm}\xymatrix{\bullet\ar@{->}[d]\sb-{r_{i}^{\gamma}}\\
\ 
}
\]
where $r_{i}^{\gamma}$ is a ramification of the first type, or\foreignlanguage{english}{
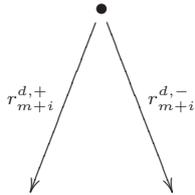
\[
\xyL{1cm}\xyC{6mm}\xymatrix{ & \bullet\ar@{->}[ldd]\sb-{r_{m+i}^{d,+}}\ar@{->}[rdd]\sp-{r_{m+i}^{d,-}}\\
\\
\  & \  & \ 
}
\]
}where $r_{m+i}^{d,+},\ r_{m+i}^{d,-}$ are ramifications of the second
type, or

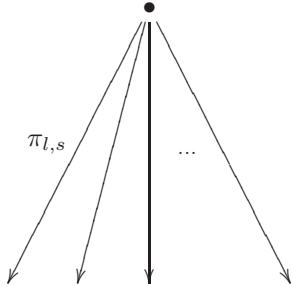
\[
\xyL{1cm}\xyC{6mm}\xymatrix{ &  & \bullet\ar@{->}[llddd]\sb-{\displaystyle \pi_{l,s}}\ar@{->}[lddd]\ar@{->}[ddd]\ar@{}|-{\ldots}[rddd]\ar@{->}[rrddd]\\
\\
\\
\  & \  & \  & \  & \ 
}
\]
where the pair $\left(l,s\right)$ is of the form $\left(i,j\right),\left(m+i,j\right)$
or $\left(m+i,m+j\right)$ (corresponding to the three types of blow-up
transformation) and each branch corresponds to an element of the collection
$\pi_{l,s}$.

\smallskip{}

A \emph{tree} is defined inductively as follows: a tree of \emph{height
}zero is a vertex $\bullet$ ; a tree of height $\leq N$ is obtained
from a tree $T$ of height $\leq N-1$ by adjoining an elementary
tree to the end of each branch of $T$. The height of $T$ will be
denoted by $\mathrm{h}\left(T\right)$. A branch $\mathfrak{b}$ of
a tree $T$ can be represented as an ordered tuple (from the vertex
to the end of the branch) of elementary transformations $\left(\nu_{1},\ldots,\nu_{N}\right)$. 

\smallskip{}

An \emph{admissible tree} is a tree $T$ such that for each branch
$\mathfrak{b}=\left(\nu_{1},\ldots,\nu_{N}\right)$ of $T$, the map
$\rho_{\mathfrak{b}}=\nu_{1}\circ\ldots\circ\nu_{N}$ is an admissible
transformation. It follows that $\mathfrak{b}$ induces a homomorphism
$T_{\mathfrak{b}}:\mathbb{R}\left\llbracket X^{*},Y\right\rrbracket \to\mathbb{R}\left\llbracket X'^{*},Y'\right\rrbracket $,
by setting $T_{\mathfrak{b}}\left(F\right)=F\circ\nu_{1}\circ\ldots\circ\nu_{N}$. \end{defn}
\begin{rems}
\label{rem:tree preserves good families}$\ $
\begin{enumerate}
\item Let $T$ be an admissible tree and $\mathfrak{b}$ be one of its branches.
Since $T_{\mathfrak{b}}$ is a homomorphism, if $U$ is a unit, then
$T_{\mathfrak{b}}\left(U\right)$ is also a unit.
\item Let $X=\left(X_{1},\ldots,X_{m}\right)$, $Y=\left(Y_{1},\ldots,Y_{n}\right)$.
It is easy to see that, if $T$ is an admissible tree and $\mathcal{F}\subset\mathbb{R}\left\llbracket X^{*},Y\right\rrbracket $
is a family with good total support, then, for every branch $\mathfrak{b}$
of $T$, the family $T_{\mathfrak{b}}\left(\mathcal{F}\right)$ still
has good total support. More precisely, let $F\in\mathbb{R}\left\llbracket X^{*},Y\right\rrbracket $
and $H\in\mathbb{R}\left\llbracket X^{*},\hat{Y}\right\rrbracket $,
where $Y=\left(Y_{1},\ldots,Y_{n}\right)$ and $\hat{Y}=\left(Y_{1},\ldots,Y_{n-1}\right)$;
suppose $\mathrm{Supp}_{X}\left(F\right)\subseteq S_{1}\times\ldots\times S_{m}$
and $\mathrm{Supp}_{X}\left(H\right)\subseteq S'_{1}\times\ldots\times S'_{m}$,
where $S_{i},S'_{i}\subset[0,\infty)$ are well ordered sets. Then
we have $\mathrm{Supp}_{X}\left(L_{i,c}\left(F\right)\right)=\mathrm{Supp}_{X}\left(r_{m+i}^{d,\pm}\left(F\right)\right)=\mathrm{Supp}_{X}\left(F\right)$
and $\mathrm{Supp}_{X}\left(\tau_{h}\left(F\right)\right)\subseteq\tilde{S}_{1}\times\ldots\times\tilde{S}_{m}$,
with $\tilde{S}_{k}=\left\{ a+nb:\ a\in S_{k},\ b\in S'_{k},\ n\in\mathbb{N}\right\} $.
Moreover, $\mathrm{Supp}_{X}\left(r_{i}^{\gamma}\left(F\right)\right)\subseteq\tilde{S}_{1}\times\ldots\times\tilde{S}_{m}$,
with $\tilde{S}_{i}=\left\{ \gamma a:\ a\in S_{i}\right\} $ and $\tilde{S}_{k}=S_{k}$
for $k\not=i$. Finally, for $1\leq i,j\leq m$ and $\lambda\in[0,\infty)$,
we have $\mathrm{Supp}_{X}\left(\pi_{i,j}^{\lambda}\left(F\right)\right)\subseteq\tilde{S}_{1}\times\ldots\times\tilde{S}_{m}$,
with $\tilde{S}_{j}=\left\{ a+b:\ a\in S_{j},\ b\in S_{i}\right\} $
and $\tilde{S}_{k}=S_{k}$ for $k\not=j$ (the argument for the other
types of blow-up transformation is similar).
\item Let $T$ be an admissible tree and $\mathfrak{b}$ be one of its branches,
inducing a homomorphism $T_{\mathfrak{b}}:\mathbb{R}\left\llbracket X^{*},Y\right\rrbracket \to\mathbb{R}\left\llbracket X'^{*},Y'\right\rrbracket $.
It follows from Lemma \ref{Lem: elementary transf induce injective morph}
that if $F\in\mathbb{R}\left\llbracket X^{*},Y\right\rrbracket \setminus\left\{ 0\right\} \cap\text{Im\ensuremath{\left(\mathcal{T}\right)}}$,
then $T_{\mathfrak{b}}\left(F\right)\in\mathbb{R}\left\llbracket X'^{*},Y'\right\rrbracket \setminus\left\{ 0\right\} \cap\text{Im\ensuremath{\left(\mathcal{T}\right)}}.$
\end{enumerate}
\end{rems}

\subsection{Monomialisation of generalised power series\label{sub:Monomialisation-of-generalised}}
\begin{defn}
A series $F\in\mathbb{R}\left\llbracket X^{*}\right\rrbracket \setminus\left\{ 0\right\} $
is \emph{normal} if there exist $\alpha\in\mathbb{A}^{m}$ and an
invertible series $U\in\mathbb{R}\left\llbracket X^{*}\right\rrbracket $
such that $F\left(X\right)=X^{\alpha}U\left(X\right)$ .
\end{defn}
In analogy with \cite[Lemma 4.7]{bm_semi_subanalytic} and \cite[Lemma 2.2]{rsw},
we have the following result.
\begin{lem}
\label{lem: prod of series}$\ $ 
\begin{enumerate}
\item The series $F_{1},\ldots,F_{k}\in\mathbb{R}\left\llbracket X^{*}\right\rrbracket \setminus\left\{ 0\right\} $
are all normal if and only if the series $\prod_{i=1}^{k}F_{i}$ is
normal.
\item If $F_{1},F_{2}$ and $F_{1}-F_{2}$ are normal, then either $F_{1}|F_{2}$
or $F_{2}|F_{1}$.
\end{enumerate}
\end{lem}
\begin{proof}
Regarding 1., suppose that $\prod F_{i}\left(X\right)$ is normal.
Then there exist two disjoint subsets $Z$ and $W$ of the set $X$
of all the variables, and multi-indices $\alpha,\ \beta$, and series
$U\left(X\right),\tilde{F_{1}}\left(X\right),\ldots,\tilde{F}_{k}\left(X\right)$
such that
\[
Z^{\alpha}\prod\tilde{F}_{i}\left(X\right)=W^{\beta}U\left(X\right),
\]
where all the components of $\alpha$ and $\beta$ are strictly positive,
no power of any of the variables divides any of the $\tilde{F}_{i}$
and $U\left(0\right)\not=0$.

If $W=\emptyset$, then the $\tilde{F}_{i}$ are all units, and the
statement is proved. Otherwise, suppose that $X_{1}\in W$; then $\prod\tilde{F}_{i}\left(0,X_{2},\ldots,X_{m+n}\right)=0$.
But this is impossible, because no power of $X_{1}$ divides any of
the $\tilde{F}_{i}$.

Regarding 2., there are multi-indices $\alpha_{1},\ \alpha_{2}$ and
units $U_{1},\ U_{2}$ such that $F_{1}=X^{\alpha_{1}}U_{1}$ and
$F_{2}=X^{\alpha_{2}}U_{2}$. Hence, the minimal elements of $\mathcal{\mathrm{Supp}}\left(F_{1}-F_{2}\right)$
are contained in the set $\left\{ \alpha_{1},\alpha_{2}\right\} $.
Since $F_{1}-F_{2}$ is normal, either $\alpha_{1}\leq\alpha_{2}$
or $\alpha_{2}\leq\alpha_{1}$.\end{proof}
\begin{notation}
We fix $m,n,m',n'\in\mathbb{N}$ such that $m+n=m'+n'$. Let $X=\left(X_{1},\ldots,X_{m}\right),\ Y=\left(Y_{1},\ldots,Y_{n}\right)$
and $X'=\left(X_{1}',\ldots,X_{m'}'\right),\ Y'=\left(Y_{1}',\ldots,Y_{n'}'\right)$.
If $T$ is an admissible tree and $\mathfrak{b}$ is one of its branches,
we will always implicitly assume that $T_{\mathfrak{b}}:\mathbb{R}\left\llbracket X^{*},Y^{*}\right\rrbracket \to\mathbb{R}\left\llbracket X'^{*},Y'\right\rrbracket $.
Let $\hat{Y}=\left(Y_{1},\ldots,Y_{n-1}\right)$ and $\hat{Y'}=\left(Y'_{1},\ldots,Y'_{n'-1}\right)$. 
\end{notation}
The main result of this subsection is the following monomialisation
algorithm for generalised power series. The proof methods take inspiration
from \cite{bm_semi_subanalytic} and \cite{rsw}.
\begin{thm}
\label{thm: monomialisation}Let $F_{1},\ldots,F_{p}\in\mathbb{R}\left\llbracket X^{*},Y\right\rrbracket \setminus\left\{ 0\right\} \cap\mathrm{Im}\left(\mathcal{T}\right)$.
There exists an admissible tree $T$ such that for each branch $\mathfrak{b}$
of $T$ the series $T_{\mathfrak{b}}\left(F_{1}\right),\ldots,T_{\mathfrak{b}}\left(F_{p}\right)\in\mathbb{R}\left\llbracket X'^{*},Y'\right\rrbracket \setminus\left\{ 0\right\} \cap\mathrm{Im}\left(\mathcal{T}\right)$
are normal and linearly ordered by division.\end{thm}
\begin{rem}
In the proof of Theorem \ref{thm: monomialisation} we will often
compose branches of different admissible trees. We leave it to the
reader to check at each stage that such compositions are allowed,
i.e. the composition defines a branch of some admissible tree.\end{rem}
\begin{lem}
\label{lem:singular blow ups}Let $\mathcal{F}=\left\{ F_{k}:\ k\in\mathbb{N}\right\} \subset\mathbb{R}\left\llbracket X^{*},Y\right\rrbracket $
be a family such that $\mathrm{Supp}_{X}\left(\mathcal{F}\right)$
is good. Then there is an admissible tree $T$ (consisting of ramifications
and blow-up transformations) such that, for every branch $\mathfrak{b}$
of $T$, we have $T_{\mathfrak{b}}:\mathbb{R}\left\llbracket X^{*},Y\right\rrbracket \to\mathbb{R}\left\llbracket X'^{*},Y'\right\rrbracket $
with $m'\leq m$ and there exist $\alpha\in[0,\infty)^{m'}$ and series
$G_{k}\in\mathbb{R}\left\llbracket X'^{*},Y'\right\rrbracket \ \left(k\in\mathbb{N}\right)$
such that, for every $k\in\mathbb{N}$, $T_{\mathfrak{b}}\left(F_{k}\right)\left(X',Y'\right)=X'^{\alpha}G_{k}\left(X',Y'\right)$
and for some $k_{0}\in\mathbb{N}$, $G_{k_{0}}\left(0,Y'\right)\not\equiv0$.\end{lem}
\begin{proof}
Here we view $\mathcal{F}$ as a subset of $\mathbb{A}\left\llbracket X^{*}\right\rrbracket $,
with $\mathbb{A}=\mathbb{R}\left\llbracket Y\right\rrbracket $. In
\cite[4.11]{vdd:speiss:gen} the authors define the \emph{blow-up
height} of a finite set of monomials, denoted by $b_{X}$. It follows
from the definition of $b_{X}$ that if $b_{X}\left(\mathcal{F}_{\mathrm{min}}\right)=\left(0,0\right)$,
then there exists $\alpha\in[0,\infty)^{m}$ such that $\mathcal{F}_{\mathrm{min}}=\left\{ X^{\alpha}\right\} $.
The proof is by induction on the pairs $\left(m,b_{X}\left(\mathcal{F}_{\text{min}}\right)\right)$,
ordered lexicographically. If $m=0$, there is nothing to prove. If
$m=1$, then $b_{X}\left(\mathcal{F}_{\mathrm{min}}\right)=\left(0,0\right)$.
In this case, for every $k\in\mathbb{N}$, there are a series $G_{k}$
such that $F_{k}=X^{\alpha}G_{k}$ and for some $k_{0}\in\mathbb{N}$
we have $G_{k}\left(0,Y\right)\not\equiv0$. 

Hence we may assume that $m>1$ and $b_{X}\left(\mathcal{F}_{\mathrm{min}}\right)\not=\left(0,0\right)$.
It follows from the proof of \cite[Proposition 4.14]{vdd:speiss:gen}
that there are $i,j\in\left\{ 1,\ldots,m\right\} $ and suitable ramifications
$r_{i}^{\gamma},\ r_{j}^{\delta}$ of the variables $X_{i}$ and $X_{j}$
such that $b_{X}(r_{i}^{\gamma}\circ r_{j}^{\delta}\circ\mathfrak{\pi}_{i,j}^{0}\left(\mathcal{F}_{\mathrm{min}}\right))<b_{X}\left(\mathcal{F}_{\mathrm{min}}\right)$
and $b_{X}(r_{i}^{\gamma}\circ r_{j}^{\delta}\circ\pi_{i,j}^{\infty}\left(\mathcal{F}_{\mathrm{min}}\right))<b_{X}\left(\mathcal{F}_{\mathrm{min}}\right)$
for all $k\in\mathbb{N}$. Notice that $r_{i}^{\gamma}\circ r_{j}^{\delta}\circ\mathfrak{\pi}_{i,j}^{0}\left(\mathcal{F}_{\mathrm{min}}\right)$
and $r_{i}^{\gamma}\circ r_{j}^{\delta}\circ\mathfrak{\pi}_{i,j}^{\infty}\left(\mathcal{F}_{\mathrm{min}}\right)$
are finite collections of monomials and they are in fact the collections
of the minimal monomials of the families $\mathcal{F}_{0}=\left\{ r_{i}^{\gamma}\circ r_{j}^{\delta}\circ\pi_{i,j}^{0}\left(F\right):\ F\in\mathcal{F}\right\} $and
$\mathcal{F}_{\infty}=\left\{ r_{i}^{\gamma}\circ r_{j}^{\delta}\circ\pi_{i,j}^{\infty}\left(F\right):\ F\in\mathcal{F}\right\} $,
respectively. Moreover, for every series $ $ and every $\lambda\in\left(0,\infty\right)$,
the series of the family $\mathcal{F}_{\lambda}=\left\{ r_{i}^{\gamma}\circ r_{j}^{\delta}\circ\pi_{i,j}^{\lambda}\left(F\right):\ F\in\mathcal{F}\right\} $
belong to $\mathbb{R}\left\llbracket X'^{*},Y'\right\rrbracket $,
where $m'=m-1$. By Remark \ref{rem:tree preserves good families}
the families $\mathcal{F}_{\lambda}$ ($\lambda\in\left[0,\infty\right]$)
have good total support, so the inductive hypothesis applies and we
obtain the required conclusion.\end{proof}
\begin{rem}
\label{rem:power of a unit}Let $F=X^{\alpha}U\left(X\right)\in\mathbb{R}\left\llbracket X^{*}\right\rrbracket \cap\text{Im}\left(\mathcal{T}\right)$
be a normal series and let $k\in\mathbb{N}$. Then $F^{1/k}$ is a
well defined series and belongs also to $\text{Im}\left(\mathcal{T}\right)$.
In fact, $U^{1/k}-1$ can be viewed as the solution $Y$ to the implicit
function problem $\left(1+Y\right)^{k}-U\left(X\right)=0$.
\end{rem}
\bigskip{}

\begin{proof}
[Proof of Theorem \ref{thm: monomialisation}] The proof is by induction
on $m+n$. In view of Lemma \ref{lem: prod of series}, it is enough
to prove the theorem for one series $F\in\mathbb{R}\left\llbracket X^{*},Y\right\rrbracket $.
By Lemma \ref{lem:singular blow ups}, we may assume that $n>0$ and
there exist $\alpha\in[0,\infty)^{m}$ and a series $G\in\mathbb{R}\left\llbracket X{}^{*},Y\right\rrbracket $
such that $F\left(X,Y\right)=X{}^{\alpha}G\left(X,Y\right)$ and $G\left(0,Y\right)\not\equiv0$.
It is well known (see for example \cite[6.1]{vdd:speiss:gen}) that,
after performing some linear transformation of the form $L_{n,c}$,
the series $G$ becomes regular in the variable $Y_{n}$, of some
order $d\in\mathbb{N}$.

Suppose that $d=1$. Then $\frac{\partial G}{\partial Y_{n}}\left(0\right)\not=0$.
Hence, by the Implicit Function Theorem, there exists a series $A\left(X,\hat{Y}\right)$
such that $G\left(X,\hat{Y},A\left(X,\hat{Y}\right)\right)=0$ and
$A\left(0,0\right)=0$. By Axiom 6 of \ref{emp:properties of the morph},
$A\in\text{Im\ensuremath{\left(\mathcal{T}\right)}}$, so $A\left(X,\hat{Y}\right)=\mathcal{T}\left(a\left(x,\hat{y}\right)\right),$
for some germ $a\in\mathcal{A}_{m,n-1}$. Then, $\tau_{a}\left(G\right)=G\left(X,\hat{Y},A\left(X,\hat{Y}\right)\right)+Y_{n}U\left(X,Y\right)$,
for some unit $U\left(X,Y\right)\in\text{Im\ensuremath{\left(\mathcal{T}\right)}}$,
and we are done.

Suppose next that $d>1$. Since $\frac{\partial^{d}G}{\partial Y_{n}^{d}}\left(0\right)\not=0$,
by the Implicit Function Theorem, there exists a series $B\left(X,\hat{Y}\right)\in\text{Im\ensuremath{\left(\mathcal{T}\right)}}$
(hence $B\left(X,\hat{Y}\right)=\mathcal{T}\left(b\left(x,\hat{y}\right)\right),$
for some germ $b\in\mathcal{A}_{m,n-1}$) such that $\frac{\partial^{d-1}G}{\partial Y_{n}^{d-1}}\left(X,\hat{Y},B\left(X,\hat{Y}\right)\right)=0$
and $B\left(0,0\right)=0$. Hence by Taylor's Theorem,
\[
\tau_{b}\left(G\right)\left(X,Y\right)=U\left(X,Y\right)Y_{n}^{d}+G_{1}\left(X,\hat{Y}\right)Y_{n}^{d-1}+\ldots+G_{d}\left(X,\hat{Y}\right),
\]
where $U$ is a unit and $G_{i}=\frac{\partial^{d-i}G}{\partial Y_{n}^{d-i}}\left(X,\hat{Y},B\left(X,\hat{Y}\right)\right)$.
Hence $G_{1}\left(X,\hat{Y}\right)=0$.

By the inductive hypothesis, there is an admissible tree $\check{T}$
(acting as the identity on $Y_{n}$) such that, for every branch $\check{\mathfrak{b}}$
of $\check{T}$, the series in the set $\left\{ \check{T}_{\check{\mathfrak{b}}}\left(G_{i}\right):\ i=2,\ldots,d\right\} \cup\left\{ \check{T}_{\check{\mathfrak{b}}}\left(X^{\alpha}\right)\right\} $
are normal. Let $r^{\pm}:=r_{m+1}^{d!,\pm}\circ\ldots\circ r_{m+n}^{d!,\pm}$
and let $\check{G_{i}}:=\check{T}_{\check{\mathfrak{b}}}\circ r^{+}\left(G_{i}\right)$
(an argument analogous to the one which follows will hold for $r^{-}$).
By the inductive hypothesis and Remark \ref{rem:power of a unit},
there is an admissible tree $\tilde{T}$ (acting as the identity on
$Y_{n}$) such that, for every branch $\tilde{\mathfrak{b}}$ of $\tilde{T}$,
the series in the set $\left\{ \tilde{T}_{\tilde{\mathfrak{b}}}\left[\check{G_{i}}^{1/i}\right]:\ i=2,\ldots,d\right\} \cup\left\{ \check{T}_{\check{\mathfrak{b}}}\circ r^{+}\circ\tilde{T}_{\tilde{\mathfrak{b}}}\left(X^{\alpha}\right)\right\} $
are normal and linearly ordered by division i.e. $\tilde{T}_{\tilde{\mathfrak{b}}}\left[\check{G_{i}}^{1/i}\right]=\left(X'\hat{Y'}\right)^{\frac{\alpha_{i}}{i}}U_{i}\left(X',\hat{Y'}\right)$,
where $U_{i}$ are units and $\alpha_{i}=\left(\alpha_{i,1},\ldots,\alpha_{i,m+n-1}\right)\in[0,\infty)^{m'}\times\mathbb{N}^{n'-1}$,
with $i|\alpha_{i,m'+j}$ for all $j=1,\ldots,n'$. Since $\tilde{T}_{\tilde{\mathfrak{b}}}$
is a ring homomorphism, if we put $T_{\mathfrak{b}}:=\tau_{B}\circ\check{T}_{\check{\mathfrak{b}}}\circ r^{+}\circ\tilde{T}_{\tilde{\mathfrak{b}}}$,
we obtain
\[
T_{\mathfrak{b}}\left(G\right)=\hat{U}\left(X',Y'\right)Y_{n}'^{d}+\sum_{i=2}^{d}\left(X'\hat{Y'}\right)^{\alpha_{i}}\hat{U}_{i}\left(X',\hat{Y'}\right)Y_{n'}'^{d-i},
\]
 where $\hat{U}=\check{T}_{\check{\mathfrak{b}}}\circ r^{+}\circ\tilde{T}_{\tilde{\mathfrak{b}}}\left(U\right)$
and $\hat{U}_{i}=U_{i}^{i}$.

Let us rename $X'=X$ and $Y'=Y$. We are going to perform a series
of ramifications and blow-ups with the aim of decreasing the order
of $T_{\mathfrak{b}}\left(G\right)$ in the variable $Y_{n}$, possibly
after factoring out a monomial in the variables $\left(X,\hat{Y}\right)$. 

Let $l\in\left\{ 2,\ldots,d\right\} $ be maximal such that $\frac{\alpha_{l}}{l}\leq\frac{\alpha_{i}}{i}$
for all $i=2,\ldots,d$. 

We first consider the variables in the tuple $X$. Suppose that $X_{1}$
does appear in the monomial $\left(X\hat{Y}\right)^{\alpha_{l}}$
and let $\gamma=\frac{\alpha_{l,1}}{l}>0$. Notice that $d-i+\frac{\alpha_{1,i}}{\gamma}\geq d$
for all $i=2,\ldots d$. Let $\hat{X}=\left(X_{2},\ldots,X_{m}\right)$
and $\alpha_{i}'=\left(\alpha_{i,2},\ldots,\alpha_{i,m+n-1}\right)$.
We perform the ramification $r_{1}^{1/\gamma}$, so 
\[
T_{\mathfrak{b}}\circ r_{1}^{1/\gamma}\left(G\right)=\overline{U}\left(X,Y\right)Y_{n}^{d}+\sum_{i=2}^{d}\left(\hat{X}\hat{Y}\right)^{\alpha_{i}'}\overline{U}_{i}\left(X,\hat{Y}\right)X_{1}^{\alpha_{1,i}/\gamma}Y_{n}^{d-i},
\]

where $\overline{U}=r_{1}^{1/\gamma}\left(\hat{U}\right)$ and $\overline{U}_{i}=r_{1}^{1/\gamma}\left(\hat{U}_{i}\right)$. 

We consider the blow-up charts in the collection $\pi_{m+n,1}$. For
a better readability, in what follows we will still denote the variables
as $\left(X,Y\right)$ after blow-up transformation. 

\textbf{Case 1. $\lambda=\pm\infty$}

We will only treat the case $\lambda=+\infty$, the other case is
analogous. The transformation $\pi_{m+n,1}^{+\infty}$ maps $\left(X,Y\right)$
to $\left(X_{1}Y_{n},X_{2},\ldots,X_{m},Y\right)$. Then, 
\begin{align*}
T_{\mathfrak{b}}\circ r_{1}^{1/\gamma}\circ\pi_{m+n,1}^{+\infty}\left(G\right) & =\tilde{U}\left(X,Y\right)Y_{n}^{d}+\sum_{i=2}^{d}\left(\hat{X}\hat{Y}\right)^{\alpha_{i}'}\tilde{U}_{i}\left(X,Y\right)X_{1}^{\alpha_{1,i}/\gamma}Y_{n}^{d-i+\frac{\alpha_{1,i}}{\gamma}}\\
 & =Y_{n}^{d}\left[\tilde{U}\left(X,Y\right)+\sum_{i=2}\left(\hat{X}\hat{Y}\right)^{\alpha_{i}'}\tilde{U}_{i}\left(X,Y\right)X_{1}^{\alpha_{1,i}/\gamma}Y_{n}^{-i+\frac{\alpha_{1,i}}{\gamma}}\right]\\
 & =Y_{n}^{d}V\left(X,Y\right),
\end{align*}
where $\tilde{U}=\pi_{m+n,1}^{+\infty}\left(\overline{U}\right)$,
$\tilde{U}_{i}=\pi_{m+n,1}^{+\infty}\left(\overline{U_{i}}\right)$
and $V$ has the form $\tilde{U}+X_{1}^{\beta}B\left(X,Y\right)$,
for some $\beta\in\mathbb{K}^{>0}$ and $B\in\mathbb{R}\left\llbracket X^{*},Y\right\rrbracket $,
and hence is a unit. 

\textbf{Case 2.} $\lambda=0$

The transformation $\pi_{m+n,1}^{0}$ maps $\left(X,Y\right)$ to
$\left(X,\hat{Y},X_{1}Y_{n}\right)$. Then, 
\begin{align*}
T_{\mathfrak{b}}\circ r_{1}^{1/\gamma}\circ\pi_{m+n,1}^{0}\left(G\right) & =\tilde{U}\left(X,Y\right)X_{1}^{d}Y_{n}^{d}+\sum_{i=2}^{d}\left(\hat{X}\hat{Y}\right)^{\alpha_{i}'}\tilde{U}_{i}\left(X,Y\right)X_{1}^{d-i+\frac{\alpha_{1,i}}{\gamma}}Y_{n}^{d-i}\\
 & =X_{1}^{d}\left[\tilde{U}\left(X,Y\right)Y_{n}^{d}+\left(\hat{X}\hat{Y}\right)^{\alpha_{l}'}\tilde{U}_{l}\left(X,Y\right)Y_{n}^{d-l}\right.\\
 & \ \ \ \ \ \ \ \ \ \ \ \ \ \ \ \ \ \ \ \ \ \ \ \ \ \ \left.+\sum_{i=2,\ i\not=l}^{d}\left(\hat{X}\hat{Y}\right)^{\alpha_{i}'}\tilde{U}_{i}\left(X,Y\right)X_{1}^{-i+\frac{\alpha_{1,i}}{\gamma}}Y_{n}^{d-i}\right]\\
 & =X_{1}^{d}G_{0}\left(X,Y\right),
\end{align*}
where $\tilde{U}=\pi_{m+n,1}^{0}\left(\overline{U}\right)$ and $\tilde{U}_{i}=\pi_{m+n,1}^{0}\left(\overline{U_{i}}\right)$.
Notice that $X_{1}$ does not appear any more in the coefficient of
$Y_{n}^{d-l}$ in $G_{0}$. 

\textbf{Case 3.} $\lambda\not=0,\pm\infty$

The transformation $\pi_{m+n,1}^{\lambda}$ maps $\left(X,Y\right)$
to $\left(X,\hat{Y},X_{1}\left(\lambda+Y_{n}\right)\right)$. Then,
\begin{align*}
T_{\mathfrak{b}}\circ r_{1}^{1/\gamma}\circ\pi_{m+n,1}^{\lambda}\left(G\right) & =\tilde{U}\left(X,Y\right)X_{1}^{d}\left(\lambda+Y_{n}\right)^{d}+\sum_{i=2}^{d}\left(X'\hat{Y}\right)^{\alpha_{i}'}\tilde{U}_{i}\left(X,Y\right)X_{1}^{d-i+\frac{\alpha_{1,i}}{\gamma}}\left(\lambda+Y_{n}\right)^{d-i}\\
 & =X_{1}^{d}\left[\tilde{U}\left(X,Y\right)Y_{n}^{d}+d\lambda\tilde{U}\left(X,Y\right)Y_{n}^{d-1}+\sum_{i=2}^{d}B_{i}\left(X,\hat{Y}\right)Y_{n}^{d-i}\right]\\
 & =X_{1}^{d}G_{\lambda}\left(X,Y\right),
\end{align*}
where $\tilde{U}=\pi_{m+n,1}^{\lambda}\left(\overline{U}\right)$,
$\tilde{U}_{i}=\pi_{m+n,1}^{\lambda}\left(\overline{U_{i}}\right)$
and $B_{i}\in\mathbb{R}\left\llbracket X^{*},\hat{Y}\right\rrbracket $.
Notice that, thanks to the initial Tschirnhausen translation $\tau_{b}$,
$G_{\lambda}$ is regular in $Y_{n}$ of order at most $d-1$.

\medskip{}

Arguing in the same way for every variable of the tuple $X$ which
actually appears in the monomial $\left(X\hat{Y}\right)^{\alpha_{l}}$,
after factoring out a monomial in the variables $X$, either we monomialise
$G$ (case 1), or we reduce the order of regularity in $Y_{n}$ (case
3), or we eliminate the variables $X$ from the monomial $\left(X\hat{Y}\right)^{\alpha_{l}}$
(case 2). 

Next we perform a blow-up transformation to obtain similar results
in the variables $\hat{Y}$. If the variable $Y_{1}$ appears in the
monomial $\left(X\hat{Y}\right)^{\alpha_{l}}$, we perform the blow-up
transformation $\pi_{m+n,m+1}$. For the charts $\pi_{m+n,m+1}^{\infty}$
and $\pi_{m+n,m+1}^{\lambda}$ ($\lambda\in\mathbb{R}$), we obtain
the same result as in cases 1 and 3 respectively, namely either $G$
has become normal or, after factoring out a power of $Y_{1}$, the
order of regularity in $Y_{n}$ has decreased. For the chart $\pi_{m+n,m+1}^{0}$,
we obtain, after factoring out a power of $Y_{1}$, that the exponents
$\alpha_{i,m+1}$ of $Y_{1}$ each decrease by the quantity $i$,
and hence by repeating the process we can reduce to the case $\alpha_{l,1}=0$.
Hence, as in case 2, we have eliminated the variable $Y_{1}$ from
the monomial $\left(X\hat{Y}\right)^{\alpha_{l}}$. \medskip{}

Summing up, there is an admissible tree $\hat{T}$ such that, for
every branch $\hat{\mathfrak{b}}$ of $\hat{T}$, we have that $T_{\mathfrak{b}}\circ\hat{T}_{\hat{\mathfrak{b}}}\left(X^{\alpha}\right)$
is a monomial (in the variables $\left(X,Y\right)$) and $T_{\mathfrak{b}}\circ\hat{T}_{\hat{\mathfrak{b}}}\left(G\right)=\left(X\hat{Y}\right)^{\alpha_{\hat{\mathfrak{b}}}}V_{\hat{\mathfrak{b}}}\left(X,Y\right)G_{\hat{\mathfrak{b}}}\left(X,Y\right)$,
where $\alpha_{\hat{\mathfrak{b}}}\in[0,\infty)^{m}\times\mathbb{N}^{n-1}$,
$V_{\hat{\mathfrak{b}}}$ is a unit and $G_{\hat{\mathfrak{b}}}$
is either a monomial in $Y_{n}$ (case 1), or it is regular in $Y_{n}$
of order 1 (repeated use of cases 2 and 3 ). In this latter case,
we compose with a further Tschirnhausen translation in order to render
\foreignlanguage{english}{$G_{\hat{\mathfrak{b}}}\left(X,Y\right)$}
normal. This concludes the proof of Theorem \ref{thm: monomialisation}.
\end{proof}

\subsection{Monomialisation of germs\label{subsec:monomialisation of germs}}

Recall that $f\in\mathcal{A}_{m,n}$ is \emph{normal} if there exist
$\alpha\in\mathbb{A}^{m},\ N\in\mathbb{N}^{n}$ and $u\in\mathcal{A}_{m,n}$
such that $f\left(x,y\right)=x^{\alpha}y^{N}u\left(x,y\right)$ and
$u\left(0,0\right)\not=0$. 
\begin{rem}
\label{rem: f normal iff F normal}If $f\in\mathcal{A}_{m,n}$ then
it follows from the axioms in \ref{emp:properties of the morph} that
$f\left(x,y\right)$ is normal if and only if $\mathcal{\mathcal{T}}\left(f\right)\left(X,Y\right)$
is normal.\end{rem}
\begin{thm}
\label{thm: geom monomial}Let $f_{1},\ldots,f_{p}\in\mathcal{A}_{m,n}$.
Then there exist a polyradius $r'$ such that the $f_{j}$ have a
representative on $\hat{I}_{m,n,r'}$, and a finite family 
\[
\mathcal{F}=\left\{ \rho_{i}:\hat{I}_{m_{i},n_{i},r_{i}}\to\hat{I}_{m,n,r}\right\} _{i=1,\ldots,N}
\]
 of admissible transformations such that, for all $i=1,\ldots,N$,
the germs $f_{1}\circ\rho_{i},\ldots,f_{p}\circ\rho_{i}$ are normal
and linearly ordered by division, and 
\[
\hat{I}_{m,n,r'}\subseteq\bigcup_{i=1}^{N}\rho_{i}\left(\hat{I}_{m_{i},n_{i},r_{i}}\right).
\]
\end{thm}
\begin{rem}
Let $\nu:\hat{I}_{m',n',r'}\to\hat{I}_{m,n,r}$ be an elementary transformation.
Notice that for every polyradius $r'''\leq r'$ there exists a polyradius
$r''\leq r$ such that:
\begin{itemize}
\item if $\nu$ is either $\tau_{h}$, or $L_{i,c}$, or $r_{i}^{\gamma}$,
then $\hat{I}_{m,n,r''}\subseteq\nu\left(\hat{I}_{m,n,r'''}\right)$;
\item $\hat{I}_{m,n,r''}\subseteq r_{m+i}^{d,+}\left(\hat{I}_{m,n,r'''}\right)\cup r_{m+i}^{d,-}\left(\hat{I}_{m,n,r'''}\right)$;
\item $\hat{I}_{m,n,r''}\subseteq\bigcup_{\lambda\in\left(0,\infty\right)}\pi_{i,j}^{\lambda}\left(\hat{I}_{m-1,n+1,r'''}\right)\cup\pi_{i,j}^{0}\left(\hat{I}_{m,n,r'''}\right)\cup\pi_{i,j}^{\infty}\left(\hat{I}_{m,n,r'''}\right);$
\item $\hat{I}_{m,n,r''}\subseteq\bigcup_{\lambda\in\mathbb{R}}\pi_{m+i,j}^{\lambda}\left(\hat{I}_{m,n,r'''}\right)\cup\pi_{m+i,j}^{+\infty}\left(\hat{I}_{m+1,n-1,r'''}\right)\cup\pi_{m+i,j}^{-\infty}\left(\hat{I}_{m+1,n-1,r'''}\right);$
\item $\hat{I}_{m,n,r''}\subseteq\bigcup_{\lambda\in\mathbb{R}\cup\left\{ \infty\right\} }\pi_{m+i,m+j}^{\lambda}\left(\hat{I}_{m,n,r'''}\right).$
\end{itemize}
Moreover, by a compactness argument as in \cite[p. 4406]{vdd:speiss:gen},
it is easy to see that the three last inclusions remain true for suitable
finite families of parameters $\lambda$.\end{rem}
\begin{proof}
[Proof of Theorem \ref{thm: geom monomial}]Notice first that, using
Lemma \ref{lem: prod of series} and Remark \ref{rem: f normal iff F normal},
we only need to prove the statement for one germ $f\in\mathcal{A}_{m,n}$.
By Theorem \ref{thm: monomialisation} and by compatibility of $\mathcal{T}$
with the operations in \ref{emp:properties of the morph}, there exists
an admissible tree $T$ such that for each branch $\mathfrak{b}$
of $T$, the germ $f\circ\rho_{\mathfrak{b}}$ is normal. We prove
a stronger statement, namely that the admissible transformations in
the statement are in fact induced by branches of $T$. The proof is
by induction on the height of $T$. If the height of $T$ is zero,
then $f$ is normal. Hence, let us assume $T$ has positive height,
consider the initial vertex of $T$ and the elementary tree $T_{0}$
immediately attached to it. 

Let $\nu:\hat{I}_{m_{\nu},n_{\nu},r_{\nu}}\to\hat{I}_{m,n,r}$ be
an elementary transformation induced by a branch of $T_{0}$ (i.e.
$\nu$ is either a Tschirnhausen translation, or a linear transformation,
or a ramification, or it is a blow-up chart). The germ $g^{\nu}=f\circ\nu$
belongs to $\mathcal{A}_{m_{\nu},n_{\nu}}$ and can be monomialised
by a tree $T'$ with the property that, for every branch $\mathfrak{b}^{'}$
of $T'$, there exists a unique branch $\mathfrak{b}$ of $T$ such
that $\rho_{\mathfrak{b}}=\nu\circ\rho_{\mathfrak{b}^{'}}$. Since
the height of $T'$ is smaller than the height of $T$, the inductive
hypothesis applies and there are a polyradius $r_{\nu}'$ and a finite
family $\mathcal{F}^{\nu}=\left\{ \rho_{i}^{\nu}:\hat{I}_{m_{i}^{\nu},n_{i}^{\nu},r_{i}^{\nu}}\to\hat{I}_{m_{\nu},n_{\nu},r_{\nu}}\right\} _{i=1,\ldots,N_{\nu}}$
as in the statement of the theorem such that the germs $g^{\nu}\circ\rho_{i}^{\nu}$
are normal and $\hat{I}_{m_{\nu},n_{\nu},r_{\nu}'}\subseteq\bigcup_{i=1}^{N_{\nu}}\rho_{i}^{\nu}\left(\hat{I}_{m_{i}^{\nu},n_{i}^{\nu},r_{i}^{\nu}}\right)$.
Since the admissible transformations $\nu\circ\rho_{i}^{\nu}$ are
induced by branches of $T$, we have that the germs $f\circ\nu\circ\rho_{i}^{\nu}$
are normal and $\nu\left(\hat{I}_{m_{\nu},n_{\nu},r_{\nu}'}\right)\subseteq\bigcup_{i=1}^{N_{\nu}}\nu\circ\rho_{i}^{\nu}\left(\hat{I}_{m_{i}^{\nu},n_{i}^{\nu},r_{i}^{\nu}}\right)$.

Now, by Remark \ref{rem: elemntary transf send quadrants to sectors},
there are a polyradius $r'$ and a finite collection $\mathcal{G}$
of branches of $T_{0}$ such that $\hat{I}_{m,n,r'}\subseteq{\displaystyle \bigcup_{\nu\in\mathcal{G}}}\nu\left(\hat{I}_{m_{\nu},n_{\nu},r_{\nu}'}\right)$,
and we are done.
\end{proof}

\section{A parametrisation theorem\label{sec:Parametrisation-of--subanalytic-1}}

The main result of this section is Theorem \ref{thm: param subanal},
which essentially provides a way to parametrise every bounded $\mathbb{R}_{\mathcal{A}}$-definable
set by means of maps whose components are in $\mathcal{A}$. This
provides the main step in the proof of Theorem A, which will be completed
at the end of this section, and a key tool for proving Theorem B,
in the next section. 
\begin{defn}
\label{def:semi and subanalytic sets}

A set $A\subset\hat{I}_{m,n,r}$ is said to be \emph{$\mathcal{A}_{m,n}$-basic}
if it is a finite union of sets of the form 
\[
\{(x,y)\in\hat{I}_{m,n,r}:\ g_{0}(x,y)=0,g_{1}(x,y)>0,\dots,g_{k}(x,y)>0\},
\]
 where $g_{0},\ldots,g_{k}\in\mathcal{A}_{m,n,r}$.

A set $A\subset\mathbb{R}^{m+n}$ is said to be \emph{$\mathcal{A}_{m,n}$-semianalytic}
if for every point $a\in\mathbb{R}^{m+n}$ there exists $r_{a}\in\zeroapinf^{m+n}$
such that for every choice of signs $\sigma=(\sigma_{1},\dots,\sigma_{m})\in\left\{ -1,1\right\} ^{m}$,
there exists an $\mathcal{A}_{m,n}$-basic set $A_{a,\sigma}\subset\hat{I}_{m,n,r_{a}}$
with 
\[
A\cap(h_{a,\sigma}(\hat{I}_{m,n,r_{a}}))=h_{a,\sigma}(A_{a,\sigma}),
\]
 where $h_{a,\sigma}(x,y)=(\sigma_{1}x_{1}+a_{1},\dots,\sigma_{m}x_{m}+a_{m},y_{1}+a_{m+1},\ldots,y_{n}+a_{m+n})$.

We will simply say that a set is $\mathcal{A}$-basic or $\mathcal{A}$-semianalytic
when the indices $m,n$ are either obvious from the context or irrelevant. \end{defn}
\begin{rem}
\label{rem: basic is semi}Notice that since all functions in $\mathcal{A}$
are $\mathcal{A}$-analytic (see Definition \ref{def: A-analytic}),
an $\mathcal{A}$-basic set is also $\mathcal{A}$-semianalytic.\end{rem}
\begin{defn}
\label{def:quadrant}Let $r=\left(r_{1},\ldots r_{m+n}\right)\in\left(0,\infty\right)^{m+n}$
be a polyradius. A set $Q\subseteq\hat{I}_{m,n,r}$ of the form $B_{1}\times\ldots\times B_{m}$,
where $B_{i}$ is either $\left\{ 0\right\} $, or $\left(-r_{i},0\right)$,
or $\left(0,r_{i}\right)$, is called a \emph{sub-quadrant} of $\mathbb{R}^{m+n}$.
The cardinality of the set $C:=\left\{ i:\ B_{i}\not=\left\{ 0\right\} \right\} $,
denoted by $\mathrm{dim}\left(Q\right)$, is called the \emph{dimension
}of $Q$. Notice that $\hat{I}_{m,n,r}$ is a finite union of sub-quadrants. 
\end{defn}

The following proposition states that the germ at zero of an $\mathcal{A}$-basic
subset of $\mathbb{R}^{m+n}$ can be transformed into a finite union
of sub-quadrants of $\mathbb{R}^{m+n}$ by means of a finite family
of admissible transformations.
\begin{prop}
\label{prop: param basic sets}Let $A\subseteq\hat{I}_{m,n,r}$ be
an $\mathcal{A}_{m,n}$-basic set. Then there exist a neighbourhood
$W$ of 0 in $\mathbb{R}^{m+n}$ and a finite family $\mathcal{F}=\left\{ \left(\rho_{i},Q_{i}\right):\ i=1,\ldots,N\right\} $,
where $\rho_{i}:\hat{I}_{m_{i}',n_{i}',r'_{i}}\to\hat{I}_{m,n,r}$
is an admissible transformation and $Q_{i}\subseteq\overline{Q_{i}}\subseteq\hat{I}_{m_{i}',n_{i}',r_{i}'}$
is a sub-quadrant, such that $\rho_{i}\restriction Q_{i}:\ Q_{i}\to\rho_{i}\left(Q_{i}\right)$
is a diffeomorphism and 
\[
W\cap A=\bigcup_{i=1}^{N}\rho_{i}\left(Q_{i}\right).
\]
\end{prop}
\begin{proof}
The proof is by induction on the dimension $m+n$ of the ambient space.
Let $S:=\left\{ x_{1}=0\right\} \cup\ldots\cup\left\{ x_{m}=0\right\} \cup\left\{ y_{1}=0\right\} \cup\ldots\cup\left\{ y_{n}=0\right\} $.
The set $A\cap S$ is an $\mathcal{A}$-basic sets contained in an
ambient space of dimension strictly smaller than $m+n$, hence, by
the inductive hypothesis, it is sufficient to prove the proposition
for the set $A':=A\setminus S$. We may assume that $A'$ is of the
following form: 
\[
A'=\left\{ \left(x,y\right)\in\hat{I}_{m,n,r}:\ f_{0}\left(x,y\right)=0\wedge f_{1}\left(x,y\right)>0\wedge\ldots\wedge f_{p}\left(x,y\right)>0\right\} ,
\]
where $f_{i}\in\mathcal{A}_{m,n}$. 

By Theorem \ref{thm: monomialisation}, there is an admissible tree
$T$ such that for each branch $\mathfrak{b}$ of $T$, the germs
$f_{i}\circ\rho_{\mathfrak{b}}$ are all normal. We prove by induction
on the height of $T$ that the admissible transformations which parametrise
$A'$ are in fact induced by branches of $T$. If the height of $T$
is zero, then all the $f_{i}$ are normal and $A'$ is a sub-quadrant
of $\hat{I}_{m,n,r}$, so for some box around zero $W$, the closure
of the sub-quadrant $W\cap A'$ is contained in $\hat{I}_{m,n,r}$.
Hence, let us assume $T$ has positive height, consider the initial
vertex of $T$ and the elementary tree $T_{0}$ immediately attached
to it. 

Let $\nu:\hat{I}_{m_{\nu},n_{\nu},r_{\nu}}\to\hat{I}_{m,n,r}$ be
an elementary transformation induced by a branch of $T_{0}$. Consider
the germs $g_{i}^{\nu}=f_{i}\circ\nu\in$$\mathcal{A}_{m_{\nu},n_{\nu}}$.
The set

\[
A_{\nu}=\left\{ \left(x,y\right)\in\hat{I}_{m_{\nu},n_{\nu},r_{\nu}}:\ g_{0}^{\nu}\left(x,y\right)=0\wedge g_{1}^{\nu}\left(x,y\right)>0\wedge\ldots\wedge g_{p}^{\nu}\left(x,y\right)>0\right\} 
\]
is $\mathcal{A}_{m_{\nu},n_{\nu}}$-basic and the germs $g_{i}^{\nu}$
can be monomialised by a tree $T'$ with the property that, for every
branch $\mathfrak{b}^{'}$ of $T'$, there exists a unique branch
$\mathfrak{b}$ of $T$ such that $\rho_{\mathfrak{b}}=\nu\circ\rho_{\mathfrak{b}^{'}}$.
Since the height of $T'$ is smaller than the height of $T$, the
inductive hypothesis applies and there are a neighbourhood $W_{\nu}$
of 0 in $\mathbb{R}^{m+n}$ and a finite family $\mathcal{F}_{\nu}=\left\{ \left(\rho_{i}^{\nu},Q_{i}^{\nu}\right):\ i=1,\ldots,N_{\nu}\right\} $
as in the statement such that the $\rho_{i}^{\nu}$ are induced by
branches of $T'$ and $W_{\nu}\cap A_{\nu}=\bigcup_{i=1}^{N_{\nu}}\rho_{i}^{\nu}\left(Q_{i}^{\nu}\right)$.
By Remark \ref{rem: elemntary transf send quadrants to sectors},
there are a polyradius $r'$ and a finite collection $\mathcal{G}$
of branches of $T_{0}$ such that $\hat{I}_{m,n,r'}\subseteq{\displaystyle \bigcup_{\nu\in\mathcal{G}}}\nu\left(\hat{I}_{m_{\nu},n_{\nu},r_{\nu}}\cap W_{\nu}\right)\subseteq\hat{I}_{m,n,r}$.
Hence, there is a neighbourhood $W$ of zero such that ${\displaystyle \bigcup_{\nu\in\mathcal{G}}}\nu\left(\hat{I}_{m_{\nu},n_{\nu},r_{\nu}}\cap W_{\nu}\right)=\hat{I}_{m,n,r}\cap W$.
It follows that 

\[
A'\cap W=A'\cap{\displaystyle \bigcup_{\nu\in\mathcal{G}}}\nu\left(\hat{I}_{m_{\nu},n_{\nu},r_{\nu}}\cap W_{\nu}\right)={\displaystyle \bigcup_{\nu\in\mathcal{G}}}\nu\left(A_{\nu}\cap W_{\nu}\right)={\displaystyle \bigcup_{\nu\in\mathcal{G}}}\bigcup_{i=1}^{N_{\nu}}\nu\circ\rho_{i}^{\nu}\left(Q_{i}^{\nu}\right).
\]
We claim that $\nu\restriction A_{\nu}$ is a diffeomorphism onto
its image (and hence so is $\nu\circ\rho_{i}^{\nu}\restriction Q_{i}^{\nu}$).
This is clear if $\nu$ is either of type $L_{i,c},\ \text{or}\ \tau_{h}$,
or $r_{m+i}^{d,\pm}$. If $\nu=r_{i}^{\gamma}$, then $A_{\nu}\cap\left\{ x_{i}'=0\right\} =\emptyset$
(because $A'\cap S=\emptyset$). For the same reason, if $\nu=\pi_{m+i,j}^{\lambda}$
then $A_{\nu}\cap\left\{ x_{j}'=0\right\} =\emptyset$, and similarly
for the other blow-up charts. Summing up, $A_{\nu}$ does not meet
the subset of $\hat{I}_{m_{\nu},n_{\nu},r_{\nu}}$ on which $\nu$
is either not bijective or not differentiable. \end{proof}
\begin{cor}
\label{cor: param semi-anal}Let $A\subseteq\mathbb{R}^{m+n}$ be
a bounded $\mathcal{A}_{m,n}$-semianalytic set. Then there exists
a finite family $\mathcal{F}=\left\{ \left(\rho_{i},Q_{i},a_{i},\sigma_{i}\right):\ i=1,\ldots,N\right\} $,
where $a_{i}\in\mathbb{R}^{m+n}$, $\sigma_{i}\in\left\{ -1,1\right\} ^{m}$,
$\rho_{i}:\hat{I}_{m_{i}',n_{i}',r'_{i}}\to\hat{I}_{m,n,r}$ is an
admissible transformation and $Q_{i}\subseteq\hat{I}_{m_{i}',n_{i}',r_{i}'}$
is a sub-quadrant, such that $h_{a_{i},\sigma_{i}}\circ\rho_{i}\restriction Q_{i}:\ Q_{i}\to h_{a_{i},\sigma_{i}}\circ\rho_{i}\left(Q_{i}\right)$
is a diffeomorphism and 
\[
A=\bigcup_{i=1}^{N}h_{a_{i},\sigma_{i}}\circ\rho_{i}\left(Q_{i}\right),
\]
where $h_{a_{i},\sigma_{i}}$ is as in Definition \ref{def:semi and subanalytic sets}. \end{cor}
\begin{rem}
Since $h_{a_{i},\sigma_{i}}\circ\rho_{i}\restriction Q_{i}$ is a
diffeomorphism onto its image, it follows that $A$ has dimension
(in the sense of \cite[8.2]{vdd:speiss:gen}). Notice also that the
components of $h_{a_{i},\sigma_{i}}\circ\rho_{i}$ belong to $\mathcal{A}_{m_{i}',n_{i}',r_{i}'}$.\end{rem}
\begin{proof}
The set $A$ is a finite union of translates of reflections of $\mathcal{A}_{m,n}$-basic
sets, by a compactness argument. Hence the parametrisation required
is obtained from the parametrisation in Proposition \ref{prop: param basic sets}
composed with suitable translations and reflections.
\end{proof}
The next few definitions and statements are inspired by the approach
to model-completeness developed in \cite{bm_semi_subanalytic}, \cite{vdd:speiss:gen}
and \cite{rsw}. The proofs follow the usual pattern, with some minor
differences dictated by the generality of our setting. It is however
worth noticing that in order to prove the Fibre Cutting Lemma \ref{empty: Fibre Cutting},
we need to use both quasianalyticity (see Definition \ref{def: quasi-analyticity})
and $\mathcal{A}$-analyticity (see Definition \ref{def: A-analytic}).
\begin{defn}
\label{def: trivial manifolds}Let $F=\left(F_{1},\ldots,F_{m+n}\right):\hat{I}_{m,n,r}\to I_{0,m+n,r'}$
be a map such that $F_{i}\in\mathcal{A}_{m,n,r}$, and let $Q\subseteq\overline{Q}\subseteq\hat{I}_{m,n,r}$
be a sub-quadrant of dimension $q\leq m+n$. The set 
\[
\Gamma\left(F\restriction Q\right)=\left\{ \left(u,w\right):\ u\in Q,\ w=F\left(u\right)\right\} \subseteq I_{m,m+2n,\left(r,r'\right)}
\]
 is called a \emph{trivial manifold. }

A trivial manifold $M$ is clearly a $\mathcal{C}^{1}$ manifold of
dimension $q$ and an $\mathcal{A}$-basic set. The frontier $\text{fr}\left(M\right)=\overline{M}\setminus M$
is the set $\left\{ \left(u,w\right):\ u\in\text{fr}\left(Q\right),\ w=F\left(u\right)\right\} $,
which, by $\mathcal{A}$-analyticity of the components of $F$ (see
Definition \ref{def: A-analytic}), is an $\mathcal{A}$-semianalytic
set. Moreover, $\text{dim}\left(\text{fr}\left(M\right)\right)<\text{dim}\left(M\right)$. \end{defn}
\begin{notation}
Let $k,l\in\mathbb{N}$ with $k\leq l$. Denote by $\Pi_{k}^{l}:\mathbb{R}^{l}\to\mathbb{R}^{k}$
the projection onto the \textbf{last} $k$ coordinates, i.e. $\Pi_{k}^{l}\left(z_{1},\ldots,z_{l}\right)=\left(z_{l-k+1},\ldots,z_{l}\right)$. \end{notation}
\begin{cor}
\label{cor: decomp of s.a. into trivial manifolds}Let $A\subseteq\mathbb{R}^{m+n}$
be a bounded $\mathcal{A}_{m,n}$-semianalytic set. Then there exist
finitely many trivial manifolds $M_{1},\ldots,M_{N}\subseteq\mathbb{R}^{2\left(m+n\right)}$
such that $A={\displaystyle \bigcup_{i=1}^{N}\Pi_{m+n}^{2\left(m+n\right)}\left(M_{i}\right)}$
and $\Pi_{m+n}^{2\left(m+n\right)}\restriction M_{i}$ is an immersion.\end{cor}
\begin{proof}
Apply Corollary \ref{cor: param semi-anal} to $A$ and let $M_{i}=\Gamma\left(h_{a_{i},\sigma_{i}}\circ\rho_{i}\restriction Q_{i}\right)$.
Notice that $\text{dim}\left(M_{i}\right)=\text{dim}\left(Q_{i}\right)\leq\text{dim}\left(A\right)$.\end{proof}
\begin{defn}
\label{def: A-manifold}An $\mathcal{A}_{m,n}$-basic set $M\subseteq I_{m,n,r}$
is called an \emph{$\mathcal{A}_{m,n}$-manifold }if $M$ is a $\mathcal{C}^{1}$-manifold
of dimension $d\leq m+n$ and there are $h_{1},\ldots,h_{m+n-d}\in\mathcal{A}_{m,n}$
such that $h_{i}\restriction M\equiv0$ and for all $z\in M$ the
vectors $\nabla h_{1}\left(z\right),\ldots,\nabla h_{m+n-d}\left(z\right)$
are linearly independent (see \cite[Def. 8.3]{vdd:speiss:gen}). Notice
that a trivial manifold is an $\mathcal{A}_{m,n}$-manifold.

Let $\iota:\left\{ 1,\ldots,d\right\} \to\left\{ 1,\ldots,m+n\right\} $
be a strictly increasing sequence and let $\Pi_{\iota}\left(z_{1},\ldots,z_{m+n}\right)=\left(z_{\iota\left(1\right)},\ldots,z_{\iota\left(d\right)}\right)$.
Let $M_{\iota}=\left\{ z\in M:\ \Pi_{\iota}\restriction M\ \text{has\ rank\ }d\ \text{at\ }z\right\} $.
Then $M_{\iota}=\left\{ z\in M:\ f_{\iota}\left(z\right)\not=0\right\} $
for some $f_{\iota}\in\mathcal{A}_{m,n,r}$, which is a polynomial
in the partial derivatives and modified derivatives $\partial_{i}$
of $h_{1},\ldots,h_{m+n-d}$. In particular, $M_{\iota}$ is an $\mathcal{A}_{m,n}$-manifold
of dimension $d$ and $M$ is the union of all the $M_{\iota}$.

Let $k\leq m+n$ and for every increasing sequence $\iota$, define
$m_{\iota}\left(k\right)$ as the dimension of the vector space $\Pi_{\iota}\left(\mathbb{R}^{m+n}\right)\cap\Pi_{k}^{m+n}\left(\mathbb{R}^{m+n}\right)$.
Notice that $m_{\iota}\left(k\right)\in\left\{ 0,\ldots,d\right\} $.
Suppose that $\Pi_{k}^{m+n}\restriction M_{\iota}$ has constant rank
$m_{\iota}\left(k\right)$. Then, arguing as in \cite[8.11]{vdd:speiss:gen},
it is easy to see that for every $a\in\mathbb{R}^{k}$, the fibre
$M_{a}:=\left(\Pi_{k}^{m+n}\right)^{-1}\left(a\right)\cap M_{\iota}$
is either empty or an $\mathcal{A}_{m,n}$-manifold of dimension $d-m_{\iota}\left(k\right)$
and every connected component of the fibre has non-empty frontier. \end{defn}
\begin{void}
\label{empty: Fibre Cutting}\textbf{Fibre Cutting Lemma. }\emph{Let
$M_{\iota}\subseteq I_{m,n,r}$ be as in the above definition and
let $k\leq m+n$. Suppose that $\Pi_{k}^{m+n}\restriction M_{\iota}$
has constant rank $m_{\iota}\left(k\right)$ and that $m_{\iota}\left(k\right)<d=\text{dim}\left(M_{\iota}\right)$.
Then there is an $\mathcal{A}_{m,n,r}$-basic set $A\subseteq M_{\iota}$
such that $\text{dim}\left(A\right)<d$ and $\Pi_{k}^{m+n}\left(A\right)=\Pi_{k}^{m+n}\left(M_{\iota}\right)$.}\end{void}
\begin{proof}
Let $f,g_{1},\ldots,g_{q}\in\mathcal{A}_{m,n,r}$ be such that

\[
M_{\iota}=\left\{ \left(x,y\right)\in I_{m,n,r}:\ f\left(x,y\right)=0,\ g_{1}\left(x,y\right)=0,\ldots,g_{q}\left(x,y\right)=0\right\} .
\]
Recall that for every $a\in\Pi_{k}^{m+n}\left(M_{\iota}\right)$,
the fibre $M_{a}:=\left(\Pi_{k}^{m+n}\right)^{-1}\left(a\right)\cap M_{\iota}$
is an $\mathcal{A}_{m,n}$-manifold of dimension $d-m_{\iota}\left(k\right)>0$
and every connected component of the fibre has non-empty frontier.
Let $r=\left(s,t\right)$. The function $g\left(x,y\right)=\prod_{i=1}^{m}x_{i}\left(s_{i}-x_{i}\right)\cdot\prod_{i=1}^{n}\left(t_{i}^{2}-y_{i}^{2}\right)\cdot\prod_{i=1}^{q}g_{i}\left(x,y\right)\in\mathcal{A}_{m,n,r}$
is positive on $M_{a}$ and vanishes on $\text{fr}\left(M_{a}\right)$,
hence $g\restriction M_{a}$ has a critical point on every connected
component of $M_{a}$. The set

\[
A=\left\{ \left(x,y\right)\in M_{\iota}:\ g\restriction M_{a}\text{\ is\ critical\ at\ \ensuremath{\left(x,y\right)},\ where }a:=\Pi_{k}^{m+n}\left(x,y\right)\right\} 
\]
is clearly $\mathcal{A}_{m,n,r}$-basic and $\Pi_{k}^{m+n}\left(A\right)=\Pi_{k}^{m+n}\left(M_{\iota}\right)$.
We claim that $A$ has empty interior in $M_{\iota}$, and hence $\text{dim}\left(A\right)<d$.

Let $a\in\Pi_{k}^{m+n}\left(M_{\iota}\right)$ and let $h_{1},\ldots,h_{m+n-d}\in\mathcal{A}_{m,n}$
such that $h_{i}\restriction M_{a}\equiv0$ and for all $z=\left(x,y\right)\in M_{a}$
the vectors $\nabla h_{1}\left(z\right),\ldots,\nabla h_{m+n-d}\left(z\right)$
are linearly independent. If $C$ is a connected component of $M_{a}$,
let 

\[
S=\left\{ z\in C:\ g\restriction M_{a}\text{\ is\ critical\ at }z\right\} =\left\{ z\in C:\ \text{d}g\left(z\right)\land\text{d}h_{1}\left(z\right)\land\ldots\land\text{d}h_{m+n-d}\left(z\right)=0\right\} .
\]
It suffices to prove that $S$ has empty interior on $C$. Clearly,
$S$ is closed in $C$. We prove that, if $S$ has non-empty interior
in $C$, then $S$ is also open in $C$, which contradicts the fact
that $g$ is not constant on the fibres. Let $z_{0}\in S\setminus\overset{\circ}{S}$
and define $\overline{g}\left(z\right):=g\left(z+z_{0}\right),\ \overline{h}_{i}\left(z\right):=h_{i}\left(z+z_{0}\right)$
for $i=1,\ldots,m+n-d$. By $\mathcal{A}$-analyticity (see Definition
\ref{def: A-analytic}), the germs at zero of these functions belong
to $\mathcal{A}_{0,m+n}$. By the Implicit Function Theorem and Axiom
6 in \ref{emp:properties of the morph}, there are a polyradius $r'\in\mathbb{R}^{d}$
and a map $\varphi=\left(\varphi_{1},\ldots,\varphi_{m+n}\right)\in\left(\mathcal{A}_{0,d,r'}\right)$
such that $\varphi\left(0\right)=z_{0}$ and $\varphi$ is a local
parametrisation of $C$ around $z_{0}$. Let $F\in\mathcal{A}_{0,d,r'}$
be such that 

\begin{align*}
\varphi^{-1}\left(S\right) & =\left\{ w\in I_{0,d,r'}:\ \text{d}\left(\overline{g}\circ\varphi\right)\left(w\right)\land\text{d}\left(\overline{h}_{1}\circ\varphi\right)\left(w\right)\land\ldots\land\text{d}\left(\overline{h}_{m+n-d}\circ\varphi\right)\left(w\right)=0\right\} \\
 & =\left\{ w\in I_{0,d,r'}:\ F\left(w\right)=0\right\} .
\end{align*}
Since $\varphi$ is a homeomorphism, $\varphi^{-1}\left(\overset{\circ}{S}\right)\subseteq I_{0,d,r'}$
is an open set on which $F$ vanishes, and $0\in\text{cl}\left(\varphi^{-1}\left(\overset{\circ}{S}\right)\right)$
. By quasianalyticity, there exists $r''\leq r'$ such that $F$ vanishes
on $I_{0,d,r''}$, which contradicts the fact that $z_{0}\notin\overset{\circ}{S}$. \end{proof}
\begin{prop}
\label{prop: trivial manifolds for subanal sets}Let $A\subseteq\mathbb{R}^{m+n}$
be a bounded $\mathcal{A}_{m,n}$-semianalytic set and $k\leq m+n$.
Then there exist finitely many trivial manifolds $M_{1},\ldots,M_{N}\subseteq\mathbb{R}^{2\left(m+n\right)}$,
with $\text{dim}\left(M_{i}\right)\leq k$, such that $\Pi_{k}^{m+n}\left(A\right)={\displaystyle \bigcup_{i=1}^{N}\Pi_{k}^{2\left(m+n\right)}\left(M_{i}\right)}$
and for every $i=1,\ldots,N$ there is a strictly increasing sequence
$\iota:\left\{ 1,\ldots,\text{dim}\left(M_{i}\right)\right\} \to\left\{ m+n-k+1,\ldots,m+n\right\} $
such that $\Pi_{\iota}\restriction M_{i}$ is an immersion.\end{prop}
\begin{proof}
The proof is by induction on $d=\text{dim}\left(A\right)$, for all
$m,n,k\in\mathbb{N}$. If $d=0$, then by Corollary \ref{cor: param semi-anal}
$A$ is a finite set and there is nothing to prove. Let $d>0$. By
Corollary \ref{cor: decomp of s.a. into trivial manifolds} and the
inductive hypothesis, it is enough to prove the proposition for an
$\mathcal{A}_{m,n}$-manifold $M$ of dimension $d$, instead of $A$.
We argue by induction on $r=\text{max}\left\{ \mbox{\ensuremath{\text{rk}_{z}\Pi_{k}^{m+n}\restriction}M}:\ z\in M\right\} $.
If $r=0$, then $\Pi_{k}^{m+n}\left(M\right)$ is a finite set and
there is nothing to prove, so let us assume that $r>0$. Let $M_{1}=\left\{ z\in M:\ \text{rk}_{z}\Pi_{k}^{m+n}\restriction M=r\right\} $
and $M_{2}=\left\{ z\in M:\ \text{rk}_{z}\Pi_{k}^{m+n}\restriction M<r\right\} $.
Then $M_{1}$ and $M_{2}$ are $\mathcal{A}_{m,n}$-semianalytic and
$M_{1}$ is open in $M$, and hence is an $\mathcal{A}_{m,n}$-manifold.
By Corollary \ref{cor: decomp of s.a. into trivial manifolds} and
the inductive hypothesis on the rank, we obtain the statement of the
proposition for $M_{2}$. Hence we may assume that $\Pi_{k}^{m+n}\restriction M$
has constant rank $r$. 

Suppose first that $r=d$. Notice that in this case $d\leq k$ and
$M$ is the union of all $\mathcal{A}_{m,n}$-manifolds $M_{\iota}$
(as in Definition \ref{def: A-manifold}) such that $\iota:\left\{ 1,\ldots,d\right\} \to\left\{ m+n-k+1,\ldots,m+n\right\} $
is a strictly increasing sequence. By Corollary \ref{cor: decomp of s.a. into trivial manifolds}
and the inductive hypothesis on the dimension, the proposition holds
then trivially for $M_{\iota}$. 

Hence we may assume that $r<d$. For $\iota:\left\{ 1,\ldots,d\right\} \to\left\{ 1,\ldots,m+n\right\} $
an increasing sequence, consider the $\mathcal{A}_{m,n}$-manifold
$M_{\iota}$. If $r=m_{\iota}\left(k\right)$, then by the Fibre Cutting
Lemma \ref{empty: Fibre Cutting}, Corollary \ref{cor: decomp of s.a. into trivial manifolds}
and the inductive hypothesis on the dimension, the proposition holds
for $M_{\iota}$. Now, one can check that for every $z\in M_{3}:=M\setminus\bigcup\left\{ M_{\iota}:\ m_{\iota}\left(k\right)=r\right\} $
we have $\text{rk}_{z}\Pi_{k}^{m+n}\restriction M<r$, hence $M_{3}=\emptyset$
and we are done.
\end{proof}
We are now ready to prove the parametrisation result which was announced
at the beginning of this section. In Subsection \ref{subsec:Proof-of-o-minimality}
we will show that every bounded $\mathbb{R}_{\mathcal{A}}$-definable
set is a projection of an $\mathcal{A}$-semianalytic set, hence the
result below provides a parametrisation theorem for all bounded $\mathbb{R}_{\mathcal{A}}$-definable
sets.
\begin{thm}
\label{thm: param subanal}Let $A\subseteq\mathbb{R}^{m+n}$ be a
bounded $\mathcal{A}_{m,n}$-semianalytic set and let $k\leq m+n$.
Then there exists $N\in\mathbb{N}$ and for all $i=1,\ldots,N$, there
exist $m_{i}',n_{i}'\in\mathbb{N}$, with $m_{i}'+n_{i}'=m+n$, a
polyradius $r_{i}$, a sub-quadrant $Q_{i}\subseteq\hat{I}_{m_{i}',n_{i}',r_{i}}$
and a map $H_{i}:\hat{I}_{m_{i}',n_{i}',r_{i}}\to\mathbb{R}^{k}$,
whose components are in $\mathcal{A}_{m_{i}',n_{i}',r_{i}}$, such
that $H_{i}\restriction Q_{i}:Q_{i}\to H_{i}\left(Q_{i}\right)$ is
a diffeomorphism and 
\[
\Pi_{k}^{m+n}\left(A\right)=\bigcup_{i=1}^{N}H_{i}\left(Q_{i}\right).
\]
\end{thm}
\begin{proof}
By Proposition \ref{prop: trivial manifolds for subanal sets}, there
are trivial manifolds $M_{i}\subseteq\mathbb{R}^{2\left(m+n\right)}$,
with $\text{dim}\left(M_{i}\right)\leq k$, such that $\Pi_{k}^{m+n}\left(A\right)={\displaystyle \bigcup_{i=1}^{N}\Pi_{k}^{2\left(m+n\right)}\left(M_{i}\right)}$
and $\Pi_{k}^{2\left(m+n\right)}\restriction M_{i}$ is an immersion.
By definition of trivial manifold, there exist maps $F_{i}:\hat{I}_{m_{i}',n_{i}',r_{i}}\to I_{0,m+n,r_{i}'}$
with components in $\mathcal{A}_{m_{i}',n_{i}',r_{i}}$, and sub-quadrants
$Q_{i}\subseteq\overline{Q}_{i}\subseteq\hat{I}_{m_{i}',n_{i}',r_{i}}$
of dimension $q_{i}\leq k$ such that $M_{i}=\Gamma\left(F_{i}\restriction Q_{i}\right)$.
Let $G_{i}:\hat{I}_{m_{i}',n_{i}',r_{i}}\to I_{0,2\left(m+n\right),\left(r_{i},r_{i}'\right)}$
be the map defined as $G_{i}\left(z\right)=\left(z,F_{i}\left(z\right)\right)$,
whose components are in $\mathcal{A}_{m_{i}',n_{i}',r_{i}}$. Then
the maps $H_{i}:=\Pi_{k}^{2\left(m+n\right)}\circ G_{i}$ and the
quadrants $Q_{i}$ satisfy the required properties. 
\end{proof}

\subsection{Proof of Theorem A\label{subsec:Proof-of-o-minimality}}

We use Gabrielov's approach, as illustrated in \cite[2.1-2.9]{vdd:speiss:gen}. 
\begin{defn}
\label{def: sub-lambda-set}Le $\Lambda_{n}$ be the collection of
all $\mathcal{A}$-semianalytic subsets of the unit cube $\left[-1,1\right]^{n}$.
Following \cite[Definition 2.1]{vdd:speiss:gen}, we say that a set
$B\subseteq\mathbb{R}^{n}$ is a \emph{sub-$\Lambda$-set} if there
exist $m\geq n$ and a set $A\in\Lambda_{n}$ such that $B=\Pi_{n}^{m}\left(A\right)$. 
\end{defn}
The collection $\Lambda=\left(\Lambda_{n}\right)_{n\in\mathbb{N}}$
clearly satisfies conditions (I),(II) and (III) of \cite[2.3]{vdd:speiss:gen}.
Condition (IV), namely the $\Lambda$-Gabrielov property, is also
satisfied by Proposition \ref{prop: trivial manifolds for subanal sets},
since, up to composing with a homothety, we may assume that the trivial
manifolds in the proposition are contained in the unit cube. It follows
that $\Lambda$ satisfies the Theorem of the Complement \cite[2.7]{vdd:speiss:gen}
and \cite[Corollary 2.8]{vdd:speiss:gen}. 
\begin{defn}
\label{def: decomposition of Rn}Let $n\in\mathbb{N}$ and $x=\left(x_{1},\ldots,x_{n}\right)$.
For $u\in\mathbb{R}$, define

\begin{align*}
\tau\left(u\right) & =\begin{cases}
u & \text{if}\ |u|\leq1\\
1/u & \text{if}\ |u|>1
\end{cases}.
\end{align*}
Let $\iota\subseteq\left\{ 1,\ldots,n\right\} $. Define:

\[
R_{\iota}^{n}:=\left\{ x\in\mathbb{R}^{n}:\ |x_{i}|\leq1\ \text{for\ }i\in\iota,\ |x_{i}|>1\ \text{for}\ i\notin\iota\right\} ,
\]

\[
I_{\iota}^{n}:=\left\{ x\in\mathbb{R}^{n}:\ |x_{i}|\leq1\ \text{for\ }i\in\iota,\ |x_{i}|<1\ \text{for}\ i\notin\iota\right\} ,
\]
\[
\xyC{0mm}\xyL{0mm}\xymatrix{\tau_{\iota}^{n}\colon & R_{\iota}^{n}\ar[rrrr] & \  & \  & \  & I_{\iota}^{n}\\
 & x\ar@{|->}[rrrr] &  &  &  & \left(\tau\left(x_{1}\right),\ldots,\tau\left(x_{n}\right)\right)
}
.
\]
\end{defn}
\begin{rem}
\label{rem: tau i a bijection}Notice that $\tau_{\iota}^{n}:R_{\iota}^{n}\to I_{\iota}^{n}$
is a bijection, hence it commutes with the boolean set operations.
In particular, for $A\subseteq R_{\iota}^{n}$, we have $\tau_{\iota}^{n}\left(R_{\iota}^{n}\setminus A\right)=I_{\iota}^{n}\setminus\tau_{\iota}^{n}\left(A\right)$.
Let $\kappa\subseteq\left\{ 1,\ldots,n\right\} $ and $\pi$ be the
projection of $\mathbb{R}^{n}$ onto the vector subspace spanned by
the coordinates $\left\{ x_{i}:\ i\in\kappa\right\} $. Let $\iota^{*}=\iota\cap\kappa$.
Since $\tau_{\iota}^{n}$ acts coordinate-wise, we have that $\tau_{\iota^{*}}^{|\kappa|}\left(\pi\left(A\right)\right)=\pi\left(\tau_{\iota}^{n}\left(A\right)\right)$.
Finally, note that the sets $R_{\iota}^{n}$ partition $\mathbb{R}^{n}$
and that the union of all the sets $I_{\iota}^{n}$ is the closed
unit cube.\end{rem}
\begin{void}
\label{vuoto: model comp and o-min}We now prove the model-completeness
and o-minimality of $\mathbb{R}_{\mathcal{A}}$.
\end{void}
Let $\tilde{f_{1}},\ldots,\tilde{f_{k}}$ be as in Definition \ref{def:structures},
i.e. $\tilde{f_{i}}$ coincides with some $f_{i}\in\mathcal{A}_{n',n'',1}$
(with $n'+n''=n$) inside the unit cube, and is zero outside the cube.
Let $P\left(x,y\right)\in\mathbb{Q}\left[x,y\right]$ be a polynomial
in $n+k$ variables. It is easy to see that, if $A=\left\{ x\in R_{\iota}^{n}:\ P\left(x,\tilde{f_{1}}\left(x\right),\ldots,\tilde{f_{k}}\left(x\right)=0\right)\right\} $,
then the set $\tau_{\iota}^{n}\left(A\right)$ is an $\mathcal{A}$-basic
set (this follows from the fact that the unbounded variables appear
only semi-algebraically in the formula defining $A$). Now, routine
manipulations show that every set $A\subseteq\mathbb{R}^{n}$ definable
in the structure $\mathbb{R}_{\mathcal{A}}$ is of the form 
\[
A=\bigcup_{\iota\subseteq\left\{ 1,\ldots,n\right\} }\bigcup_{\iota^{*}\subseteq\left\{ 1,\ldots,m\right\} }A_{\iota,\iota^{*}}\text{,\ with}
\]
\[
A_{\iota,\iota^{*}}=\left\{ x\in R_{\iota}^{n}:\ \mathcal{Q}_{1}y_{1}\ldots\mathcal{Q}_{m}y_{m}\ y\in R_{\iota^{*}}^{m}\ \land\ P\left(x,y,\tilde{f_{1}}\left(x,y\right),\ldots,\tilde{f_{k}}\left(x,y\right)\right)=0\right\} ,
\]
where $y=\left(y_{1},\ldots,y_{m}\right)$, $\mathcal{Q}_{i}$ is
either the existential or the universal quantifier, $P$ is a polynomial
in $n+m+k$ variables and $\tilde{f_{i}}$ are as in Definition \ref{def:structures}.
Using Remark \ref{rem: tau i a bijection}, we immediately see that
$\tau_{\iota}^{n}\left(A_{\iota,\iota^{*}}\right)$ is definable in
the structure $\left(I,\Lambda\right)$ (in particular it is a sub-$\Lambda$-set,
by \cite[Corollary 2.8]{vdd:speiss:gen}) and that $A_{\iota,\iota^{*}}$
is existentially definable and has finitely many connected components. 
\begin{rem}
\label{rem: bounded sets are sub-lambda}Let $A\subseteq\left[-1,1\right]^{n}$
be an $\mathbb{R}_{\mathcal{A}}$-definable set. Using the above decomposition
of $A$ and the substitution $y_{i}\mapsto1/y_{i}$ whenever $y_{i}>1$,
one can see that $A$ is actually a sub-$\Lambda$-set. In particular,
the Parametrisation Theorem \ref{thm: param subanal} applies to $A$.\end{rem}
\begin{void}
\label{vuoto: polybdd}We now prove polynomial boundedness.\end{void}
\begin{rem}
\label{rem:composition in one var}Let $f\in\mathcal{A}_{1}$. Then,
since the support of $\mathcal{T}\left(f\right)$ is a well ordered
set, $f$ is normal. It follows from Lemma \ref{lem:composition with monomials}
that if $g$ is a function of one variable, which is obtained as a
finite composition of functions in $\mathcal{A}$ and vanishing at
0, then $g\in\mathcal{A}$.
\end{rem}
Let $\varepsilon>0$ and let $f:\left(0,\varepsilon\right)\to\mathbb{R}$
be definable in $\mathbb{R}_{\mathcal{A}}.$ We proceed as in \cite[Theorem B, pag. 4419]{vdd:speiss:gen},
using Theorem \ref{thm: param subanal} and Remark \ref{rem:composition in one var}
instead of 9.6 and \cite[Lemma 7.10]{krs}%
\footnote{Notice that in \cite[Lemma 7.10]{krs} one should replace $\lambda$
by $\lambda^{-1/\alpha}$ in the blow-up $\mathbf{r}^{\rho,\lambda}$
and in the definition of $g$.%
} instead of 9.9. We conclude that there exists $h\in\mathcal{A}_{1}$
such that $f\left(t\right)=h\left(t\right)$ for all sufficiently
small $t>0$. In particular, there exist $\alpha\in\mathbb{K}$ and
a unit $u\left(t\right)\in\mathcal{A}_{1}$ such that $f\left(t\right)=t^{\alpha}u\left(t\right)$.
This finishes the proof of Theorem A.

\section{Vertical monomialisation\label{sec:Vertical-monomialisation}}

In this section we prove Theorem B.
\begin{proviso}
\textbf{\label{Proviso: admissible transf} }Consider the ramifications
$r_{m+i}^{d,\pm}:\hat{I}_{m,n,r}\to\hat{I}_{m,n,r}$ in Definition
\ref{Def: elementary transformations} and let $\sigma_{m+i}^{\pm}$
be the restriction of $r_{m+i}^{1,\pm}$ to the half-space $\left\{ y_{i}'\geq0\right\} $.
In this section we will consider the transformations $\sigma_{m+i}^{\pm}$
as elementary transformations (and extend accordingly the notion of
admissible transformation in Definition \ref{Def: admiss transf}).
Notice that $\sigma_{m+i}^{+}\left(\hat{I}_{m+1,n-1,r}\right)\cup\sigma_{m+i}^{-}\left(\hat{I}_{m+1,n-1,r}\right)=\hat{I}_{m,n,r}$.
\end{proviso}
The main result of this section is the following.
\begin{thm}
\label{thm: monomialis of def functions}Let $D\subseteq\mathbb{R}^{N}$
and $\eta:D\to\mathbb{R}$ be an $\mathbb{R}_{\mathcal{A}}$-definable
function such that the graph of $\eta$ is a sub-$\Lambda$-set. Then
there exist a polyradius $r$ and a finite family $\mathcal{F}=\left\{ \rho_{i}:\hat{I}_{m_{i},n_{i},r_{i}}\to\mathbb{R}^{N}\right\} _{i=1,\ldots,M}$
(with $m_{i}+n_{i}=N$) of admissible transformations such that for
all $i=1,\ldots,M$ the function $\eta\circ\rho_{i}$ belongs to $\mathcal{A}_{m_{i},n_{i},r_{i}}$
and its germ at zero is normal, and $D\cap\hat{I}_{0,N,r}\subseteq\bigcup_{i=1}^{M}\rho_{i}\left(\hat{I}_{m_{i},n_{i},r_{i}}\right)$.
\end{thm}
Before proving the above theorem, we give a list of consequences.
\begin{cor}
\label{cor: def functions are terms}Let $D\subseteq\mathbb{R}^{N}$
and $\eta:D\to\mathbb{R}$ be an $\mathbb{R}_{\mathcal{A}}$-definable
function. Then there exist finitely many terms $t_{1},\ldots,t_{M}$
of the language $\mathcal{L}:=\mathcal{L}_{\mathcal{A}}\cup\left\{ ^{-1}\right\} \cup\left\{ \sqrt[n]{\ }:\ n\in\mathbb{N}\right\} $
such that 
\[
\forall x\in D\ \exists i\in\left\{ 1,\ldots,M\right\} \ \ \eta\left(x\right)=t_{i}\left(x\right).
\]
\end{cor}
\begin{proof}
The proof is by induction on $N$. If $N=1$, we can conclude by \ref{vuoto: polybdd},
after using the substitution $x\mapsto-x$ on $D\cap\mathbb{R}^{<0}$,
the substitution $x\mapsto1/x$ on $D\cap\mathbb{R}\setminus[-1,1]$
and arguing with $1/\eta$ on $\left\{ x\in D:\ |\eta\left(x\right)|>1\right\} $.
Hence we may suppose $N>1$. Notice that if $\text{dim}\left(D\right)<N$,
then by cell decomposition we may assume that, up to a permutation
of the variables, $D$ is the graph of some definable function $\eta':D'\to\mathbb{R}$,
where $D'\subseteq\mathbb{R}^{N-1}$. Hence, the function $\eta\left(x_{1},\ldots,x_{N}\right)$
coincides with the function $\eta\left(x_{1}\ldots,x_{N-1},\eta'\left(x_{1},\ldots,x_{N-1}\right)\right)$
and we can conclude by the inductive hypothesis. 

Let $\Gamma\left(\eta\right)$ be the graph of $\eta$. Let $A=D\times\mathbb{R}$
and consider the partition of $A$ given by the sets $A_{\iota,\iota^{*}}$
as in Subsection \ref{subsec:Proof-of-o-minimality}. Then each set
$\Gamma\left(\eta\right)\cap A_{\iota,\iota^{*}}$ is either empty
or the graph of an $\mathbb{R}_{\mathcal{A}}$-definable function
and $\tau_{\iota}^{N+1}\left(\Gamma\left(\eta\right)\cap A_{\iota,\iota^{*}}\right)$
is a sub-$\Lambda$-set. Moreover,

\[
\left(x,y\right)\in\Gamma\left(\eta\right)\cap A_{\iota,\iota^{*}}\Leftrightarrow\tau_{\iota}^{N+1}\left(x,y\right)\in\Gamma\left(\tau_{\iota}^{N+1}\circ\eta\circ\left(\tau_{\iota}^{N+1}\right)^{-1}\right)\cap\tau_{\iota}^{N+1}\left(A_{\iota,\iota^{*}}\right),
\]
hence it suffices to prove the corollary for the function $\tilde{\eta}:=\tau_{\iota}^{N+1}\circ\eta\circ\left(\tau_{\iota}^{N+1}\right)^{-1}$,
whose graph is a sub-$\Lambda$-set. We apply Theorem \ref{thm: monomialis of def functions}
to the function $\tilde{\eta}:\tilde{D}\to\mathbb{R}$ and obtain
in particular that the function $\tilde{\eta}\circ\rho_{i}:\hat{I}_{m_{i},n_{i},r_{i}}\to\mathbb{R}$
is an $\mathcal{L}$-term $t_{i}$. Arguing by induction on the length
of $\rho_{i}$, it is easy to see that there is a closed sub-$\Lambda$-set
$S\subseteq\hat{I}_{m_{i},n_{i},r_{i}}$ of dimension strictly smaller
than $N$ such that $\rho_{i}\restriction\hat{I}_{m_{i},n_{i},r_{i}}\setminus S$
is a diffeomorphism onto its image and that the components of the
map $\left(\rho_{i}\restriction\hat{I}_{m_{i},n_{i},r_{i}}\setminus S\right)^{-1}$
are $\mathcal{L}$-terms. Hence, for $x\in\rho_{i}\left(\hat{I}_{m_{i},n_{i},r_{i}}\setminus S\right)$
we have that $\tilde{\eta}\left(x\right)=t_{i}\circ\rho_{i}^{-1}\left(x\right)$
is an $\mathcal{L}$-term. Notice that the function $\tilde{\eta}\restriction\rho_{i}\left(S\right)$
can be dealt with by using the inductive hypothesis. Hence we have
proved the corollary for $\tilde{\eta}\restriction\tilde{D}\cap\hat{I}_{0,N,r}$,
for some polyradius $r$. 

By applying the same argument to the function $\tilde{\eta}\left(x+a\right)$
for every point $a$ of the closed unit cube, we obtain the full statement
of the corollary by a compactness argument.
\end{proof}
As an immediate consequence of the above corollary we obtain:
\begin{proof}
[Proof of Theorem B]Let $A\subseteq\mathbb{R}^{N}$ be an $\mathbb{R}_{\mathcal{A}}$-definable
set. We prove by induction on $N$ that $A$ is quantifier free definable
in the language $\mathcal{L}:=\mathcal{L}_{\mathcal{A}}\cup\left\{ ^{-1}\right\} \cup\left\{ \sqrt[n]{\ }:\ n\in\mathbb{N}\right\} $.
This is clear if $N=1$. Hence, we may suppose that $N>1$ and that,
by cell decomposition, $A$ is a cell; in particular, there exist
a cell $C\subseteq\mathbb{R}^{N-1}$ and $\mathbb{R}_{\mathcal{A}}$-definable
functions $f,g:C\to\mathbb{R}$ such that $A$ is either the graph
of $f\restriction C$, or the epigraph of $f\restriction C$, or the
set $\left\{ \left(x,y\right)\in\mathbb{R}^{N-1}\times\mathbb{R}:\ x\in C\ \land\ g\left(x\right)<y<f\left(x\right)\right\} $.
By Corollary \ref{cor: def functions are terms}, there is a finite
partition of $C$ into definable sets such that on every set of the
partition the functions $f$ and $g$ coincides with some $\mathcal{L}$-terms
$t_{1}$ and $t_{2}$, respectively. By the inductive hypothesis,
each set of the partition is quantifier free definable in the language
$\mathcal{L}$, hence so is $A$, in each of the above cases.
\end{proof}
Finally, we obtain the following Rectilinearisation Theorem, in the
spirit of \cite{hironaka_real_analytic}.
\begin{namedthm}
{Rectilinearisation Theorem}\label{cor: rectilinearisation}Let $D\subseteq\mathbb{R}^{N}$
be a sub-$\Lambda$-set. Then there exist a neighbourhood $W$ of
0 in $\mathbb{R}^{N}$ and a finite family $\mathcal{F}=\left\{ \left(\rho_{i},Q_{i}\right):\ i=1,\ldots,M\right\} $,
where $\rho_{i}:\hat{I}_{m_{i},n_{i},r{}_{i}}\to\mathbb{R}^{N}$ is
an admissible transformation (with $m_{i}+n_{i}=N$) and $Q_{i}\subseteq\overline{Q_{i}}\subseteq\hat{I}_{m_{i},n_{i},r_{i}}$
is a sub-quadrant, such that $\rho_{i}\restriction Q_{i}:\ Q_{i}\to\rho_{i}\left(Q_{i}\right)$
is a diffeomorphism and 
\[
W\cap D=\bigcup_{i=1}^{M}\rho_{i}\left(Q_{i}\right).
\]

\end{namedthm}
The Rectilinearisation Theorem is a consequence of Proposition \ref{prop: strong rectilin}
below. We need some definitions.
\begin{notation}
Let $N\in\mathbb{N}$. For the rest of the section, whenever we write
an admissible transformation $\rho:\hat{I}_{m_{\rho},n_{\rho},r_{\rho}}\to\hat{I}_{m,n,r}$,
we will implicitly assume that $m+n=m_{\rho}+n_{\rho}=N$ and that
$r=\left(r',r''\right)\in\left(0,\infty\right)^{m}\times\left(0,\infty\right)^{n},\ r_{\rho}=\left(r_{\rho}',r_{\rho}''\right)\in\left(0,\infty\right)^{m_{\rho}}\times\left(0,\infty\right)^{n_{\rho}}$
are polyradii in $\mathbb{R}^{N}$.
\end{notation}
Recall that, if $f\in\mathcal{A}_{m,n}$ and $\rho:\hat{I}_{m_{\rho},n_{\rho},r_{\rho}}\to\hat{I}_{m,n,r}$
is an admissible transformation, then $f\circ\rho\in\mathcal{A}_{m_{\rho},n_{\rho}}$.
The next definition is intended to extend this property to the case
when the arity of $f$ is bigger than $N$.
\begin{defn}
\label{def: rho respects f}Let $N,l\in\mathbb{N}$ and $\hat{m},\hat{n},k_{1},k_{2}\in\mathbb{N}$
with $\hat{m}+\hat{n}=N$ and $k_{1}+k_{2}=l$. Let $\rho:\hat{I}_{m_{\rho},n_{\rho},r_{\rho}}\to\hat{I}_{m,n,r}$
be an admissible transformation. We say that $\rho$ \emph{respects}
the elements of $\mathcal{A}_{\hat{m}+k_{1},\hat{n}+k_{2}}$ if $m\geq\hat{m}$. 

Let $s=\left(s_{1},s_{2}\right)\in\left(0,\infty\right)^{k_{1}}\times\left(0,\infty\right)^{k_{2}}$
be a polyradius and $\hat{r}:=\left(r',s_{1},r'',s_{2}\right)$, $\widehat{r_{\rho}}:=\left(r_{\rho}',s_{1},r_{\rho}'',s_{2}\right)$.
Let $x=\left(x_{1},\ldots,x_{N}\right),\ x'=\left(x_{1}',\ldots,x_{N}'\right),\ u=\left(u_{1},\ldots,u_{l}\right),\ u'=\left(u_{1}',\ldots,u_{l}'\right)$. 

Define the map
\[
\xyC{0mm}\xyL{0mm}\xymatrix{\hat{\rho}\colon & \hat{I}_{m_{\rho}+k_{1},n_{\rho}+k_{2},\widehat{r_{\rho}}}\ar[rrrr] & \  & \  & \  & \hat{I}_{m+k_{1},n+k_{2},\hat{r}}\\
 & \left\langle x',u'\right\rangle \ar@{|->}[rrrr] &  &  &  & \left\langle x,u\right\rangle 
}
,
\]
where $\left\langle x',u'\right\rangle $ is the ordered tuple $\left(x_{1}',\ldots,x_{m_{\rho}}',u_{1}',\ldots,u_{k_{1}}',x_{m_{\rho}+1}',\ldots,x_{N}',u_{k_{1}+1}',\ldots,u_{l}'\right)$,
$\left\langle x,u\right\rangle $ is the ordered tuple $\left(x_{1},\ldots,x_{m},u_{1},\ldots,u_{k_{1}},x_{m+1},\ldots,x_{N},u_{k_{1}+1},\ldots,u_{l}\right)$,
and $x=\rho\left(x'\right)$ , $u=u'$. The map $\hat{\rho}$ is an
admissible transformation, which we call a \emph{trivial extension}
of $\rho$. 

Notice that $\rho$ respects $f\in\mathcal{A}_{\hat{m}+k_{1},\hat{n}+k_{2}}$
if and only if $f\circ\hat{\rho}\in\mathcal{A}_{m_{\rho}+k_{1},n_{\rho}+k_{2},\widehat{r_{\rho}}}$.$ $
\end{defn}

In the next definition we will consider two sets of variables, the
``horizontal variables'', usually denoted by $x$, and the ``vertical
variables'', usually denoted by $u$. We will define a special type
of admissible transformations, the ``vertical admissible transformations'',
which respect in some way the partition of the set of variables into
horizontal and vertical. The aim is to obtain a class of admissible
transformations $\rho:\left(x',u'\right)\mapsto\left(x,u\right)$
with the property that, given $f\in\mathcal{A}$, if we can solve
explicitly the equation $f\circ\rho\left(x',u'\right)=0$ with respect
to $u'$, then we can solve explicitly the equation $f\left(x,u\right)=0$
with respect to $u$.
\begin{defn}
\label{def:vertical transformations}Let $r=\left(s,t\right),\ r_{\rho}=\left(s_{\rho},t_{\rho}\right)\in\left(0,\infty\right)^{N}$$\times\left(0,\infty\right)^{l}$
be polyradii. Consider a map
\[
\xyC{0mm}\xyL{0mm}\xymatrix{\rho\colon & \hat{I}_{r_{\rho}}\ar[rrrr] & \  & \  & \  & \hat{I}_{r}\\
 & \langle x',u'\rangle\ar@{|->}[rrrr] &  &  &  & \langle x,u\rangle
}
,
\]
where $\hat{I}_{r_{\rho}}$ is either $\hat{I}_{m_{\rho},n_{\rho}+l,r_{\rho}}$
(type 1) or $\hat{I}_{m_{\rho}+l,n_{\rho},r_{\rho}}$ (type 2) and
$\hat{I}_{r}$ is either $\hat{I}_{m,n+l,r}$ (type 1) or $\hat{I}_{m+l,n,r}$
(type 2). If $\hat{I}_{r_{\rho}}$ is of type 1, then let $\left\langle x',u'\right\rangle $
be the ordered pair $\left(x',u'\right)$, whereas if $\hat{I}_{r_{\rho}}$
is of type 2, then let $\left\langle x',u'\right\rangle $ be the
ordered pair $\left(x_{1}',\ldots,x_{m_{\rho}}',u_{1}',\ldots,u_{l}',x_{m_{\rho}+1}',\ldots,x_{N}'\right)$.
If $\hat{I}_{r}$ is of type 1, then let $\left\langle x,u\right\rangle $
be the ordered pair $\left(x,u\right)$ and $\rho'=\left(\rho_{1},\ldots,\rho_{N}\right),$
$\rho''=\left(\rho_{N+1},\ldots,\rho_{N+l}\right)$, whereas if $\hat{I}_{r}$
is of type 2, then let $\left\langle x,u\right\rangle $ be the ordered
pair $\left(x_{1},\ldots,x_{m},u_{1},\ldots,u_{l},x_{m+1},\ldots,x_{N}\right)$
and $\rho'=\left(\rho_{1},\ldots,\rho_{m},\rho_{m+l+1},\ldots,\rho_{N+l}\right),$
$\rho''=\left(\rho_{m+1},\ldots,\rho_{m+l}\right)$.

We say that $\rho=\left\langle \rho',\rho''\right\rangle $ is a \emph{vertical
admissible transformation} (of type $\left(i,j\right)\in\left\{ 1,2\right\} ^{2}$)
if the following conditions are satisfied:
\begin{itemize}
\item $\hat{I}_{r_{\rho}}$ is of type $i$ and $\hat{I}_{r}$ is of type
$j$;
\item $\rho$ is an admissible transformation;
\item $\rho'$ does not depend on $u'$, hence we may write $x=\rho'\left(x'\right)$;
\item there exists a closed sub-$\Lambda$-set $S_{\rho}\subseteq\hat{I}_{m_{\rho},n_{\rho},s_{\rho}}$
such that $\text{dim}\left(S_{\rho}\right)<N$ and $\rho'\restriction\hat{I}_{m_{\rho},n_{\rho},s_{\rho}}\setminus S_{\rho}$
is a diffeomorphism onto its image. Moreover, for every $x'\in\hat{I}_{m_{\rho},n_{\rho},s_{\rho}}\setminus S_{\rho}$,
the map $u'\mapsto\rho''\left(\left\langle x',u'\right\rangle \right)$
is a diffeomorphism onto its image, and we denote by $\gamma_{\rho}$
the map $\left(x',u\right)\mapsto u'$.
\end{itemize}
Notice that $\rho\restriction\hat{I}_{r_{\rho}}\setminus S_{\rho}\times\mathbb{R}^{l}$
is a diffeomorphism onto its image.

\smallskip{}

Let $D\subseteq\mathbb{R}^{N}$ and $\Phi:D\to\mathbb{R}^{l}$ be
a map whose graph is a sub-$\Lambda$-set and let $D_{\rho}:=\left(\rho'\right)^{-1}\left(D\right)\setminus S_{\rho}$.
Suppose that $\Phi\left(\rho'\left(D_{\rho}\right)\right)\subseteq\rho''\left(\hat{I}_{r_{\rho}}\right)$.
We define the map $\Phi_{\rho}:D_{\rho}\to\mathbb{R}^{l}$ as follows:
for every $x'\in D_{\rho},\ \Phi_{\rho}\left(x'\right):=\gamma_{\rho}\left(x',\Phi\circ\rho'\left(x'\right)\right)$.
\end{defn}

Examples of admissible transformations which are not vertical are
the blow-up chart $\left(x,u\right)\mapsto\left(xu,u\right)$ and
the linear transformation $\left(x,u\right)\mapsto\left(x+cu,u\right)$,
because in these cases the first component of the transformation depends
on the vertical variable $u$. Another example is the blow-up chart
$\left(x,u_{1},u_{2}\right)\mapsto\left(x,u_{1}u_{2},u_{2}\right)$,
because the second component of the transformation is not a bijection. 

Our aim in this section is to give a monomialisation algorithm which
only uses vertical admissible transformations. We will do so at the
expenses of covering a proper subset of the domain of the functions.
Such a subset will turn out to be the graph of a definable map. This
motivates the next definition.
\begin{defn}
\label{def: property *}$ $Let $N,l\in\mathbb{N}$ with $N\geq l$.
Let $S\subseteq D\subseteq\mathbb{R}^{N}$ be sub-$\Lambda$-sets,
with $\text{dim}\left(S\right)<N$.
\begin{enumerate}
\item Let $\mathcal{F}$ be a finite family of admissible transformations
\[
\rho:\hat{I}_{m_{\rho},n_{\rho},r_{\rho}}\to\hat{I}_{m,n,r}.
\]
We say that $\mathcal{F}$ \emph{satisfies the covering} \emph{property
with respect to} $D$ if for every choice of polyradii $r_{\rho}'\leq r_{\rho}\ \left(\rho\in\mathcal{F}\right)\ $
there exists a polyradius $r^{*}\leq r$ such that
\[
\hat{I}_{m,n,r^{*}}\cap D\cap\rho\left(\hat{I}_{m_{\rho},n_{\rho},r_{\rho}'}\right)\not=\emptyset\text{\ and\ }\hat{I}_{m,n,r^{*}}\cap D\subseteq\bigcup_{\rho\in\mathcal{F}}\rho\left(\hat{I}_{m_{\rho},n_{\rho},r_{\rho}'}\right).
\]

\item Let $\Phi:D\to\mathbb{R}^{l}$ be a map whose graph is a sub-$\Lambda$-set
and let $\mathcal{F}$ be a finite family of vertical admissible transformations
$ $(of the same type)
\[
\rho:\hat{I}_{r_{\rho}}\to\hat{I}_{r}.
\]
 We say that $\mathcal{F}$ \emph{satisfies the covering property
with respect to} $\left(\Phi,S\right)$ if for every choice of polyradii
$r_{\rho}'\leq r_{\rho}\ \left(\rho\in\mathcal{F}\right)\ $ there
exists a polyradius $r^{*}\leq r$ such that
\[
\hat{I}_{r^{*}}\cap\Gamma\left(\Phi\restriction D\setminus S\right)\cap\rho\left(\hat{I}_{r_{\rho}'}\right)\not=\emptyset\text{\ and\ }\hat{I}_{r^{*}}\cap\Gamma\left(\Phi\restriction D\setminus S\right)\subseteq\bigcup_{\rho\in\mathcal{F}}\rho\left(\hat{I}_{r_{\rho}'}\right).
\]
 
\end{enumerate}
\end{defn}
The following remarks will be used several times throughout the rest
of the paper.
\begin{rems}
\label{rem: evolution of phi}$\ $Let $D$ and $\Phi$ be as above. 
\begin{enumerate}
\item Let $\mathcal{F}$ be a finite family of admissible transformations
which satisfies the covering property with respect to $D$. Suppose
that for every $\rho\in\mathcal{F}$ there is a finite family $\tilde{\mathcal{F}}_{\rho}$
of admissible transformations such that $\tilde{\mathcal{F}}_{\rho}$
satisfies the covering property with respect to $\rho^{-1}\left(D\right)$.
Then $\mathcal{G}$ satisfies the covering property with respect to
$D$, where $\mathcal{G}=\left\{ \rho\circ\tilde{\rho}:\ \rho\in\mathcal{F},\ \tilde{\rho}\in\tilde{\mathcal{F}}_{\rho}\right\} $.
\item Let $\mathcal{F}$ be a finite family of admissible transformations
which satisfies the covering property with respect to $D$. In the
notation of Definition \ref{def: rho respects f}, let $\pi$ be the
projection $\left\langle x,u\right\rangle \mapsto x$ and $\hat{D}\subseteq\mathbb{R}^{N+l}$
be a sub-$\Lambda$-set such that $\pi\left(\hat{D}\right)=D$. Then
$ $\foreignlanguage{english}{$\hat{\mathcal{F}}$} satisfies the
covering property with respect to $\hat{D}$, where $\hat{\mathcal{F}}$
is the family of all admissible transformations $\hat{\rho}$ obtained
by extending trivially $\rho\in\mathcal{F}$.
\item Let $\mathcal{F}$ be a finite family of admissible transformations
and suppose that $\mathcal{F}$ satisfies the covering property with
respect to $D$. Then $ $\foreignlanguage{english}{$\hat{\mathcal{F}}$}
satisfies the covering property with respect to $\left(\Phi,\emptyset\right)$,
where $\hat{\mathcal{F}}$ is a finite family of vertical admissible
transformations (of the same type) obtained by extending trivially
each $\rho\in\mathcal{F}$ to $\hat{\rho}:\hat{I}_{r_{\rho}}\to\hat{I}_{r}$,
where $\hat{I}_{r_{\rho}},\ \hat{I}_{r}$ are of either of the two
types. Conversely, if $\mathcal{F}$ is a family of vertical admissible
transformations which satisfies the covering property with respect
to $\left(\Phi,\emptyset\right)$, then the family $\mathcal{F}'=\left\{ \rho':\ \rho\in\mathcal{F}\right\} $
of admissible transformations satisfies the covering property with
respect to $D$.
\item Let $\mathcal{F}$ be a finite family of vertical admissible transformations
(of the same type). Suppose that $S\supseteq\bigcup_{\rho\in\mathcal{F}}\rho\left(S_{\rho}\right)$
and that $\mathcal{F}$ satisfies the covering property with respect
to $\left(\Phi,S\right)$. Suppose furthermore that for every $\rho\in\mathcal{F}$
there are a finite family $\tilde{\mathcal{F}}_{\rho}$ of vertical
admissible transformations (of the same type) and a sub-$\Lambda$-set
$S_{\rho}'\supseteq\bigcup_{\tilde{\rho}\in\tilde{\mathcal{F}}_{\rho}}\tilde{\rho}\left(S_{\tilde{\rho}}\right)$
of dimension strictly smaller than $N$ such that $\tilde{\mathcal{F}}_{\rho}$
satisfies the covering property with respect to $\left(\Phi_{\rho},S_{\rho}'\right)$.
Then $\mathcal{G}$ satisfies the covering property with respect to
$\left(\Phi,S'\right)$, where $\mathcal{G}=\left\{ \rho\circ\tilde{\rho}:\ \rho\in\mathcal{F},\ \tilde{\rho}\in\tilde{\mathcal{F}}_{\rho}\right\} $
and $S'=S\cup\bigcup_{\rho\in\mathcal{F}}\rho\left(S_{\rho}'\right)$.
\end{enumerate}
\end{rems}
The above definitions and remarks allow us to revisit the statement
of Theorem \ref{thm: geom monomial}.
\begin{rem}
\label{rem: monomialise respecting f}Let $f\in\mathcal{A}_{\hat{m}+k_{1},\hat{n}+k_{2}}$
and $g\in\mathcal{A}_{\check{m},\check{n}}$. Define $m:=\text{max}\left\{ \hat{m},\check{m}\right\} $
and $n:=N-m$. Then we can apply Theorem \ref{thm: geom monomial}
to $g$, seen as an element of $\mathcal{A}_{m,n}$ and obtain that
the admissible transformations in the statement respect $f$. Moreover,
using Remark \ref{rem: elemntary transf send quadrants to sectors},
the conclusion of Theorem \ref{thm: geom monomial} can be strengthened
by saying that $\mathcal{F}$ satisfies the covering property with
respect to $\hat{I}_{m,n,r'}$.
\end{rem}
The purpose of statements \textbf{(B)} and \textbf{(C) }in the next
theorem is to solve a given system of equations with respect to the
vertical variables. Statement \textbf{(A)} implies directly Theorem
\ref{thm: monomialis of def functions}. 
\begin{notation}
Let $\Phi=\left(\varphi_{1},\ldots,\varphi_{l}\right):D\to\mathbb{R}^{l}$
be a map whose graph is a sub-$\Lambda$-set. We denote by $j\left(\Phi\right)$
the cardinality of the set $\left\{ i:\ 1\leq i\leq l,\ \varphi_{i}\text{\ is\ not\ identically\ }0\right\} $.
Up to a permutation, we may always assume that the first $j\left(\Phi\right)$
coordinates of $\Phi$ are not identically zero. We set $\hat{\Phi}=\emptyset$
if $j\left(\Phi\right)=0$ and $\hat{\Phi}=\left(\varphi_{1},\ldots,\varphi_{j\left(\Phi\right)}\right)$
if $j\left(\Phi\right)>0$. For $F\left(x,u\right)\in\mathcal{A}_{\check{m},\check{n}+l}$,
we let $F_{0}\left(x,u_{1},\ldots,u_{j\left(\Phi\right)}\right):=F\left(x,u_{1},\ldots,u_{j\left(\Phi\right)},0\right)\in\mathcal{A}_{\check{m},\check{n}+j\left(\Phi\right)}$.\end{notation}
\begin{thm}
\label{thm: ABC}Let $N,l,\hat{m},\hat{n},k_{1},k_{2}\in\mathbb{N}$
with $l\leq N=\hat{m}+\hat{n}$ and let $f\in\mathcal{A}_{\hat{m}+k_{1},\hat{n}+k_{2}}$.
$ $Let $D\subseteq[-1,1]{}^{N}$ be a sub-$\Lambda$-set and suppose
that for every sufficiently small polyradius $r'$ in $\mathbb{R}^{N}$,
the intersection $\hat{I}_{\hat{m},\hat{n},r'}\cap D$ is not empty.

\noindent \textbf{(A)$_{N}$} Let $\eta:D\to\mathbb{R}$ be a function
whose graph is a sub-$\Lambda$-set. Then there exists a finite family
$\mathcal{F}$ of admissible transformations
\[
\rho:\hat{I}_{m_{\rho},n_{\rho},r_{\rho}}\to\hat{I}_{m,n,r}
\]
such that \emph{$\mathcal{F}$ }satisfies the covering property with
respect to\emph{ $D$ }and for every $\rho\in\mathcal{F}$,
\begin{itemize}
\item $m=\hat{m}$, hence $\rho$ respects $f$;
\item $\eta\circ\rho\in\mathcal{A}_{m_{\rho},n_{\rho},r_{\rho}}$ and is
normal.
\end{itemize}
\noindent \textbf{(B)$_{N,l}$} $ $Let $\check{m},\check{n}\in\mathbb{N}$
with $\check{m}+\check{n}=N$ and let $F\left(x,u\right)\in\mathcal{A}_{\check{m},\check{n}+l}$.
Let $\text{dim}\left(D\right)=N$ and $\Phi:D\to\mathbb{R}^{l}$ be
a map whose graph is a sub-$\Lambda$-set.

Then there exist a finite family $\mathcal{F}$ of vertical admissible
transformations 
\[
\rho:\hat{I}_{m_{\rho},n_{\rho}+j\left(\Phi\right),r_{\rho}}\to\hat{I}_{m,n+j\left(\Phi\right),r}
\]
 and $ $ a sub-$\Lambda$-set $S\subseteq D$ of dimension strictly
smaller than $N$ such that $\mathcal{F}$ satisfies the covering
property with respect to $\left(\hat{\Phi},S\right)$ and for every
$\rho\in\mathcal{F}$,
\begin{itemize}
\item $m=\text{max}\left\{ \hat{m},\check{m}\right\} $, hence $\rho'$
respects $f$ and $\rho$ respects $F$;
\item either $j\left(\hat{\Phi}_{\rho}\right)<j\left(\Phi\right)$ or $F_{0}\circ\rho$
is normal. 
\end{itemize}
\textbf{(C)$_{N,l}$} $ $Let $\check{m},\check{n}\in\mathbb{N}$
with $\check{m}+\check{n}=N$ and let $f_{1}\left(x,u\right),\ldots,f_{N}\left(x,u\right)\in\mathcal{A}_{\check{m},\check{n}+N}$.
Let $\text{dim}\left(D\right)=N$ and $\Phi:D\to\mathbb{R}^{N}$ be
a map whose graph is a sub-$\Lambda$-set.

Suppose that $j\left(\Phi\right)=l$ and that
\[
\forall\left(x,u\right)\in D\times\mathbb{R}^{N}\ \ \ \ \ \begin{cases}
f_{1}\left(x,u\right)=0\\
\vdots\\
f_{N}\left(x,u\right)=0
\end{cases}\Leftrightarrow\ \ u=\Phi\left(x\right).
\]
Then there exist a finite family $\mathcal{F}$ of vertical admissible
transformations 
\[
\rho:\hat{I}_{m_{\rho},n_{\rho}+N,r_{\rho}}\to\hat{I}_{m,n+N,r}
\]
 and $ $ a sub-$\Lambda$-set $S\subseteq D$ of dimension strictly
smaller than $N$ such that $\mathcal{F}$ satisfies the covering
property with respect to $\left(\Phi,S\right)$ and for every $\rho\in\mathcal{F}$,
\begin{itemize}
\item $m=\text{max}\left\{ \hat{m},\check{m}\right\} $, hence $\rho'$
respects $f$ and $\rho$ respects $f_{1},\ldots,f_{N}$;
\item $\forall\left(x,u\right)\in D_{\rho}\setminus\left(\rho'\right)^{-1}\left(S\right)\times\mathbb{R}^{N}\ \ \ \ \ \begin{cases}
f_{1}\circ\rho\left(x,u\right)=0\\
\vdots\\
f_{N}\circ\rho\left(x,u\right)=0
\end{cases}\Leftrightarrow\ \ u=0$. 
\end{itemize}
\end{thm}
Before proving Theorem \ref{thm: ABC}, we show how it implies the
Rectilinearisation Theorem. We prove the following stronger statement.
\begin{prop}
\label{prop: strong rectilin}Let $m,n,k_{1},k_{2}\in\mathbb{N}$,
with $N=m+n$, and $f\in\mathcal{A}_{m+k_{1},n+k_{2}}$. Let $D\subseteq\mathbb{R}^{N}$
be a sub-$\Lambda$-set. Then there exist a neighbourhood $W$ of
0 in $\mathbb{R}^{N}$ and a finite family $\mathcal{F}=\left\{ \left(\rho_{i},Q_{i}\right):\ i=1,\ldots,M\right\} $,
where $\rho_{i}:\hat{I}_{m_{i},n_{i},r{}_{i}}\to\hat{I}_{m,n,r}$
is an admissible transformation (which respects $f$) and $Q_{i}\subseteq\overline{Q_{i}}\subseteq\hat{I}_{m_{i},n_{i},r_{i}}$
is a sub-quadrant, such that $\rho_{i}\restriction Q_{i}:\ Q_{i}\to\rho_{i}\left(Q_{i}\right)$
is a diffeomorphism and 
\[
W\cap\hat{I}_{m,n,r}\cap D=\bigcup_{i=1}^{M}\rho_{i}\left(Q_{i}\right).
\]
\end{prop}
\begin{proof}
Notice that, by Proposition \ref{prop: param basic sets} and Remark
\ref{rem: monomialise respecting f}, the proposition has already
been proved whenever $D$ is an $\mathcal{A}$-basic set. We aim to
show that we can reduce to this situation.

The proof is by induction on the pairs $\left(N,d\right)$, where
$d=\text{dim}\left(D\right)$, ordered lexicographically. The basic
cases $\left(0,0\right)$ and $\left(N,0\right)$ are straightforward.
By a cell decomposition argument, we may assume that $D$ is a cell
and, without loss of generality, that 
\[
D=\left\{ \left(x,y\right):\ x\in C,\ \theta\left(x,y\right)\right\} ,
\]
where, either
\begin{itemize}
\item (Case 1) $d=N,\ x=\left(x_{1},\ldots,x_{N-1}\right),\ y=x_{N}$, $C\subseteq\mathbb{R}^{N-1}$
is a sub-$\Lambda$-set and a cell of dimension $N-1$ and $\theta\left(x,y\right)$
is \foreignlanguage{english}{$y>\varphi_{0}\left(x\right)$,} where
the graph of the function $\varphi_{0}:C\to\mathbb{R}$ is a sub-$\Lambda$-set,
or
\item (Case 2) $d<N,\ x=\left(x_{1},\ldots,x_{d}\right),\ y=\left(y_{1},\ldots,y_{N-d}\right)=\left(x_{d+1},\ldots,x_{N}\right)$,
$C\subseteq\mathbb{R}^{d}$ is a sub-$\Lambda$-set and a cell of
dimension $d$ and $\theta\left(x,y\right)$ is $\bigwedge_{i=1}^{N-d}y_{i}=\varphi_{i}\left(x\right)$,
where the graphs of the functions $\varphi_{i}:C\to\mathbb{R}$ are
sub-$\Lambda$-sets.
\end{itemize}
We will treat the two cases simultaneously. Notice that in both cases
$c:=\text{dim}\left(C\right)\leq N-1$. 

We first show that we can assume that $\varphi_{i}\in\mathcal{A}$.
In fact, by the statement \textbf{(A)}$_{c}$, there exists a finite
family $\mathcal{F}$ of admissible transformations $\rho:\hat{I}_{l_{1}',l_{2}',r'}\to\hat{I}_{l_{1},l_{2},r}$,
where $l_{1}+l_{2}=l_{1}'+l_{2}'=c$, such that 
\[
\hat{I}_{l_{1},l_{2},r}\cap C\subseteq\bigcup_{\rho\in\mathcal{F}}\rho\left(\hat{I}_{l_{1}',l_{2}',r'}\right)
\]
and for all $\rho\in\mathcal{F}$ we have that $\rho$ respects $f$
and $\varphi_{i}\circ\rho\in\mathcal{A}_{l_{1}',l_{2}'}$. Arguing
by induction on the length of $\rho$ it is easy to see that there
exists a closed sub-$\Lambda$-set $S_{\rho}\subseteq\mathbb{R}^{c}$
of dimension $<c$ such that $\rho\restriction\hat{I}_{l_{1}',l_{2}',r'}\setminus S_{\rho}$
is a diffeomorphism onto its image. Let $B_{\rho}=\rho^{-1}\left(\hat{I}_{l_{1},l_{2},r}\cap C\right),\ B_{\rho}'=B_{\rho}\setminus S_{\rho}$
and $B_{\rho}''=B_{\rho}\cap S_{\rho}$. Let $\hat{\rho}:\hat{I}_{m',n',\hat{r'}}\to\hat{I}_{m,n,\hat{r}}$
be the trivial extension of $\rho$ and $\hat{\mathcal{F}}=\left\{ \hat{\rho}:\ \rho\in\mathcal{F}\right\} $.
Let 
\begin{align*}
D_{\rho}'= & \left\{ \left(x,y\right):\ x\in B_{\rho}',\ \theta_{\rho}\left(x,y\right)\right\} ,\\
D_{\rho}''= & \left\{ \left(x,y\right):\ x\in B_{\rho}'',\ \theta_{\rho}\left(x,y\right)\right\} ,
\end{align*}
where $\theta_{\rho}\left(x,y\right)$ is $y>\varphi_{0}\circ\rho\left(x\right)$
in case 1 and $\bigwedge_{i=1}^{N-d}y_{i}=\varphi_{i}\circ\rho\left(x\right)$
in case 2. Notice that there exists a neighbourhood $W\subseteq\mathbb{R}^{N}$
of $0$ such that
\[
\bigcup_{\hat{\rho}\in\hat{\mathcal{F}}}\hat{\rho}\left(D_{\rho}'\right)\cup\hat{\rho}\left(D_{\rho}''\right)=W\cap\hat{I}_{m,n,\hat{r}}\cap D.
\]
In both cases $\text{dim}\left(\hat{\rho}\left(D_{\rho}''\right)\right)\leq\text{dim}\left(D_{\rho}''\right)<d$,
so by the inductive hypothesis on the dimension $d$ of $D$, and
by the fact that $\hat{\rho}\restriction D_{\rho}'$ is a diffeomorphism
onto its image, we have reduced to the situation where the $\varphi_{i}$
are in $\mathcal{A}$.

Next we show that we can reduce to the case when $C$ is a sub-quadrant
(and hence $D$ is an $\mathcal{A}$-basic set). In order to see this,
since $c<N$, by the inductive hypothesis on the dimension $N$ of
the ambient space, we can apply the proposition to $C$. Hence there
are a neighbourhood $\tilde{W}\subseteq\mathbb{R}^{c}$ of $0$ and
a finite family $\mathcal{F}$ of admissible transformations $\rho:\hat{I}_{l_{1}',l_{2}',r'}\to\hat{I}_{l_{1},l_{2},r}$
(respecting $f,\varphi_{0},\ldots,\varphi_{N-l}$), where $l_{1}+l_{2}=l_{1}'+l_{2}'=c$,
and sub-quadrants $Q\subseteq\overline{Q}\subseteq\hat{I}_{l_{1}',l_{2}'.r'}$
such that $\rho\restriction Q$ is a diffeomorphism onto its image
and $\tilde{W}\cap\hat{I}_{l_{1},l_{2},r}\cap C={\displaystyle \bigcup_{\left(\rho,Q\right)\in\mathcal{F}}\rho\left(Q\right)}$.
Let $\hat{\rho}:\hat{I}_{m',n',\hat{r'}}\to\hat{I}_{m,n,\hat{r}}$
be the trivial extension of $\rho$ and $\hat{\mathcal{F}}=\left\{ \hat{\rho}:\ \rho\in\mathcal{F}\right\} $.
Let 
\begin{align*}
D_{\rho}= & \left\{ \left(x,y\right):\ x\in Q,\ \theta_{\rho}\left(x,y\right)\right\} ,
\end{align*}
where $\theta_{\rho}\left(x,y\right)$ is $y>\varphi_{0}\circ\rho\left(x\right)$
in case 1 and $\bigwedge_{i=1}^{N-d}y_{i}=\varphi_{i}\circ\rho\left(x\right)$
in case 2. Notice that there exists a neighbourhood $W\subseteq\mathbb{R}^{N}$
of $0$ such that
\[
\bigcup_{\hat{\rho}\in\hat{\mathcal{F}}}\hat{\rho}\left(D_{\rho}\right)=W\cap\hat{I}_{m,n,\hat{r}}\cap D,
\]
and since $\hat{\rho}\restriction D_{\rho}$ is a diffeomorphism onto
its image, we have reduced to the situation where $D$ is an $\mathcal{A}$-basic
set, and we are done.
\end{proof}
\medskip{}

The rest of the section is devoted to the proof of Theorem \ref{thm: ABC},
which is by induction on the pairs $\left(N,l\right)$, ordered lexicographically.
\smallskip{}

If $N=0$, then statements \textbf{(A)}, \textbf{(B)} and \textbf{(C)}
are trivial.

For\textbf{ }any $N$, we have that \textbf{(B)}$_{N,0}$ follows
from Theorem \ref{thm: geom monomial} and Remark \ref{rem: monomialise respecting f}
and \textbf{(C)}$_{N,0}$ is trivially true. Hence let us assume that
$N,l\geq1$.
\begin{lem}
\label{lem: weak monomialisation}Let $N\geq1$ and suppose that \textbf{(A)}$_{N-1}$
holds. Let $l,\hat{m},\hat{n},k_{1},k_{2}\in\mathbb{N}$ with $l\leq N=\hat{m}+\hat{n}$
and let $f\in\mathcal{A}_{\hat{m}+k_{1},\hat{n}+k_{2}}$. Let $D\subseteq[-1,1]{}^{N}$
be a sub-$\Lambda$-set and suppose that for every sufficiently small
polyradius $r'$ in $\mathbb{R}^{N}$, the intersection $\hat{I}_{\hat{m},\hat{n},r'}\cap D$
is not empty. Let $\text{dim}\left(D\right)=N$ and $\Phi:D\to\mathbb{R}^{l}$
be a map whose graph is a sub-$\Lambda$-set such that $j\left(\Phi\right)=l$.
Then there are constants $K_{1},K_{2}\in\mathbb{R}^{>0}$ and a finite
family $\mathcal{F}$ of admissible transformations
\[
\rho:\hat{I}_{m_{\rho},n_{\rho},r_{\rho}}\to\hat{I}_{m,n,r}
\]
such that \emph{$\mathcal{F}$ }satisfies the covering property with
respect to\emph{ $D$ }and for every $\rho\in\mathcal{F}$,
\begin{itemize}
\item $m=\hat{m}$, hence $\rho$ respects $f$;
\item for every $i=1,\ldots,l$ there are exponents $\alpha_{i}\in\mathbb{A}^{n}$
and functions $U_{i}:D\to\mathbb{R}$ (whose graph is a sub-$\Lambda$-set)
with $K_{1}\leq|U_{1}|\leq K_{2}$ such that $\varphi_{i}\circ\rho\left(x\right)=x^{\alpha_{i}}U_{i}\left(x\right)$. 
\end{itemize}
\end{lem}
\begin{rem}
\label{rem: weakly normal}Notice that we do not require at this stage
that $U_{i}$ be an element of $\mathcal{A}$, hence we say that $\varphi_{i}\circ\rho$
is \emph{weakly normal}.\end{rem}
\begin{proof}
Let us first consider the case $N=1$. Notice that, by polynomial
boundedness, for all $x\in D\cap\mathbb{R}^{\geq0}$ we have $\varphi_{i}\left(x\right)=x^{\alpha_{i}/\beta_{i}}U_{i}\left(x\right)$,
for some $U_{i}$ as in the statement of the statement of the lemma
and $\alpha_{i},\beta_{i}\in\mathbb{A}$. Let $\beta=\beta_{1}\cdot\ldots\cdot\beta_{l}$.
If $\hat{m}=1$ then $\mathcal{F}=\left\{ \sigma_{1}^{+}\circ r_{1}^{\beta}\right\} $
satisfies the conclusion of the lemma. If $\hat{n}=1$ then $\mathcal{F}=\left\{ \sigma_{1}^{+}\circ r_{1}^{\beta},\sigma_{1}^{-}\circ r_{1}^{\beta}\right\} $
satisfies the conclusion of the lemma. 

Let $N>1$. Let $\hat{x}=\left(x_{1},\ldots,x_{N-1}\right),\ y=x_{N}$
and $\hat{D}$ be the projection of $D$ onto the coordinate space
spanned by $\hat{x}$. By cell decomposition and by \cite[Thm. 2.1]{vdd:speiss:preparation_theorems},
we may assume that there are $\beta_{1},\ldots,\beta_{l}\in\mathbb{K}$,
$\mathbb{R}_{\mathcal{A}}$-definable functions $a_{0},a_{1},\ldots,a_{l}:\hat{D}\to\mathbb{R}$
and $U_{1},\ldots,U_{l}:D\to\mathbb{R}$ such that $\forall x=\left(\hat{x},y\right)\in D\ y>a_{0}\left(\hat{x}\right)$
and $\forall i=1,\ldots,l\ \ $ 
\[
\forall x=\left(\hat{x},y\right)\in D\ \ \varphi_{i}\left(\hat{x},y\right)=\left(y-a_{0}\left(\hat{x}\right)\right)^{\beta_{i}}a_{i}\left(\hat{x}\right)U_{i}\left(\hat{x},y\right)\text{\ and\ }\frac{1}{2}<|U_{i}|<\frac{3}{2}.
\]

By a further cell decomposition we may assume that all the units $U_{i}$
are positive on $D$ and that for all $i=0,1,\ldots,l$ either $\forall\hat{x}\in\hat{D}\ a_{i}\left(\hat{x}\right)\leq1$
or $\forall\hat{x}\in\hat{D}\ a_{i}\left(\hat{x}\right)>1$. For $i=0,\ldots,l$
let 
\[
\hat{a}_{i}\left(\hat{x}\right):=\begin{cases}
a_{i}\left(\hat{x}\right) & \text{if\ }a_{i}\leq1\text{\ on }\hat{D}\\
1/a_{i}\left(\hat{x}\right) & \text{if\ }a_{i}>1\text{\ on }\hat{D}
\end{cases}
\]
and let $\hat{a}\left(\hat{x}\right)=\prod_{i=0}^{l}\hat{a}_{i}\left(\hat{x}\right)$.
By Remark \ref{rem: bounded sets are sub-lambda} the graph of $\hat{a}$
is a sub-$\Lambda$-set. Hence, by \textbf{(A)}$_{N-1}$, there is
a finite family $\mathcal{F}$ of admissible transformations such
that every $\rho\in\mathcal{F}$ extends trivially to $\hat{\rho}:\hat{I}_{m_{\rho},n_{\rho}.r_{\rho}}\to\hat{I}_{\hat{m},\hat{n},r}$
(hence $\hat{\rho}$ respects $f$), $\hat{\mathcal{F}}:\left\{ \hat{\rho}:\ \rho\in\mathcal{F}\right\} $
satisfies the covering property with respect to $D$ and for all $\hat{\rho}\in\hat{\mathcal{F}},$
for all $i=0,.\ldots l$ we have $\hat{a}_{i}\circ\hat{\rho}\in\mathcal{A}_{m_{\rho},n_{\rho},r_{\rho}}$.
Let
\[
g_{\rho}\left(\hat{x},y\right)=\begin{cases}
y-\hat{a}_{0}\circ\rho\left(\hat{x}\right) & \text{if\ }a_{0}\leq1\text{\ on }\hat{D}\\
y\cdot\hat{a}_{0}\circ\rho\left(\hat{x}\right)-1 & \text{if\ }a_{0}>1\text{\ on }\hat{D}
\end{cases}
\]
and $h_{\rho}\left(\hat{x},y\right):=g_{\rho}\left(\hat{x},y\right)\cdot\prod_{i=0}^{l}\hat{a_{i}}\circ\rho\left(\hat{x}\right)\in\mathcal{A}_{m_{\rho},n_{\rho},r_{\rho}}.$
By Theorem \ref{thm: geom monomial} and Remark \ref{rem: monomialise respecting f},
there is a finite family $\tilde{\mathcal{F}}_{\rho}$ of admissible
transformations such that $\tilde{\mathcal{F}}_{\rho}$ satisfies
the covering property with respect to $\hat{I}_{m_{\rho},n_{\rho},r_{\rho}}$
and for every $\tilde{\rho}:\hat{I}_{m_{\tilde{\rho}},n_{\tilde{\rho}},r_{\tilde{\rho}}}\to\hat{I}_{m_{\rho},n_{\rho},r_{\rho}}\in\tilde{\mathcal{F}}_{\rho}$
we have that $\tilde{\rho}$ respects $f\circ\hat{\rho}$ and finally
there are $\gamma_{0},\ldots,\gamma_{l+1}\in\mathbb{A}^{N}$ and units
$v_{0}\left(x\right),\ldots,v_{l+1}\left(x\right)\in\mathcal{A}_{m_{\tilde{\rho}},n_{\tilde{\rho}},r_{\tilde{\rho}}}$
such that
\[
\hat{a}_{i}\circ\rho\circ\tilde{\rho}\left(x\right)=x^{\gamma_{i}}v_{i}\left(x\right)\text{\ for\ }i=0,\ldots,l\text{,\ and\ }g_{\rho}\circ\tilde{\rho}\left(x\right)=x^{\gamma_{l+1}}v_{l+1}\left(x\right).
\]
After a suitable sequence $\sigma$ of sign-changing transformations
as in \ref{Proviso: admissible transf}, we may assume that $x\in\left(\mathbb{R}^{\geq0}\right)^{N}$,
hence,
\[
\varphi_{i}\circ\hat{\rho}\circ\tilde{\rho}\circ\sigma\left(x\right)=\begin{cases}
x^{\gamma_{l+1}\beta_{i}+\gamma_{i}}U_{i}\left(x\right) & \text{if\ }a_{0}\leq1\text{\ and\ }a_{i}\leq1\text{\ on }\hat{D}\\
x^{\gamma_{l+1}\beta_{i}-\gamma_{i}}U_{i}\left(x\right) & \text{if\ }a_{0}\leq1\text{\ and\ }a_{i}>1\text{\ on }\hat{D}\\
x^{-\gamma_{0}+\gamma_{l+1}\beta_{i}+\gamma_{i}}U_{i}\left(x\right) & \text{if\ }a_{0}>1\text{\ and\ }a_{i}\leq1\text{\ on }\hat{D}\\
x^{-\gamma_{0}+\gamma_{l+1}\beta_{i}-\gamma_{i}}U_{i}\left(x\right) & \text{if\ }a_{0}>1\text{\ and\ }a_{i}>1\text{\ on }\hat{D}
\end{cases},
\]
for some $U_{i}$ as in the statement of the lemma. Notice that all
the multi-exponents in the expression above belong to $\left(\mathbb{K}^{\geq0}\right)^{N}$,
because $\Phi$ is bounded. Hence, after an appropriate series of
ramifications, the above multi-exponents belong to $\mathbb{A}^{N}$
and we can conclude by the first remark in \ref{rem: evolution of phi}. 
\end{proof}

\subsection{(A)$_{N-1}$ and (B)$_{N,l'}$ $\left(\forall l'\leq l-1\right)$
imply (B)$_{N,l}$}

Let $\hat{u}=\left(u_{1},\ldots,u_{j\left(\Phi\right)-1}\right)$
and $v=u_{j\left(\Phi\right)}$. Let $\Phi'=\left(\varphi_{1},\ldots,\varphi_{j\left(\Phi\right)-1}\right)$.

We first show the following reduction.
\begin{void}
\label{vuoto: reduce to F regular in v}We can assume that $F_{0}\left(x,\hat{u},v\right)$
is regular of some order $d>0$ in the variable $v$, after factoring
out a monomial in the variables $x,\hat{u}$. 
\end{void}
It is enough to show that the series $\mathcal{T}\left(F_{0}\right)\left(X,\hat{U},V\right)$
is regular in the variable $V$. Clearly, we cannot simply perform
a linear transformation in the variables $x_{\check{m}+1},\ldots,x_{N},\hat{u}$
as we did in Section \ref{sub:Monomialisation-of-generalised}, because
that would not be a vertical transformation. We will follow a different
approach, inspired by the methods in commutative algebra to prove
that the ring of formal power series is Noetherian. Clearly the ring
of generalised power series is not Noetherian. However, Lemma \ref{lem:noether}
below gives us a finiteness result which is enough for our purposes.
\begin{notation}
Let $m,n,m',n',n,l\in\mathbb{N}$, with $m+n=m'+n'=N$. Let $Z=\left(Z_{1},\ldots,Z_{m}\right),\ Y=\left(Y_{1},\ldots,Y_{n}\right),\ U=\left(U_{1},\ldots,U_{l}\right)$
and $Z'=\left(Z_{1}',\ldots,Z_{m'}'\right),\ Y'=\left(Y_{1}',\ldots,Y_{n'}'\right),\ U'=\left(U_{1}',\ldots,U_{l}'\right)$.
If $V\in\left\{ Y_{1},\ldots,Y_{n},U_{1},\ldots,U_{l}\right\} ,$
we denote by $\widehat{\left(Y,U\right)}$ the set $\left\{ Y_{1},\ldots,Y_{n},U_{1},\ldots,U_{l}\right\} \setminus\left\{ V\right\} $.\end{notation}
\begin{void}
[Formal Weierstrass division]\label{emp:weierstrass}Let $F,G\in\mathbb{R}\left\llbracket Z^{*},Y,U\right\rrbracket $
and suppose that $G$ is \emph{regular of order $d\in\mathbb{N}$
in the variable $V$}, i.e. $G\left(0,V\right)=V^{d}W\left(V\right)$,
where $W\left(V\right)\in\mathbb{R}\left\llbracket V\right\rrbracket $
and $W\left(0\right)\not=0$. Then, by \cite[4.17]{vdd:speiss:gen},
there exist $Q\in\mathbb{R}\left\llbracket Z^{*},Y,U\right\rrbracket $
and $R=\sum_{i=0}^{d-1}B_{i}\left(Z,\widehat{\left(Y,U\right)}\right)V^{i}\in\mathbb{R}\left\llbracket Z^{*},\widehat{\left(Y,U\right)}\right\rrbracket \left[V\right]$
such that $F=G\cdot Q+R$. A careful analysis of the proof of \cite[4.17]{vdd:speiss:gen}
shows that $\mathrm{Supp}_{Z}\left(Q\right)$ and $\mathrm{Supp}_{Z}\left(R\right)$
are contained in the good set $\Sigma\left(\mathrm{Supp}_{Z}\left(F\right)\cup\mathrm{Supp}_{Z}\left(G\right)\right)$.\end{void}
\begin{lem}
\label{lem:noether} Let $\mathcal{G}=\left\{ F_{i}=\left(F_{i,0},\ldots,F_{i,d-1}\right):\ i\in\mathbb{N}\right\} \subset\left(\mathbb{R}\left\llbracket Z^{*},Y,U\right\rrbracket \right)^{d}$
be a family with good total support. Then there exists an admissible
tree $T$ such that, for every branch $\mathfrak{b}$ of $T$, acting
as $T{}_{\mathfrak{b}}:\mathbb{R}\left\llbracket Z{}^{*},Y,U\right\rrbracket \to\mathbb{R}\left\llbracket Z'^{*},Y',U'\right\rrbracket $,
the $\mathbb{R}\left\llbracket Z'^{*},Y',U'\right\rrbracket $-module
generated by the set 
\[
\left\{ \left(T_{\mathfrak{b}}\left(F_{i,1}\right),\ldots,T_{\mathfrak{b}}\left(F_{i,d}\right)\right):\ i\in\mathbb{N}\right\} 
\]
 is finitely generated. Moreover, $T_{\mathfrak{b}}$ induces an admissible
transformation which is vertical (with respect to the variables $U$). 
\end{lem}
In particular, for some $p\in\mathbb{N}$, $\left\{ \left(T_{\mathfrak{b}}\left(F_{i,0}\right),\ldots,T_{\mathfrak{b}}\left(F_{i,d}\right)\right):\ i=0,\ldots,p\right\} $
is a set of generators.
\begin{proof}
The proof is by induction on the pairs $\left(N+l,d\right)$, ordered
lexicographically.

Let us first examine the case $d=1$. If $N+l=1,\ $ then there exist
$\alpha\in[0,\infty)$ and $i_{0}\in\mathbb{N}$ such that $\forall i\in\mathbb{N}\ F_{i}=Z'^{\alpha}G_{i}$
and $G_{i_{0}}\left(0\right)\not=0$. Therefore the ideal generated
by the set $\left\{ F_{i}:\ i\in\mathbb{N}\right\} $ is principal,
generated by $F_{i_{0}}$. 

Let us suppose $N+l>1$. By Lemma \ref{lem:singular blow ups}, there
exists an admissible tree $T$ (acting as the identity on the variables
$U$) such that, for every branch $\mathfrak{b}$ of $T$, acting
as $T{}_{\mathfrak{b}}:\mathbb{R}\left\llbracket Z{}^{*},Y,U\right\rrbracket \to\mathbb{R}\left\llbracket Z'^{*},Y',U'\right\rrbracket $,
there exist $\alpha\in[0,\infty)^{m}$ and $G_{i}\in\mathbb{R}\left\llbracket Z'^{*},Y',U'\right\rrbracket $
such that $T_{\mathfrak{b}}\left(F_{i}\right)=Z'^{\alpha}G_{i}$ and
for some $i_{0}\in\mathbb{N}$, $G_{i_{0}}\left(0,Y',U'\right)\not\equiv0$. 

If $Y'=\emptyset$, then there is a linear change of coordinates $L_{l,c}$
in the $U'$-variables such that $G_{i_{0}}$ is regular of order
$d'\in\mathbb{N}$ in the variable $U_{l}'$. 

If $Y'\not=\emptyset$, then there is a linear change of coordinates
$L_{n',c}$ in the $\left(Y',U'\right)$-variables such that $G_{i_{0}}$
is regular of order $d'\in\mathbb{N}$ in the variable $Y_{n'}'$.
Let $V$ be either $U_{l}'$ (in the first case) or $Y_{n'}'$ (in
the second case) and let $L_{c}$ be either $L_{l,c}$ (in the first
case) or $L_{n',c}$ (in the second case), and let us rename $Z'=Z$,
$Y'=Y$ and $U'=U$. By \ref{emp:weierstrass}, there are $Q_{i}\in\mathbb{R}\left\llbracket Z{}^{*},Y,U\right\rrbracket $
and $B_{i,0},\ldots,B_{i,d'-1}\in\mathbb{R}\left\llbracket Z{}^{*},\widehat{\left(Y,U\right)}\right\rrbracket $
such that 
\[
L_{c}\left(G_{i}\right)=L_{c}\left(G_{i_{0}}\right)\cdot Q_{i}+R_{i},
\]
 where $R_{i}=\sum_{j=0}^{d'-1}B_{i,j}\left(Z,\widehat{\left(Y,U\right)}\right)V{}^{j}.$ 

By Remark \ref{rem:tree preserves good families}, the family $\left\{ L_{c}\left(G_{i}\right):\ i\in\mathbb{N}\right\} $
has good total support, and hence, by \ref{emp:weierstrass}, so does
the family $\mathcal{B}=\left\{ B_{i}=\left(B_{i,0},\ldots,B_{i,d'-1}\right):\ i\in\mathbb{N}\right\} \subset\left(\mathbb{R}\left\llbracket Z{}^{*},\widehat{\left(Y,U\right)}\right\rrbracket \right)^{d'}$. 

By the inductive hypothesis, there is an admissible tree $T'$ (acting
as the identity on the variable $V$) such that, for every branch
$\mathfrak{b'}$ of $T'$, we have $T'_{\mathfrak{b'}}:\mathbb{R}\left\llbracket Z{}^{*},Y,U\right\rrbracket \to\mathbb{R}\left\llbracket Z'^{*},Y',U'\right\rrbracket $
, the $\mathbb{R}\left\llbracket Z'^{*},\widehat{\left(Y',U'\right)}\right\rrbracket $-module
generated by the set $T'_{\mathfrak{b'}}\left(\mathcal{B}\right)$
is finitely generated and $T_{\mathfrak{b'}}'$ acts vertically with
respect to the variables $U$. Let us again rename $Z'=Z$, $Y'=Y$
and $U'=U$. We may suppose that, for some $k\in\mathbb{N}$, the
generators are $T'_{\mathfrak{b'}}\left(B_{0}\right),\ldots,T'_{\mathfrak{b'}}\left(B_{k}\right)$.
Hence, $\forall i\in\mathbb{N}$, there exist series $C_{i,0},\ldots,C_{i,k}\in\mathbb{R}\left\llbracket Z{}^{*},\widehat{\left(Y,U\right)}\right\rrbracket $
such that $T'_{\mathfrak{b'}}\left(B_{i}\right)=\sum_{s=0}^{k}C_{i,s}T'_{\mathfrak{b'}}\left(B_{s}\right)$. 

Notice that 
\[
T'_{\mathfrak{b'}}\left(R_{i}\right)=\sum_{j=0}^{d'-1}T'_{\mathfrak{b'}}\left(B_{i,j}\right)V{}^{j}=\sum_{s=0}^{k}C_{i,s}\sum_{j=0}^{d'-1}T'_{\mathfrak{b'}}\left(B_{s}\right)V{}^{j}=\sum_{s=0}^{k}C_{i,s}T'_{\mathfrak{b'}}\left(R_{s}\right).
\]
Finally, 
\begin{align*}
T_{\mathfrak{b}}\circ L_{c}\circ T'_{\mathfrak{b'}}\left(F_{i}\right)=T'_{\mathfrak{b'}}\left(Z{}^{\alpha}\cdot L_{c}\left(G_{i_{0}}\right)\cdot Q_{i}\right)+\sum_{s=0}^{k}C_{i,s}T'_{\mathfrak{b'}}\left(Z{}^{\alpha}\cdot R_{s}\right)=\\
=\left[L_{c}\circ T'_{\mathfrak{b'}}\left(Z{}^{\alpha}\cdot G_{i_{0}}\right)\right]\cdot T'_{\mathfrak{b'}}\left(Q_{i}\right)+\sum_{s=0}^{k}C_{i,s}\cdot\left[L_{c}\circ T'_{\mathfrak{b'}}\left(Z{}^{\alpha}\cdot G_{i}\right)-L_{c}\circ T'_{\mathfrak{b'}}\left(Z{}^{\alpha}\cdot G_{i_{0}}\right)\cdot T'_{\mathfrak{b'}}\left(Q_{i}\right)\right].
\end{align*}
 The series within square brackets in the last line of above formula
are indeed elements of the ideal generated by the set $S=\left\{ T_{\mathfrak{b}}\circ L_{c}\circ T'_{\mathfrak{b'}}\left(F_{i}\right):\ i\in\mathbb{N}\right\} $.
In particular, we can choose a finite set of generators within the
set $S$ itself. This concludes the proof of the case $d=1$.

As for the general case, consider the family $\mathcal{G}'=\left\{ \left(F_{i,1},\ldots,F_{i,d-1}\right):\ i\in\mathbb{N}\right\} \subset\left(\mathbb{R}\left\llbracket Z^{*},Y,U\right\rrbracket \right)^{d-1}$.
Since the total support of $\mathcal{G}'$ is good, we can apply the
inductive hypothesis and find an admissible tree $T$ such that, for
every branch $\mathfrak{b}$ of $T$, the module generated by $T_{\mathfrak{b}}\left(\mathcal{G}'\right)$
is finitely generated. 

Let $\left\{ \left(T_{\mathfrak{b}}\left(F_{i,1}\right),\ldots,T_{\mathfrak{b}}\left(F_{i,d-1}\right)\right):\ i\leq p\right\} $
be a set of generators, for some $p\in\mathbb{N}$. Now consider the
family $\mathcal{G}''=\left\{ T_{\mathfrak{b}}\left(F_{i,0}\right):\ i\in\mathbb{N}\right\} \subset\mathbb{R}\left\llbracket Z'^{*},Y',U'\right\rrbracket $.
By Remark \ref{rem:tree preserves good families}, $\mathcal{G}''$
has good total support, hence there exists an admissible tree $T'$
such that, for every branch $\mathfrak{b'}$ of $T'$, the ideal generated
by $T'_{\mathfrak{b'}}\left(\mathcal{G}''\right)$ is finitely generated.
Let $\left\{ T_{\mathfrak{b}}\circ T'_{\mathfrak{b'}}\left(F_{i,0}\right):\ i\leq s\right\} $
be a set of generators, for some $s\in\mathbb{N}$. It is then clear
that the module generated by $T'_{\mathfrak{b'}}\circ T_{\mathfrak{b}}\left(\mathcal{G}\right)$
is generated by the set $\left\{ \left(T_{\mathfrak{b}}\circ T'_{\mathfrak{b'}}\left(F_{i,0}\right),0\right):\ i\leq s\right\} \cup\left\{ \left(0,T_{\mathfrak{b}}\circ T'_{\mathfrak{b'}}\left(F_{i,1}\right),\ldots,T_{\mathfrak{b}}\circ T'_{\mathfrak{b'}}\left(F_{i,d-1}\right)\right):\ i\leq p\right\} $.
\end{proof}
Going back to the proof of \ref{vuoto: reduce to F regular in v},
let $F_{0}\left(x,\hat{u},v\right)=\sum_{i\geq0}f_{i}\left(x,\hat{u}\right)v^{i}$,
where $f_{i}\left(x,\hat{u}\right)=\frac{1}{i!}\frac{\partial^{i}F_{0}}{\partial v^{i}}\left(x,\hat{u},0\right)$,
which, by a remark in \ref{rems: properties of the algebras}, belongs
to $\mathcal{A}_{\check{m},\check{n}+j\left(\Phi\right)-1}$. The
family $\mathcal{G}=\left\{ f_{i}:\ i\in\mathbb{N}\right\} $ has
good total support, hence, by Lemma \ref{lem:noether} and Remark
\ref{rem: elemntary transf send quadrants to sectors}, there is a
finite family $\mathcal{F}$ of vertical admissible transformations
$\rho:\hat{I}_{m_{\rho},n_{\rho}+j\left(\Phi\right)-1,r_{\rho}}\to\hat{I}_{m,n+j\left(\Phi\right)-1,r}$,
where $m=\text{max}\left\{ \hat{m},\check{m}\right\} $ (hence they
all respect $f$), such that $\mathcal{F}$ satisfies the covering
property with respect to $\hat{I}_{m,n+j\left(\Phi\right)-1,r}$ (and
hence with respect to $\Phi'$ ) and for every $\rho\in\mathcal{F}$,
either $j\left(\Phi_{\rho}'\right)<j\left(\Phi'\right)$, or $j\left(\Phi_{\rho}'\right)=j\left(\Phi'\right)=j\left(\Phi\right)-1\leq l-1$
and the ideal generated by the family $\left\{ f_{i}\circ\rho:\ i\in\mathbb{N}\right\} $
is generated by $f_{0}\circ\rho,\ldots,f_{p}\circ\rho$, for some
$p\in\mathbb{N}$, i.e. there are series $Q_{i,n}\left(X,\hat{U}\right)$,
not necessarily in $\text{Im}\left(\mathcal{T}\right)$, such that
for all $n\in\mathbb{N},\ f_{n}\circ\rho=\sum_{i=0}^{p}Q_{i,n}\cdot f_{i}\circ\rho$.
Hence we can write
\[
\mathcal{T}\left(F_{0}\circ\rho\right)=\sum_{i=0}^{p}\mathcal{T}\left(f_{i}\circ\rho\right)\left(X,\hat{U}\right)V^{i}W_{i}\left(X,\hat{U},V\right),
\]
where the series $W_{i}=1+\sum_{n>p}Q_{i,n}\left(X,\hat{U}\right)V^{n-i}$
are units, not necessarily in $\text{Im}\left(\mathcal{T}\right)$.
We can apply the inductive hypothesis \textbf{(B)}$_{N,j\left(\Phi\right)-1}$
to $\tilde{F}\left(x,\hat{u}\right):={\displaystyle \prod_{0\leq i,j\leq p,i\not=j}f_{i}\circ\rho\left(f_{i}\circ\rho-f_{j}\circ\rho\right)}$
and $\tilde{\Phi}:=\Phi_{\rho}':D_{\rho}\to\mathbb{R}^{j\left(\Phi'\right)}$
and obtain that there exist a sub-$\Lambda$-set $\tilde{S}\subseteq D_{\rho}$
of dimension $\leq N-1$ and a finite family $\tilde{\mathcal{F}}$
of vertical admissible transformations $\tilde{\rho}:\hat{I}_{m_{\tilde{\rho}},n_{\tilde{\rho}}+j\left(\Phi'\right),r_{\tilde{\rho}}}\to\hat{I}_{m_{\rho},n_{\rho}+j\left(\Phi'\right),r_{\rho}}$
such that $\tilde{\mathcal{F}}$ satisfies the covering property with
respect to $\left(\tilde{\Phi},\tilde{S}\right)$ and for every $\tilde{\rho}\in\tilde{\mathcal{F}}$,
$\tilde{\rho}$ respects $f\circ\rho$ and either $j\left(\tilde{\Phi}_{\tilde{\rho}}\right)<j\left(\tilde{\Phi}\right)$
or $f_{0}\circ\rho\circ\tilde{\rho},\ldots,f_{p}\circ\rho\circ\tilde{\rho}$
are all normal and (by Lemma \ref{lem: prod of series}) linearly
ordered by division. In this latter case, after factoring out a monomial,
we obtain that $F_{0}\circ\rho\circ\tilde{\rho}$ is regular in $v$.
By Remarks \ref{rem: evolution of phi}, we have proved \ref{vuoto: reduce to F regular in v}.

\medskip{}
The next step is to show the following reduction.
\begin{void}
\label{vuoto: chute d'ordre}We may assume that $F_{0}$ is regular
of order $1$ in the variable $v$.

Since we may assume that $F_{0}$ is regular of some order $d>1$
in the variable $v$, after a Tschirnhausen translation, we can write
\[
F_{0}\left(x,\hat{u},v\right)=W\left(x,\hat{u},v\right)v^{d}+a_{2}\left(x,\hat{u}\right)v^{d-2}+\ldots+a_{d}\left(x,\hat{u}\right),
\]
where $a_{i}\in\mathcal{A}_{\check{m},\check{n}+j\left(\Phi\right)-1}$
and $W\in\mathcal{A}_{\check{m},\check{n}+j\left(\Phi\right)}$ is
a unit. By Lemma \ref{lem: weak monomialisation}, we may assume that
there are constants $K_{1},K_{2}>0$, a function $U:D\to\mathbb{R}$
(whose graph is a sub-$\Lambda$-set) and a multi-exponent $\alpha\in\mathbb{A}^{N}$
such that $\varphi_{j\left(\Phi\right)}\left(x\right)=x^{\alpha}U\left(x\right)$
and $K_{1}<|U|<K_{2}$. Without loss of generality we may assume $U$
is strictly positive on $D$.

Let $a_{0}\left(x,\hat{u}\right)=1,\ a_{1}\left(x,\hat{u}\right)=0$
and consider the family $\left\{ x^{\alpha\left(d-i\right)}a_{i}\left(x,\hat{u}\right)\right\} _{i=0,\ldots,d}$.
We apply \textbf{(B)}$_{N,j\left(\Phi\right)-1}$ simultaneously to
the members of this family, i.e. to the function $A\left(x,\hat{u}\right)=\prod_{0\leq i,j\leq d,i\not=j}x^{\alpha\left(d-i\right)}a_{i}\left(x,\hat{u}\right)\left(x^{\alpha\left(d-i\right)}a_{i}\left(x,\hat{u}\right)-x^{\alpha\left(d-j\right)}a_{j}\left(x,\hat{u}\right)\right)$,
and to $\Phi'$. Hence there are a sub-$\Lambda$-set $S\subseteq D$
of dimension $<N$ and a finite family of vertical admissible transformations
which extends trivially to a finite family $\mathcal{F}$ of vertical
admissible transformations $\rho:\hat{I}_{m_{\rho},n_{\rho}+j\left(\Phi\right),r_{\rho}}\to\hat{I}_{m,n+j\left(\Phi\right),r}$,
where $m=\text{max}\left\{ \check{m},\hat{m}\right\} $ and $n=N-m$
(hence $\rho'$ respects $f$), such that $\mathcal{F}$ satisfies
the covering property with respect to $\left(\hat{\Phi},S\right)$
and for every $\rho\in\mathcal{F}$, either $j\left(\hat{\Phi}_{\rho}\right)<j\left(\hat{\Phi}\right)$
or, for all $i=2,\ldots,d$, there exist $\gamma_{i}\in\mathbb{A}^{N},\ \delta_{i}\in\mathbb{N}^{j\left(\Phi\right)-1}$
and units $W_{i}\in\mathcal{A}_{m_{\rho},n_{\rho}+j\left(\Phi\right)-1}$
such that $a_{i}\circ\rho\left(x,\hat{u}\right)=x^{\gamma_{i}}\hat{u}^{\delta_{i}}W_{i}\left(x,\hat{u}\right)$.
Moreover, by Lemma \ref{lem: prod of series}, the family $\left\{ \left(x^{\alpha\left(d-i\right)}a_{i}\left(x,\hat{u}\right)\right)\circ\rho\right\} _{i=0,\ldots,d}$
is linearly ordered by division, hence the set of $\left(N+j\left(\varphi\right)\right)$-tuples
$\mathcal{E}=\left\{ \left(\beta\left(d-i\right)+\gamma_{i},\delta_{i}\right)\right\} _{i=0,\ldots,d}$
is totally ordered. Notice that $\varphi_{j\left(\varphi\right)}\circ\rho'\left(x\right)$
is still weakly normal, i.e. there exists $\beta\in\mathbb{A}^{N}$
such that $\varphi_{j\left(\Phi\right)}\circ\rho'\left(x\right)=x^{\beta}U\circ\rho'\left(x\right)$,
and, after suitable ramifications of the variables $x_{1},\ldots,x_{\check{m}}$,
we may suppose that $\beta\in\mathbb{N}^{N}$. 

Let $N_{0}=\text{max}\left\{ M\leq N:\ \beta_{M}\not=0\right\} $.
Define, for $j=1,\ldots,N_{0}-1$, $\rho_{\lambda,j}\left(x,\hat{u},v\right)=x_{j}^{\beta_{j}}v$
and $\rho_{\lambda,N_{0}}\left(x,\hat{u},v\right)=x_{N_{0}}^{\beta_{N_{0}}}\left(\lambda+v\right)$.
Notice that the function $\left(x',\hat{u}',v'\right)\mapsto v=\rho_{\lambda,1}\circ\ldots\circ\rho_{\lambda,N_{0}-1}\circ\rho_{\lambda,N_{0}}\left(x',\hat{u}',v'\right)$
is a finite composition of blow-up charts and extends trivially to
a vertical admissible transformation $\rho_{\lambda}:\hat{I}_{m_{\rho},n_{\rho}+j\left(\Phi\right),r_{\lambda}}\to\hat{I}_{m_{\rho},n_{\rho}+j\left(\Phi\right),r_{\rho}}$
.

Let
\[
\xyC{0mm}\xyL{0mm}\xymatrix{\varepsilon\colon & \mathbb{R}^{>0}\ar[rrrr] & \  & \  & \  & \mathbb{R}^{>0}\\
 & \lambda\ar@{|->}[rrrr] &  &  &  & \varepsilon_{\lambda}
}
\]
be any fixed function and let $\mathcal{G}$ be a finite family of
positive real numbers such that $\left[K_{1},K_{2}\right]\subseteq\bigcup_{\lambda\in\mathcal{G}}\left(\lambda-\varepsilon_{\lambda},\lambda+\varepsilon_{\lambda}\right)$.

Let $\mathcal{F}_{\mathcal{G}}=\left\{ \rho_{\lambda}:\ \lambda\in\mathcal{G}\right\} $.
Since for all $x\in D$ there exists $\lambda\in\mathcal{G}$ such
that $x^{\beta}\left(\lambda-\varepsilon_{\lambda}\right)<\varphi_{j\left(\Phi\right)}\left(x\right)<x^{\beta}\left(\lambda+\varepsilon_{\lambda}\right)$,
the family $\mathcal{F}_{\mathcal{G}}$ satisfies the covering property
with respect to $\left(\hat{\Phi}_{\rho},S_{\rho}\right)$. We show
that, for every $\rho_{\lambda}\in\mathcal{F}_{\mathcal{G}}$, either
$j\left(\left(\hat{\Phi}_{\rho}\right)_{\rho_{\lambda}}\right)<j\left(\hat{\Phi}_{\rho}\right)$,
or, possibly after factoring out a monomial in the variables $x$,
$F_{0}\circ\rho\circ\rho_{\lambda}$ is regular of order $d'<d$ in
the variable $v$. Let $\tilde{F}=F_{0}\circ\rho\circ\rho_{\lambda}$
and $\tilde{W}\left(x,\hat{u},v\right)=W\circ\rho\circ\rho_{\lambda}\left(x,\hat{u},v\right)$,
which is still a unit. Then,
\[
\tilde{F}\left(x,\hat{u},v\right)=\tilde{W}\left(x,\hat{u},v\right)\left(\lambda+v\right)^{d}x^{\beta d}+\left(\lambda+v\right)^{d-2}x^{\beta\left(d-2\right)+\gamma_{2}}\hat{u}^{\delta_{2}}W_{2}\left(x,\hat{u}\right)+\ldots+x^{\gamma_{d}}\hat{u}^{\delta_{d}}W_{d}\left(x,\hat{u}\right).
\]
Let $i_{0}=\text{min}\left\{ i:\ 0\leq i\leq d,\ \left(\beta\left(d-i\right)+\gamma_{i},\delta_{i}\right)=\text{min}\left(\mathcal{E}\right)\right\} $.
Notice that necessarily $\delta_{i_{0}}=0$. After factoring $\tilde{F}$
out by the monomial $x^{\beta\left(d-i_{0}\right)+\gamma_{i_{0}}}$,
we obtain that, either $i_{0}=0$ and, thanks to the Tschirnhausen
translation we did at the beginning, the coefficient of the term $v^{d-1}$
is a unit, or $i_{0}>0$ and the coefficient of the term $v^{d-i_{0}}$
is a unit. In either of the two cases, after factorisation $\tilde{F}$
has become regular of order $d'<d$ in $v$. Hence we can start again
with the procedure just described until we reduce to $d'=1$. By Remarks
\ref{rem: evolution of phi}, this concludes the proof of \ref{vuoto: chute d'ordre}.

\bigskip{}

Since $F_{0}$ is regular of order $1$ in the variable $v$, after
performing a last Tschirnhausen translation $\tau_{h}$ as in the
proof of Theorem \ref{thm: monomialisation}, we have that, either
$j\left(\hat{\Phi}_{\tau_{h}}\right)<j\left(\hat{\Phi}\right)$, or
$F_{0}\circ\tau_{h}\left(x,\hat{u},v\right)=x^{\alpha}\hat{u}^{\beta}v^{\gamma}W\left(x,\hat{u},v\right)$
for some $\alpha\in\mathbb{A}^{N},\left(\beta,\gamma\right)\in\mathbb{N}^{j\left(\Phi\right)}$
and a unit $W\in\mathcal{A}_{m,n+j\left(\Phi\right)}$ (where $m=\text{max}\left\{ \check{m},\hat{m}\right\} $).
Let $r$ be a polyradius in $\mathbb{R}^{N+j\left(\Phi\right)}$ such
that $W$ has a representative which does not vanish on $\hat{I}_{m,n+j\left(\Phi\right),r}$.
Notice that the size of the last coordinate of $r$ determines the
choice of $\varepsilon_{\lambda}$ in the proof of \ref{vuoto: chute d'ordre}.

Hence we can conclude the proof of \textbf{(B)}$_{N,l}$ by Remarks
\ref{rem: evolution of phi}.
\end{void}

\subsection{(B)$_{N,N}$ and (C)$_{N,l'}$ $\left(\forall l'\leq l-1\right)$
imply (C)$_{N,l}$}

We apply \textbf{(B)}$_{N,N}$ to $F\left(x,u\right):=\prod_{i=1}^{N}f_{i}\left(x,u\right)$
and $\Phi$ and obtain that there exist a finite family $\mathcal{F}$
of vertical admissible transformations $\rho:\hat{I}_{m_{\rho},n_{\rho}+l,r_{\rho}}\to\hat{I}_{m,n+l,r}$,
where $m=\text{max}\left\{ \hat{m},\check{m}\right\} $ (hence $\rho'$
respects $f$), and$ $ a sub-$\Lambda$-set $S\subseteq D$ of dimension
strictly smaller than $N$ such that $\mathcal{F}$ satisfies the
covering property with respect to $\left(\hat{\Phi},S\right)$ and
for every $\rho\in\mathcal{F}$, either $j\left(\hat{\Phi}_{\rho}\right)<j\left(\Phi\right)$
or $F_{0}\circ\rho$ is normal. Let $\hat{\mathcal{F}}$ be the family
obtained by extending trivially each $\rho\in\mathcal{F}$ to $\hat{\rho}:\hat{I}_{m_{\rho},n_{\rho}+N,\widehat{r_{\rho}}}\to\hat{I}_{m,n+N,\hat{r}}$.
By Remarks \ref{rem: evolution of phi}, $\hat{\mathcal{F}}$ satisfies
the covering property with respect to $\left(\Phi,S\right)$.

Suppose first that $j\left(\hat{\Phi}_{\rho}\right)=l'<l$. Let $\tilde{f_{i}}\left(x,u\right)=f_{i}\circ\hat{\rho}\left(x,u\right),\ \tilde{D}=D_{\rho}$
and $\tilde{\Phi}=\left(\hat{\Phi}_{\rho},0\right):\tilde{D}\to\mathbb{R}^{N}$.
Then by \textbf{(C)}$_{N,l'}$ there exist a finite family $\tilde{\mathcal{F}}$
of vertical admissible transformations $\tilde{\rho}:\hat{I}_{m_{\tilde{\rho}},n_{\tilde{\rho}}+N,r_{\tilde{\rho}}}\to\hat{I}_{m_{\rho},n_{\rho}+N,\widehat{r_{\rho}}}$
(respecting $f\circ\rho'$) and $ $ a sub-$\Lambda$-set $\tilde{S}\subseteq\tilde{D}$
of dimension strictly smaller than $N$ such that $\tilde{\mathcal{F}}$
satisfies the covering property with respect to $\left(\tilde{\Phi},\tilde{S}\right)$
and for every $\tilde{\rho}\in\tilde{\mathcal{F}}$,
\[
\forall\left(x,u\right)\in\tilde{D}_{\tilde{\rho}}\times\mathbb{R}^{N}\ \ \ \ \ \begin{cases}
\tilde{f_{1}}\circ\tilde{\rho}\left(x,u\right)=0\\
\vdots\\
\tilde{f}_{N}\circ\tilde{\rho}\left(x,u\right)=0
\end{cases}\Leftrightarrow\ \ u=0.
\]
Hence, by Remarks \ref{rem: evolution of phi}, we are done in this
case.

Now suppose that $F_{0}\circ\rho$ is normal. Let $\hat{u}=\left(u_{1},\ldots u_{l}\right)$.
Then there are $\left(\alpha_{i},\beta_{i}\right)\in\mathbb{A}^{N}\times\mathbb{N}^{l}$
and units $U_{i}\in\mathcal{A}_{m_{\rho},n_{\rho}+l}$ such that $\tilde{f_{i}}\left(x,\hat{u}\right):=\left(f_{i}\right)_{0}\circ\rho\left(x,\hat{u}\right)=x^{\alpha_{i}}u^{\beta_{i}}U_{i}\left(x,\hat{u}\right)$.
Let $\tilde{D}=D_{\rho}$ and $\tilde{\Phi}=\hat{\Phi}_{\rho}$. Notice
that
\[
\forall\left(x,\hat{u}\right)\in\tilde{D}\times\mathbb{R}^{l}\ \ \ \ \ \begin{cases}
\tilde{f}_{1}\left(x,\hat{u}\right)=0\\
\vdots\\
\tilde{f}_{N}\left(x,\hat{u}\right)=0
\end{cases}\Leftrightarrow\ \ \hat{u}=\tilde{\Phi}\left(x\right).\tag{*}
\]

Let $\tilde{S}$ be the sub-$\Lambda$-set (of dimension strictly
smaller than $N$) $\left\{ x_{1}=0\right\} \cup\ldots\cup\left\{ x_{N}=0\right\} $.
Clearly we may assume that $\rho\left(\tilde{S}\right)\subseteq S$.
Then we have
\[
\forall\left(x,\hat{u}\right)\in\tilde{D}\setminus\tilde{S}\times\mathbb{R}^{l}\ \ \ \ \ \begin{cases}
\tilde{f}_{1}\left(x,\hat{u}\right)=0\\
\vdots\\
\tilde{f}_{N}\left(x,\hat{u}\right)=0
\end{cases}\Leftrightarrow\ \ \begin{cases}
\hat{u}^{\beta_{1}}=0\\
\vdots\\
\hat{u}^{\beta_{N}}=0
\end{cases}.\tag{**}
\]
We claim that $\forall x\in\tilde{D}\setminus\tilde{S}\ \ \tilde{\Phi}\left(x\right)=0$.
In fact if this is not the case then, without loss of generality,
for some $x\in\tilde{D}\setminus\tilde{S}$ we have $\tilde{\varphi}_{1}\left(x\right)\not=0$.
This means that for every $a\in\mathbb{R}$, the tuple $\left(x,a,\tilde{\varphi}_{2}\left(x\right),\ldots,\tilde{\varphi}_{l}\left(x\right)\right)$
satisfies the system of equations on the right side of ({*}{*}). But
this contradicts the equivalence in ({*}).

To conclude, notice that 
\[
\forall\left(x,u\right)\in\tilde{D}\setminus\tilde{S}\times\mathbb{R}^{N}\ \ \ \ \ \begin{cases}
f_{1}\circ\hat{\rho}\left(x,u\right)=0\\
\vdots\\
f_{N}\circ\hat{\rho}\left(x,u\right)=0
\end{cases}\Leftrightarrow\ \ \hat{u}=\tilde{\Phi}\left(x\right),\ u_{l+1}=\ldots=u_{N}=0.
\]

\subsection{(A)$_{N-1}$ and (C)$_{N,N}$ imply (A)$_{N}$}

Notice that, if $N=1$, then by \ref{vuoto: polybdd}, for all $x\in D\cap\mathbb{R}^{\geq0}$
we have $\eta\left(x\right)=x^{\alpha/\beta}U\left(x\right)$, for
some unit $U\in\mathcal{A}_{1}$ and $\alpha,\beta\in\mathbb{A}$.
If $\hat{m}=1$ then take $\mathcal{F}=\left\{ \sigma_{1}^{+}\circ r_{1}^{\beta}\right\} $,
if $\hat{n}=1$ then take $\mathcal{F}=\left\{ \sigma_{1}^{+}\circ r_{1}^{\beta},\sigma_{1}^{-}\circ r_{1}^{\beta}\right\} $.
In both cases $\mathcal{F}$ satisfies the conclusion of \textbf{(A)}$_{1}$.

Arguing as in the first paragraph of the proof of Corollary \ref{cor: def functions are terms}
and by the inductive hypothesis \textbf{(A)}$_{N-1}$, we may assume
that $\text{dim}\left(D\right)=N$. By the same argument, it is enough
to prove the statement for $\eta\restriction D\setminus S$, where
$S$ is any sub-$\Lambda$-set of dimension $<N$. 

Since $\Gamma\left(\eta\right)$ is a sub-$\Lambda$-set, we can apply
the Parametrisation Theorem \ref{thm: param subanal} and obtain that
$\Gamma\left(\eta\right)$ is the finite union of the diffeomorphic
images of sub-quadrants under maps whose components are in $\mathcal{A}$.
Without loss of generality, we may assume that $\Gamma\left(\eta\right)=H\left(Q\right)$,
where $Q\subseteq\left(\mathbb{R}^{>0}\right)^{N}$ is a sub-quadrant
of dimension $N$ and $H=\left(g_{1},\ldots,g_{N},h\right)\in\left(\mathcal{A}_{N,0}\right)^{N+1}$.
Moreover,
\[
\forall\left(x,z\right)\in\mathbb{R}^{N+1}\ \ \ \ \ x\in D\ \text{and}\ z=\eta\left(x\right)\ \Leftrightarrow\ \exists!u\in Q\ \text{s.t.}\ \begin{cases}
x_{1}=g_{i}\left(u\right)\\
\vdots\\
x_{N}=g_{N}\left(u\right)\\
z=h\left(u\right)
\end{cases}.
\]
In particular, there is a map $\Phi:D\to Q$, whose graph is a sub-$\Lambda$-set,
such that
\[
\forall\left(x,u\right)\in D\times Q\ \ \ \ \ \begin{cases}
f_{1}\left(x,u\right)=0\\
\vdots\\
f_{N}\left(x,u\right)=0
\end{cases}\Leftrightarrow\ \ u=\Phi\left(x\right),
\]
where $f_{i}\left(x,u\right)=x_{i}-g_{i}\left(u\right)$. Note that
the $f_{i}$ might not satisfy the hypotheses of \textbf{(C)}$_{N,N}$,
because the variables $u$ may appear with real exponents in $\mathcal{T}\left(f_{i}\right)$.
Our next task is to reduce to the case when $f_{i}\in\mathcal{A}_{\check{m},\check{n}+N}$.
In order to do this, we first apply Lemma \ref{lem: weak monomialisation}
in order to reduce to the case when all the components of $\Phi$
are weakly normal (notice that $j\left(\Phi\right)=N)$, i.e. $\varphi_{i}\left(x\right)=x^{\alpha_{i}}U_{i}\left(x\right)$,
where $\alpha_{i}\in\mathbb{A}$ and $\mathbb{R}_{\mathcal{A}}$-definable
functions $U_{i}$ which are bounded away from zero. Secondly, we
argue as in the proof of \ref{vuoto: chute d'ordre} and produce a
finite family $\mathcal{F}_{\mathcal{G}}=\left\{ \rho_{\lambda}:\ \in\mathcal{G}\right\} $
of vertical admissible transformations
\[
\rho_{\lambda}:\hat{I}_{m,n+N,r_{\lambda}}\to\hat{I}_{m+N,n,r}
\]
of type $\left(1,2\right)$, obtained by performing a series of ramifications
and blow-up charts, such that $\rho_{\lambda}\left(x,u\right)=\left\langle x,x^{\alpha_{1}}\left(\lambda+u_{1}\right),\ldots,x^{\alpha_{N}}\left(\lambda+u_{N}\right)\right\rangle $.
Notice that $f_{i}\circ\rho_{\lambda}\in\mathcal{A}_{m,n+N}$.

Now we can apply \textbf{(C)}$_{N,N}$ to the $f_{i}\circ\rho_{\lambda}$
and $\Phi_{\rho_{\lambda}}$ and obtain that there exist a finite
family $\tilde{\mathcal{F}}$ of vertical admissible transformations
\[
\tilde{\rho}:\hat{I}_{m_{\tilde{\rho}},n_{\tilde{\rho}}+N,r_{\tilde{\rho}}}\to\hat{I}_{m,n+N,r}
\]
 and $ $ a sub-$\Lambda$-set $\tilde{S}\subseteq D$ of dimension
strictly smaller than $N$ such that $\mathcal{F}$ satisfies the
covering property with respect to $\left(\Phi,S\right)$ and for every
$\tilde{\rho}\in\tilde{\mathcal{F}}$,
\[
\forall\left(x,u\right)\in D_{\rho}\setminus\left(\tilde{\rho}'\right)^{-1}\left(\tilde{S}\right)\times\mathbb{R}^{N}\ \ \ \ \ \begin{cases}
f_{1}\circ\rho_{\lambda}\circ\tilde{\rho}\left(x,u\right)=0\\
\vdots\\
f_{N}\circ\rho_{\lambda}\circ\tilde{\rho}\left(x,u\right)=0
\end{cases}\Leftrightarrow\ \ u=0.
\]
Summing up, there are a finite family $\mathcal{F}$ of vertical admissible
transformations 
\[
\rho:\hat{I}_{m_{\rho},n_{\rho}+N,r_{\rho}}\to\hat{I}_{m+N,n,r}
\]
of type $\left(1,2\right)$ such that $\rho'$ respects $f$ and $\rho''$
respects $h$, and a sub-$\Lambda$-set $S\subseteq D$ of dimension
strictly smaller than $N$ such that, by Remarks \ref{rem: evolution of phi},
the family $\mathcal{F}'=\left\{ \rho':\ \rho\in\mathcal{F}\right\} $
satisfies the covering property with respect to $D\setminus S$ and
for every $\rho\in\mathcal{F}$,
\[
\eta\circ\rho'\left(x\right)=h\circ\rho''\left(x,0\right)\in\mathcal{A}_{m_{\rho},n_{\rho}}.
\]
We conclude the proof of \textbf{(A)}$_{N}$ by Theorem \ref{thm: geom monomial}
and Remark \ref{rem: monomialise respecting f}. As in the proof of
\ref{vuoto: chute d'ordre}, the size of the domain on which $\eta$
has become normal determines the choice of $\varepsilon_{\lambda}$.

\bibliographystyle{amsalpha}
\def\cprime{$'$}
\providecommand{\bysame}{\leavevmode\hbox to3em{\hrulefill}\thinspace}
\providecommand{\MR}{\relax\ifhmode\unskip\space\fi MR }
\providecommand{\MRhref}[2]{%
  \href{http://www.ams.org/mathscinet-getitem?mr=#1}{#2}
}
\providecommand{\href}[2]{#2}

\bigskip{}

\emph{Jean-Philippe Rolin}

\emph{Institut de Mathématiques de Bourgogne - UMR 5584}

\emph{Université de Bourgogne }

\emph{BP 138 }

\emph{21004 Dijon cedex (France)}

\emph{email}: \href{mailto:rolin@u-bourgogne.fr}{rolin@u-bourgogne.fr}

\medskip{}

\emph{Tamara Servi}

\emph{Centro de Matemática e Applicaçoes Fundamentais}

\emph{Av. Prof. Gama Pinto, 2 }

\emph{1649-003 Lisboa (Portugal)}

\emph{email}: \href{mailto:tamara.servi@gmail.com}{tamara.servi@gmail.com}\emph{ }
\end{document}